\documentclass[11pt]{article}

\usepackage[utf8]{inputenc}

\usepackage{amsmath}
\usepackage{amsthm}
\usepackage{mathtools}
\usepackage{amsfonts}
\usepackage{url}
\usepackage{microtype}

\usepackage{mathpazo}

\usepackage[top=1.1in, bottom=1.1in, left=1.2in, right=1.2in]{geometry}

\usepackage[font=small]{caption}
\usepackage[font=footnotesize]{subcaption}

\usepackage{graphicx}
\usepackage{afterpage}
\usepackage{setspace}
\usepackage{xcolor}
\usepackage{color}
\usepackage{floatrow}

\usepackage[scaled=.90]{helvet}
\usepackage{titlesec}
\titleformat{\section}{\normalfont\sffamily\Large\bfseries}
  {\thesection}{1em}{}
\titleformat{\subsection}{\normalfont\sffamily\large\bfseries}
  {\thesubsection}{1em}{}

\titleformat{\subsubsection}{\normalfont\sffamily\bfseries}
  {\thesubsubsection}{1em}{}

\usepackage[numbers,sort&compress]{natbib}

\usepackage{isomath}
\renewcommand{\phi}{\varphi}
\numberwithin{equation}{section}

\usepackage{bm}

\usepackage{amsmath}
\usepackage{url}
\usepackage{bm}

\usepackage{xifthen}
\usepackage{ifdraft}
\usepackage{paralist}                 
\usepackage[modulo]{lineno}

\makeatletter
\newcommand{\hypertargetraised}[1]{\Hy@raisedlink{\hypertarget{#1}{}}}
\makeatother

\usepackage[capitalize]{cleveref}
\crefname{equation}{Eq.}{Eqs.}
\Crefname{equation}{Equation}{Equations}
\crefrangelabelformat{equation}{(#3#1#4--#5#2#6)}
\crefmultiformat{equation}{Eqs. (#2#1#3}{, #2#1#3)}{#2#1#3}{#2#1#3}
\Crefmultiformat{equation}{Equations (#2#1#3}{, #2#1#3)}{#2#1#3}{#2#1#3}
\newcommand\pr[1]{\cref{#1}}
\newcommand\prange[2]{\crefrange{#1}{#2}}

\newif\ifUseTikz
\UseTikztrue 
\ifUseTikz
\usepackage{tikz,pgfplots,import,tkz-euclide}
\pgfplotsset{compat=newest}

\usetikzlibrary{arrows.meta}
\usetikzlibrary{backgrounds}
\usetikzlibrary{external}
\usetikzlibrary{pgfplots.groupplots}
\usetikzlibrary{plotmarks}

\pgfplotsset{plot coordinates/math parser=false}
\newlength\figureheight
\newlength\figurewidth

\pgfkeys{/pgf/images/include external/.code={\includegraphics[]{#1}}}

\makeatletter
\tikzset{
    every picture/.style={
        execute at begin picture={
            \let\ref\@refstar
        }
    }
}
\makeatother

\definecolor{plt-blue}{rgb}{0.0078,0.2980,0.7961}
\definecolor{plt-orange}{rgb}{1.0000,0.6431,0.2627}
\definecolor{plt-purple}{rgb}{1.0000,0.2863,0.5255}
\definecolor{plt-violet}{rgb}{0.6118,0.1765,1.0000}
\else 
\fi

\usepackage{tabularx}     
\usepackage{booktabs}     
\usepackage{multirow}
\usepackage{collcell}     
\usepackage{dcolumn}      
\usepackage{ctable}       

\usepackage{amsthm}
\newtheorem{lem}{Lemma}

\newcommand\abbrev[1]{\textsc{\small\uppercase{#1}}}

\usepackage{cancel}
\usepackage{xparse} 
\usepackage[normal]{engord} 
\newcommand\ordinal[1]{\ifthenelse{\isin{#1}{abcdefghijklmnopqrstuvwxyz}}{\ensuremath{#1^\mathrm{th}}}{\engordnumber{#1}}}

\newcommand\norm[2][{}]{\lVert #2 \rVert_{#1}}

\newcommand\cross[2]{#1\times #2}
\newcommand\dotprod[2]{#1\cdot #2}
\newcommand\bigO[1]{\ensuremath{\mathcal{O}\!\left(#1\right)}}

\newcommand{\D}[1]{\ensuremath{\,\mathrm{d}#1}}

\newcommand\pderiv[2]{{\partial_{#2} #1}}

\newcommand{\Intersection}{{\ensuremath{\cap}}}            
\newcommand{\Real}{{\ensuremath{\mathbb{R}}}}              
\newcommand{\Imag}{{\ensuremath{i}}}

\newcommand{\FFT}[1]{\ensuremath{\mathcal{F}{#1}}}
\newcommand{\IFFT}[1]{\ensuremath{\mathcal{F}^{-1}{#1}}}
\newcommand{\Transpose}[1]{\ensuremath{#1^{T}}}

\let\vector\undefined
\DeclareMathAlphabet{\mathbfsf}{\encodingdefault}{\sfdefault}{bx}{n}
\newcommand\discrete[1]{\ensuremath{\mathsf{#1}}}
\newcommand\fourier[1]{\ensuremath{\widehat{#1}}}
\newcommand\scalar[1]{\ensuremath{#1}}
\newcommand\scalard[1]{\ensuremath{\discrete{#1}}}
\newcommand\vector[1]{\ensuremath{\bm{#1}}}
\newcommand\vct[1]{\vector{#1}}
\newcommand\vectord[1]{\ensuremath{\mathbfsf{#1}}}
\newcommand\linop[1]{\ensuremath{\mathbf{#1}}}
\newcommand\linopd[1]{\ensuremath{\mathbfsf{#1}}}
\newcommand\conv[1]{\ensuremath{\mathcal{#1}}}
\newcommand\restr[2]{{\left.\kern-\nulldelimiterspace #1 \vphantom{\big|} \right|_{#2} }}

\usepackage{marginnote,cprotect}
\usepackage{soul}
\usepackage{changebar}
\makeatletter
\@ifpackageloaded{pdfcomment}{

}{
}
\makeatother

\usepackage{expl3}
\ExplSyntaxOn
\newcommand\latinabbrev[1]{
  \peek_meaning:NTF . {
    #1\@}%
  { \peek_catcode:NTF a {
      #1.\@ }%
    {#1.\@}}}
\ExplSyntaxOff

\def\ie{\latinabbrev{i.e}}

\newcommand{\Grad}[1]{\ensuremath{\nabla #1}}
\newcommand{\Div}[1]{\ensuremath{\dotprod{\nabla}{#1}}}
\newcommand{\Curl}[1]{\ensuremath{\cross{\nabla}{#1}}}
\newcommand{\Lap}[1]{\ensuremath{\Delta #1}}
\newcommand{\SurfGrad}[1]{\ensuremath{\nabla_{\Gamma} #1}}
\newcommand{\SurfDiv}[1]{\ensuremath{\dotprod{\nabla_{\Gamma}}{#1}}}

\newcommand{\SurfLap}[1]{\ensuremath{\Delta_{\Gamma} #1}}
\newcommand{\InvSurfLap}[1]{\ensuremath{\Delta_{\Gamma}^{-1} #1}}

\newcommand{\dSurfGrad}[1]{\ensuremath{\nabla_{\Gamma_h} #1}}
\newcommand{\dSurfDiv}[1]{\ensuremath{\dotprod{\nabla_{\Gamma_h}}{#1}}}

\newcommand{\dSurfLap}[1]{\ensuremath{\Delta_{\Gamma_h} #1}}

\newcommand{\HelmKer}[1]{\ensuremath{g_{#1}}}
\newcommand{\SL}[1]{\ensuremath{\conv{S}_{#1}}}
\newcommand{\DL}[1]{\ensuremath{\conv{D}_{#1}}}

\newcommand{\Domain}{{\ensuremath{\Omega}}}
\newcommand{\Boundary}{{\ensuremath{\Gamma}}}
\newcommand{\CrossSection}{{\ensuremath{S}}}
\newcommand{\CurrentLoop}{{\ensuremath{C}}}

\newcommand{\Wavenumber}{{\ensuremath{k}}}
\newcommand{\BeltramiParam}{{\ensuremath{\lambda}}}
\newcommand{\AreaElem}{{\ensuremath{\D{a}}}}
\newcommand{\VecAreaElem}{{\ensuremath{\D{\vector{a}}}}}

\newcommand{\VecLengthElem}{{\ensuremath{\D{\vector{l}}}}}
\newcommand{\Flux}{{\ensuremath{\Phi}}}
\newcommand{\TAngle}{{\ensuremath{\theta}}}
\newcommand{\PAngle}{{\ensuremath{\phi}}}
\newcommand{\MeanCurv}{{\ensuremath{H}}}

\newcommand{\Normal}{{\vector{n}}}
\newcommand{\SCoord}{{\vector{x}}}
\newcommand{\MetricTensor}{{\ensuremath{G}}}
\newcommand{\detMetricTensor}{{\ensuremath{|\linop{G}|}}}
\newcommand{\detMetricTensord}{{\ensuremath{|\linopd{G}|}}}

\DeclareMathOperator{\sinc}{sinc}
\newcommand{\tb}{{\abbrev{TB}}}                                  
\newcommand{\ghz}{\ensuremath{\mathrm{GHz}}}                     
\newcommand{\nexp}{{\ensuremath{\text{\sc{\small e-}}}}}         
\newcommand{\pexp}{{\ensuremath{\text{\sc{\small e+}}}}}         
\newcommand{\Time}{{\ensuremath{\text{T}}}}                      
\newcommand{\Tsolve}{{\ensuremath{\Time_{solve}}}}
\newcommand{\Tsetup}{\ensuremath{\Time_{setup}}}
\newcommand{\Tsmooth}{\ensuremath{\Time_{smooth}}}
\newcommand{\Tsingular}{\ensuremath{\Time_{singular}}}
\newcommand{\Tquadeval}{\ensuremath{\Time_{quad-eval}}}
\newcommand{\Tlb}{\ensuremath{\Time_{LB}}}

\newcommand{\Nsurf}{{\ensuremath{N_{s}}}}                        
\newcommand{\Nunknown}{\ensuremath{N}}                           
\newcommand{\Nt}{\ensuremath{N_\TAngle}}                         
\newcommand{\Np}{\ensuremath{N_\PAngle}}                         
\newcommand{\htor}{\ensuremath{h_\TAngle}}                       
\newcommand{\hpol}{\ensuremath{h_\PAngle}}                       
\newcommand{\patchdim}{\ensuremath{M}}                           
\newcommand{\quadorder}{\ensuremath{q}}                          
\newcommand{\gmrestol}{\ensuremath{\epsilon_{\textsc{gmres}}}}   
\newcommand{\gmresiter}{\ensuremath{N_{\textsc{gmres}}}}         
\newcommand{\lbiter}{\ensuremath{N_{LB}}}

\DeclareDocumentCommand{\nm}{O{n} O{m}}{_{#1#2}}

\newcommand{\fftdiff}[2]{\ensuremath{\linop{D}_{#2} #1}}       

\newcommand{\pou}{\ensuremath{\eta}}                             

\newcommand{\fnsupp}[1]{\ensuremath{B_{#1}}}                      

\title{\bf\sffamily Taylor States in Stellarators: A Fast High-order
  Boundary Integral Solver}

\author{Dhairya Malhotra\footnote{Courant Institute of Mathematical
    Sciences, New York University, 251 Mercer St, New York, NY
  10012. \emph{Email}: \texttt{malhotra@cims.nyu.edu}. Research
  supported in part by the Office of Naval Research under award number
  \#N00014-17-1-2451, the Simons Foundation/SFARI (560651, AB),
  and the U.S. Department of Energy, Office of Science, Fusion Energy Sciences under Awards No. DE-
FG02-86ER53223 and DE-SC0012398}, \, 
Antoine Cerfon\footnote{Courant Institute of Mathematical
    Sciences, New York University, 251 Mercer St, New York, NY
  10012. \emph{Email}: \texttt{cerfon@cims.nyu.edu}. Research partially supported by
  the Simons Foundation/SFARI (560651, AB),
  and by the U.S. Department of Energy, Office of Science, Fusion Energy Sciences under Awards No. DE-FG02-86ER53223 and DE-SC0012398}, \,
Lise-Marie Imbert-G\'erard\footnote{Department of Mathematics,
  University of Maryland, 4176 Campus Drive, College Park, MD 20742.
  \emph{Email}: \texttt{lmig@math.umd.edu}. Research
  supported in part by the Simons Foundation/SFARI (560651, AB),
  by the Applied Mathematical Sciences Program of the U.S. Department of Energy under
contract DEFGO288ER25053 and by the Office of the Assistant Secretary of Defense for Research and Engineering and
AFOSR under NSSEFF program award FA9550-10-1-0180.}, \\
and Michael O'Neil\footnote{Courant Institute of Mathematical
    Sciences, New York University, 251 Mercer St, New York, NY
  10012. \emph{Email}: \texttt{oneil@cims.nyu.edu}. Research
  supported in part by the Office of Naval Research under award numbers
  \#N00014-17-1-2451 and \#N00014-17-1-2059.}
}

\date{\today}

\begin{document}

\maketitle

\begin{abstract}
  We present a boundary integral equation solver for computing Taylor
relaxed states in non-axisymmetric solid and shell-like toroidal
geometries. The computation of Taylor states in these geometries is a
key element for the calculation of stepped pressure stellarator equilibria. The integral representation of the magnetic field in this work is based on the generalized Debye source formulation, and results in a
well-conditioned second-kind boundary integral equation. The integral
equation solver is based on a spectral discretization of the geometry
and unknowns, and the computation of the associated weakly-singular
integrals is performed with high-order quadrature based on a partition
of unity. The resulting scheme for applying the integral operator is
then coupled with an iterative solver and suitable preconditioners.
Several numerical examples are provided to demonstrate the accuracy and
efficiency of our method, and a direct comparison with the leading
code in the field is reported.

\end{abstract}

\section{Introduction}
The computation of magnetohydrodynamic (MHD) equilibria in toroidal
domains without axisymmetry is a notoriously challenging problem,
having both computational roadblocks as well as touching on
subtle mathematical questions~\cite{grad67,Bruno1996,Hudson2010,HelanderReview}. Until
recently, computational efforts to solve this problem could be divided
in two categories~\cite{Harafuji1989}. The first category of numerical
solvers relies on the assumption of the existence of nested magnetic
flux surfaces throughout the computational
domain~\cite{Bauer1978,Hirshman_VMEC1,Hirshman_VMEC2,Hirshman_VMEC3,Taylor1994}.
These solvers have played an important role in the design of new
non-axisymmetric magnetic confinement devices as well as the analysis
of experimental results obtained from existing ones. However, they are
fundamentally limited in terms of both robustness and accuracy for the
computation of equilibria with both a smooth plasma pressure profile
and smooth magnetic field line pitch. In this regime, this class of
solvers (and the model upon which they are based) is unable to
accurately approximate the singular structures in the current density
which must naturally occur in such
situations~\cite{HelanderReview,Loizu2015,ReimanComparison,Loizu2016}.
On the other hand, an alternative, second class of solvers, does not constrain the
space of solutions to equilibria with nested flux surfaces, and
computes equilibria which may have magnetic islands and chaotic
magnetic field
lines~\cite{Reiman1988,Harafuji1989,Suzuki_Hint2}. These solvers also
play an important role in the magnetic fusion program since they can
be used to study, among other significant questions, the disappearance
of magnetic islands, often called \emph{island healing}, corresponding
to an increase of the plasma pressure~\cite{Hayashi1994,Kanno2005} or
to a change of the coil configuration~\cite{Hudson2002}. They are also
often able to compute the details of the magnetic field configuration
in the vicinity of the plasma edge~\cite{Drevlak2005}. Despite these
additional capabilities, equilibrium codes in the first category are
often favored because existing solvers in the second category converge
substantially slower~\cite{Drevlak2005} and are much more
computationally intensive~\cite{Hudson2007}.

\begin{figure}[t!]
  \centering
  \includegraphics[width=0.5\textwidth]{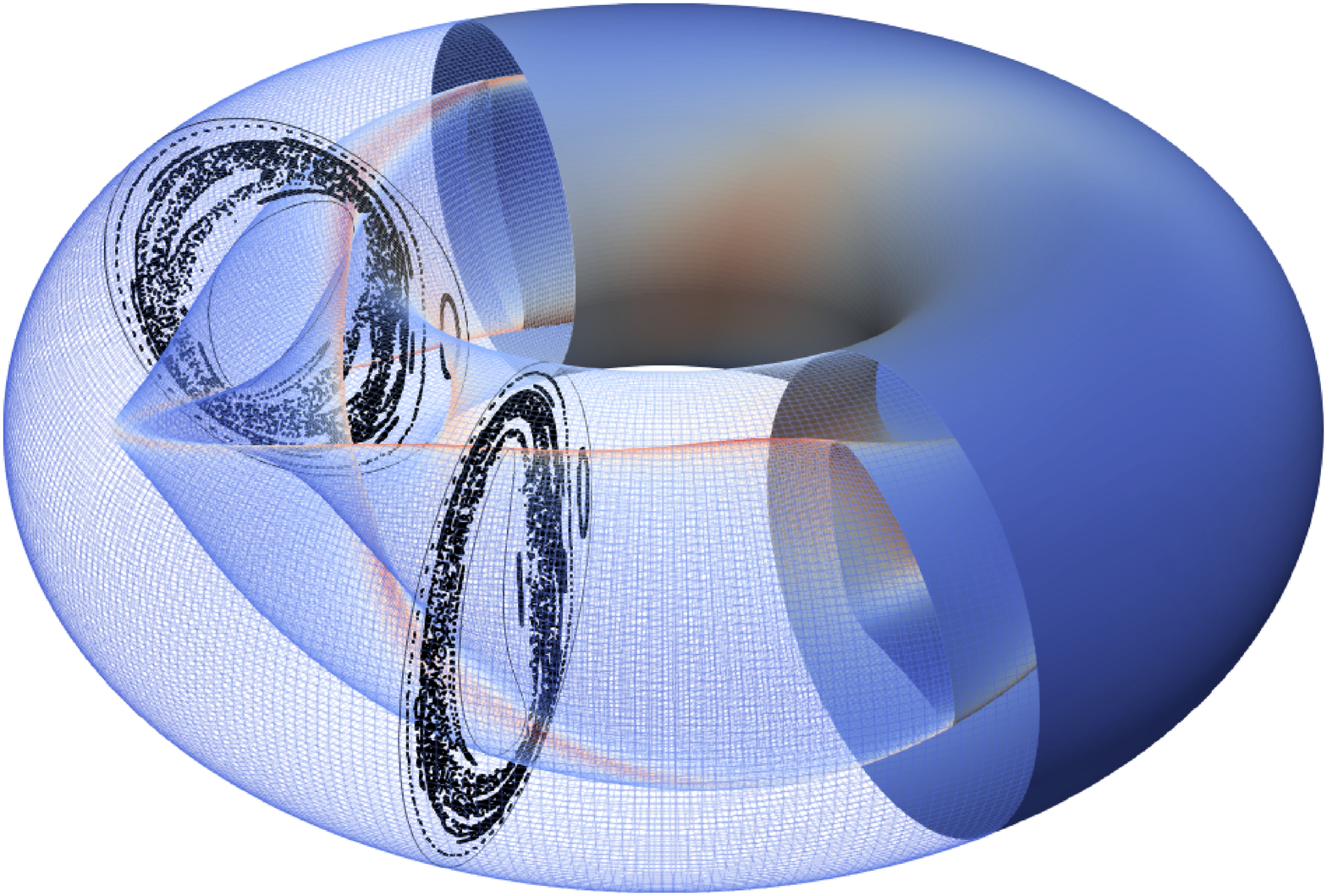}
  \includegraphics[width=0.16\textwidth]{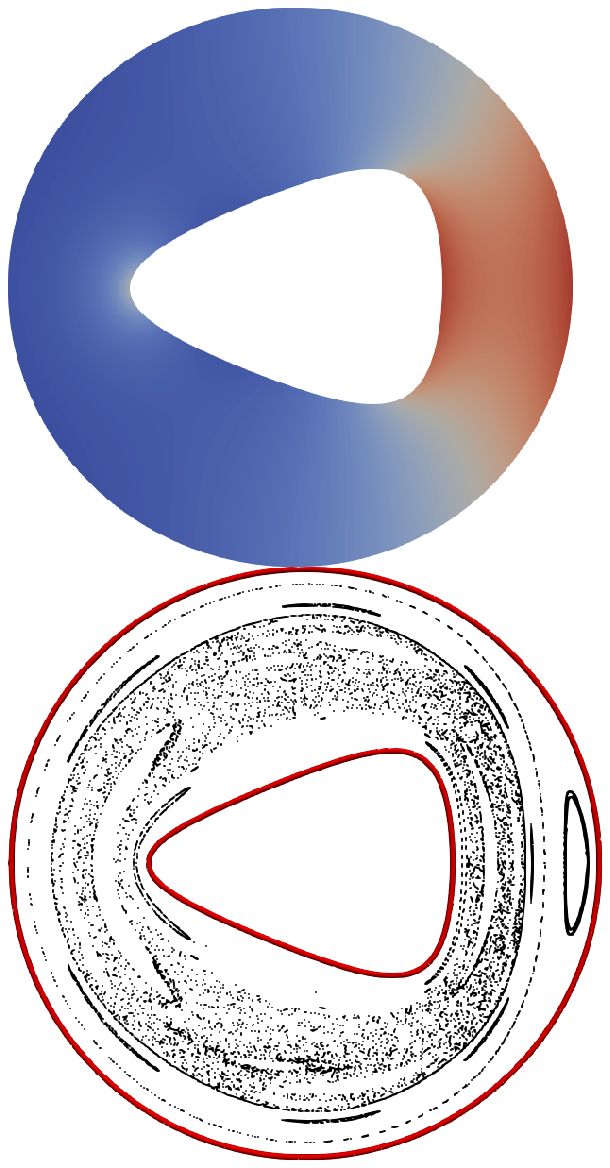}
  \includegraphics[width=0.16\textwidth]{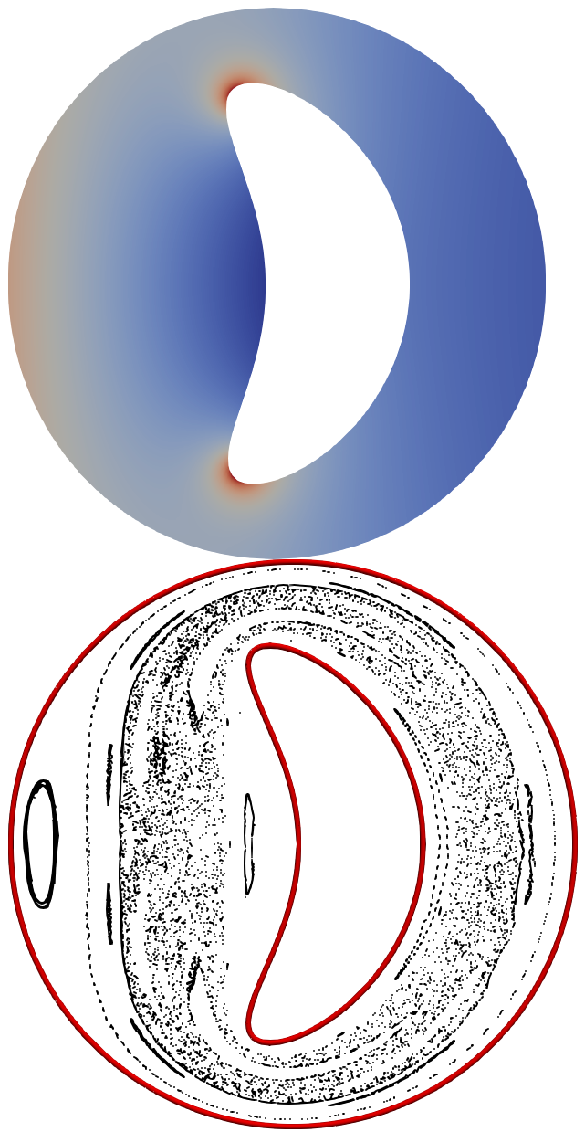}
  \includegraphics[width=0.16\textwidth]{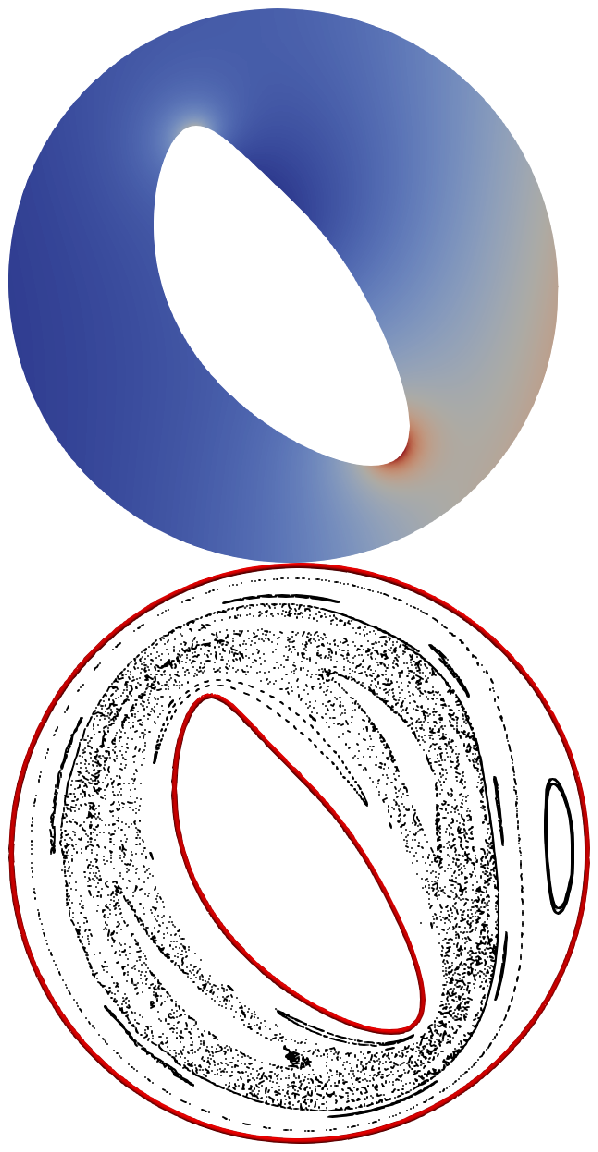}
  \caption[Taylor state in a toroidal-shell]{\label{f:highlight} An
    example of a Taylor state computed in a toroidal-shell domain.
    Using our boundary integral method, we only need to discretize the
    domain boundary. This reduces the dimensionality of the of
    unknowns needed, and leads to significant savings in computational
    work. Once the boundary integral
    equations is solved, the magnetic field~$\vector{B}$ can be
    evaluated at off-surface points very efficiently. On the
    right, we show the magnitude of~$\vector{B}$ in different
    cross-sections of the domain as well as Poincar\'e plots of
    the field in each cross-section, generated by tracing the field
    lines using a \ordinal{10}-order spectral deferred correction
    (SDC) scheme.  }
\end{figure}

Recently, a third approach has been developed which combines aspects
of the two categories of solvers described above. In this approach,
the entire computational domain is subdivided into separate regions,
each with constant pressure, assumed to have undergone Taylor
relaxation~\cite{Taylor1974} to a minimum energy state subject to
conserved fluxes and magnetic helicity. Each of these regions is
assumed to be separated by ideal MHD
barriers~\cite{Hudson2012,HudsonPPCF2012}.  This model has a
significant limitation: for general pressure profiles, solutions of
this model only truly agree with the equations of ideal MHD
equilibrium when one takes the limit of infinitely many
interfaces~\cite{Dennis2013}.  When the number of interfaces is
finite, the model is only a first-order approximation of the ideal MHD
equilibrium equations, independent of any subsequent numerical
discretization.  This third approach is nevertheless very promising
for several reasons. First, unlike the more general ideal MHD
equilibrium case, existence of solutions for the stepped pressure
equilibrium model has been established for tori whose departure from
axisymmetry is sufficiently small~\cite{Bruno1996}.  Second, numerical
solvers for this model can be used to study equilibrium configurations
with magnetic islands and regions with chaotic magnetic field
lines~\cite{Loizu2017} for a computational cost which is substantially
smaller than that of ideal MHD solvers having similar
capabilities~\cite{Reiman1988,Harafuji1989,Suzuki_Hint2}. Furthermore,
they generally have more robust convergence properties. Finally, the
model is well-suited for rigorous error convergence analysis and code
verification~\cite{Hudson2012,Loizu2016Verification}.

At present time, the only equilibrium code in the third category is
known as the Stepped Pressure Equilibrium Code
(SPEC)~\cite{spec-code}. For a given plasma pressure profile,
the computation of equilibria using SPEC is an iterative process. On
each iteration, one first computes the magnetic field~$\vector{B}$ of
the Taylor states inside each region, given by
~$\nabla\times\vector{B}=\lambda\vector{B}$ with~$\lambda$ a
specified constant (possibly different in each region). Then, the
force-balance condition on the ideal MHD interfaces is verified. If
the total pressure $p+|\vct{B}|^2/2$, where $p$ is the plasma
pressure, is continuous across each interface to the desired numerical
precision, the iterative process stops. If it is not continuous to the
desired numerical precision, the shape of the MHD interfaces is
modified in order to satisfy force balance, and a new iteration
starts. This shape optimization is a highly nonlinear procedure.

In this article, we focus on the first step within each iteration,
namely the computation of Taylor states given by
$\nabla\times\vector{B}=\lambda\vector{B}$ in toroidal domains for
which the boundary is given, and~$\lambda$ and the flux conditions are
such that the problem is well-posed~\cite{O_Neil_2018_Taylor}.  In
SPEC, this is done using a Galerkin approach~\cite{Hudson2012} in
which one solves for the magnetic vector potential~$\vector{A}$. The
components of~\vct{A} are represented using a Fourier-Chebyshev
expansion; the Fourier representation captures the double periodicity
in the poloidal and toroidal angles and the Chebyshev representation
is used for the radial variable~\cite{Loizu2016Verification}. In
contrast, in this work we present a boundary integral representation
for the Taylor state~$\vector{B}$ based on a single scalar variable
and solve the associated boundary integral equation. The advantages of
our numerical method as compared to the solver in SPEC are the typical
advantages one expects from integral equation solvers: the number of
unknowns in our approach is much smaller than in SPEC since unknowns
are only needed on the surface of the domain, our solver avoids issues
with the coordinate singularity which occurs when parameterizing the
volume of genus-one domains~\cite{HudsonPPCF2012}, and the
representation immediately leads to a well-conditioned (away from
physical interior resonances) second-kind integral equation which can
be numerically inverted to high-precision. We will present numerical
tests which demonstrate the robustness and efficiency in our approach,
showing that for a given target accuracy, our solver is significantly
faster than the SPEC code and that it avoids conditioning issues
encountered in SPEC.

Our integral equation formulation is based on the same generalized
Debye representation that we presented
in~\cite{O_Neil_2018_Taylor} for the calculation of axisymmetric
Taylor states. One may initially think that generalizing the solver
from axisymmetric geometries to non-axisymmetric ones is
straightforward. This is not so, for several reasons. First, the
numerical discretization of the boundary integral equation formulation
requires quadratures for weakly-singular kernels on the boundary of
the domain. In axisymmetric geometries, the boundary of the domain can
merely be considered a closed curve and fast and accurate quadrature
schemes are readily available (and easy to implement). However, for
the computation of Taylor states in non-axisymmetric geometries, the
boundary is a surface; this makes the accurate and fast computation of
these integrals much more challenging, both from a mathematical and
computational point of view. The same difficulties apply to
the computation of the surface gradient and the inverse
Laplace-Beltrami operator, which are much easier to evaluate along a
closed curve than on general stellarator geometries. Finally, in
axisymmetric domains, there exists simple closed form expressions for
the surface harmonic vector fields required in our formulation, and
this is not the case in non-axisymmetric domains. A significant
portion of this article focuses on the new algorithms we developed to
address these difficulties. We have implemented our method in the form
of a software library called BIEST (Boundary Integral Equation Solver
for Taylor states). The library is made publicly available for use by
the scientific community and to allow independent verification of
our results\footnote{\url{https://github.com/dmalhotra/BIEST}}.

The paper is organized as follows. In Section~\ref{sec:debye}, we
present a brief review of the generalized Debye representation for
magnetic fields satisfying the Taylor state equation, and of the
resulting boundary integral formulation for the computation of the
Taylor state. We also give a new numerically stable representation for
the computation of the flux condition when $\lambda$ approaches zero,
and present an alternative, more direct method for computing vacuum
fields (i.e. the case where $\lambda=0$). In Section~\ref{s:algo} we
provide a detailed description of the numerical solver, 
with a particular emphasis on high-order surface
quadratures for singular kernels and inverting the Laplace-Beltrami
operator. In Section~\ref{s:tests}, we test the accuracy and speed of
the singular quadrature scheme and of the Laplace-Beltrami solver, as
well as the accuracy and speed of the entire solver. In particular, we
compare the performance of our solver with that of SPEC for the
W7-X geometry~\cite{Beidler1990,Pedersen2017}, a stellarator
experiment in Greifswald, Germany. We summarize our work in
Section~\ref{s:summary} and propose directions for future work.

\section{Beltrami fields, Taylor states, and generalized Debye
  sources}\label{sec:debye}

In this section we detail the relationship between time harmonic
electromagnetic fields and Taylor states. These two classes of vector
fields can be directly linked using the generalized Debye integral
representation for Taylor states.  Using this integral representation,
second-kind boundary integral equations are then derived for
force-free fields corresponding to Taylor states in stellarator
geometries. These boundary integral equations can then be used to
solve for stepped-pressure equilibria, as discussed in the
introduction.

As discussed in the introduction, in this article we focus solely on
the task of computing Taylor states in stellarator geometries.
Taylor states are 
described by the equation
\begin{equation}
\nabla \times \vector{B} = \lambda \vector{B}
\label{eq:TaylorState}
\end{equation}
where~$\lambda$ is a given constant throughout the computational
domain, determined based on magnetic energy and helicity.
As is well-known~\cite{Hudson2012,O_Neil_2018_Taylor}, this
partial differential equation needs boundary conditions and flux
constraints on $\vector{B}$ in multiply-connected geometries
in order to be well-posed. We will specify
them shortly, but refrain from doing so at this stage in order to
focus on the generalized Debye representation for~$\vector{B}$ which
\emph{a priori} satisfies~\eqref{eq:TaylorState}, independent of the
boundary conditions or the flux constraints.

Since~$\lambda$ is a real-valued constant, it is easy to
see that by setting~$\vector{E} = i\vct{B}$ the pair~$\vct{E},
\vct{B}$ satisfy the source-free time-harmonic Maxwell equations with
wavenumber~$\lambda$, denoted by THME$(\lambda)$:
\begin{equation}\label{eq:thme}
  \begin{aligned}
    \nabla \times \vct{E} &= i\lambda \vct{B}, \qquad& 
    \nabla \times \vct{B} &= -i\lambda \vct{E},   \\
    \nabla \cdot  \vct{E} &= 0 , &
    \nabla \cdot  \vct{B} &= 0.
  \end{aligned}
\end{equation}
There is a rich literature on integral equation
representations and methods for solving various boundary value
problems for the THME$(\lambda)$. It turns out that the generalized
Debye source representation for solutions to the THME is also
particularly well-suited for solving boundary value problems for Taylor states~\cite{epstein2015}. We now turn to
a very brief overview of this representation, and the derivation of an
integral representation and integral equation for Taylor states in
magnetically confined plasmas.  Various theoretical and numerical
aspects of the generalized Debye source representation can be found
in~\cite{Epstein_2012,epstein2015,O_Neil_2018_Taylor,Epstein_2010,epstein2019}.

\subsection{Generalized Debye sources}

An advantage of the generalized Debye source representation is that
the fields are constructed so as to automatically satisfy Maxwell's
equations, leaving only the boundary condition to be met. This is in
contrast to several classical integral representations for
electromagnetic fields, as well as virtually all direct PDE
discretization schemes (e.g. finite difference and finite element
methods).

Denote by~$\Domain$ a domain (bounded or unbounded) with smooth
boundary~$\Boundary$. In~$\Domain$, any time harmonic electromagnetic
field~$\vct{E},\vct{H}$ with wavenumber $k \in \mathbb{C}$ (having
nonnegative imaginary part)
can be represented using two vector
potentials $\vct{A}, \vct{Q}$ and two scalar potentials $u, v$:
\begin{equation}
  \begin{aligned}
    \vector{E} &= \Imag \Wavenumber \vector{A} - \Grad{u} - \Curl{\vector{Q}}, \\
    \vector{H} &= \Imag \Wavenumber \vector{Q} - \Grad{v} +
    \Curl{\vector{A}}.
  \end{aligned}
  \label{e:debye-maxwell}
\end{equation}
If the vector and scalar potentials satisfy the homogeneous Helmholtz
equation in~$\Domain$, then it is easy to check that the THME$(k)$ are
automatically satisfied so long as:
\begin{equation}\label{eq:aquv}
  \nabla \cdot \vct{A} = iku, \qquad
  \nabla \cdot \vct{Q} = ikv.
\end{equation}
If these potentials are defined by
\begin{equation}
  \begin{aligned}
    \vector{A}(\vector{x}) &= \int\limits_{\Boundary} \HelmKer{k}(\vector{x}-\vector{x'}) ~\vector{j}(\vector{x'}) ~\AreaElem' ,
    &&&&&
    u(\vector{x})          &= \int\limits_{\Boundary} \HelmKer{k}(\vector{x}-\vector{x'}) ~\rho   (\vector{x'}) ~\AreaElem' ,
    \\
    \vector{Q}(\vector{x}) &= \int\limits_{\Boundary} \HelmKer{k}(\vector{x}-\vector{x'}) ~\vector{m}(\vector{x'}) ~\AreaElem' ,
    &&&&&
    v(\vector{x})          &= \int\limits_{\Boundary} \HelmKer{k}(\vector{x}-\vector{x'}) ~\sigma (\vector{x'}) ~\AreaElem' ,
  \end{aligned}
  \label{e:potentials}
\end{equation}
where
$\HelmKer{k}(\vector{r}) = {e^{\Imag \Wavenumber
    |\vector{r}|}}/{4\pi|\vector{r}|}$ is the free-space Green's
function for the Helmholtz equation in three dimensions and
$\AreaElem' = \AreaElem(\vct{x}')$
is the differential area element along $\Boundary$ at the
point $\vector{x'}$, then it can be shown that
condition~\eqref{eq:aquv} is automatically satisfied if the following
two consistency conditions are met:
\begin{align} \label{e:debye-maxwell-consistency}
    \SurfDiv{\vector{j}} = \Imag \Wavenumber \rho , \qquad
    \SurfDiv{\vector{m}} = \Imag \Wavenumber \sigma.
\end{align}
Above, $\SurfDiv{\vector{j}}$ denotes the intrinsic surface divergence
of $\vector{j}$ along $\Boundary$.  With the above formulation, if two
scalar boundary conditions are specified (or one vector condition, as
is the case in scattering from a perfect electric conductor), the
resulting equations for $u, v, \vct{A}, \vct{Q}$ are
under-determined. To avoid this, and ensure automatic satisfaction of
the consistency conditions~\eqref{e:debye-maxwell-consistency}, the
surface vector fields $\vct{j}, \vct{m}$ are explicitly constructed
from $\sigma, \rho$.  The constructions of~$\vct{j}, \vct{m}$,
referred to as the generalized Debye currents, will vary depending on
the boundary conditions being enforced on $\vct{E}, \vct{H}$. For
example, in the case of scattering from a perfect electric conductor,
the boundary conditions to be met are
\begin{equation}
\Normal \times \vct{E} = \vct{F}, \qquad \Normal \cdot \vct{H} = g,
\end{equation}
where $\vct{F}$ and $g$ are data obtained from an incoming
electromagnetic field. For these boundary conditions, it has been
shown that constructing $\vct{j}, \vct{m}$ as
\begin{equation}
  \begin{aligned}
    \vector{j} &= \Imag \Wavenumber \left( \SurfGrad{\InvSurfLap{\rho}}   - \cross{\Normal}{\SurfGrad{\InvSurfLap{\sigma}}} \right) + \vector{j}_{H} , \\
    \vector{m} &= \vct{n} \times \vct{j},
  \end{aligned}
  \label{e:j-m-field}
\end{equation}
leads to a uniquely invertible system of integral equations for the
scalar sources $\sigma, \rho$ and coefficients fixing the projection
of~$\vct{j}$ onto the subspace of harmonic vectorfields
along~$\Gamma$~\cite{Epstein_2010}. Above, we have that $\SurfGrad{}$
is the surface gradient operator, $\InvSurfLap{}$ is the inverse of
the surface Laplacian (Laplace-Beltrami operator) along $\Boundary$
restricted to the class of mean-zero functions, and $\vector{j}_{H}$
is a tangential harmonic vector field along~$\Boundary$ such
that~$\SurfDiv{\vector{j}_{H}} = 0$
and~$\SurfDiv{\cross{\Normal}{\vector{j}_{H}}} = 0$.

The space of harmonic vector fields along~$\Boundary$ has
dimension~$2G$, where $G$ is the genus of the
boundary. In~\cite{O_Neil_2018_Taylor,epstein2019}, a basis for these
harmonic vector fields is known analytically because the
boundary~$\Boundary$ is a surface of revolution. However, in the
present work, they must be obtained computationally. To do so, we
follow a procedure similar to that
in~\cite{O_Neil_2018_LaplaceBeltrami}, which is detailed later on in
Section~\ref{ss:harmonic-vec-algo}.


\subsection{Taylor state formulation\label{ss:formulation}} 


In this section we detail the explicit construction of a vector
field~$\vct{B}$, using generalized Debye sources, that \emph{a priori}
analytically satisfies the Taylor state equation
\begin{equation}
  \nabla \times \vct{B} = \lambda \vct{B}
\end{equation}
and leads to a uniquely invertible integral equation  corresponding to
the boundary condition
\begin{equation}
  \vct{B} \cdot  \vct{n} = 0.
\end{equation}
For more details on this construction, in general,
see~\cite{epstein2015}, and for details in the axisymmetric case
see~\cite{O_Neil_2018_Taylor}. It should be noted that the above
boundary value problem does \emph{not} have a unique solution
if~$\lambda$ is an eigenvalue of the curl operator acting on
vectorfields with vanishing normal component along the
boundary. Furthermore, depending on the particular geometry in which
the above boundary value problem is being solved, it must be augmented
with a suitable number of extra conditions (usually provided as flux
constraints on $\vct{B}$) in order to be well-posed.  For clarity, the
following is a very condensed version of what is contained in the
previously mentioned two references.

\subsubsection{Generalized Debye representation for Taylor states} 

For the application at hand, namely that of computing Taylor states
for magnetically confined plasmas in general toroidal domains, we must
allow for the boundary~$\Boundary$ of the domain~$\Domain$ to be
multiply connected, and possibly have multiple components.
To this end, let~$\Boundary = \partial \Domain$ be a disjoint union
of $\Nsurf$ smooth toroidal surfaces
($\Boundary_1, \Boundary_2, \dots, \Boundary_\Nsurf$).  We want to
compute vector fields $\vector{B}$ in $\Domain$ such that,
\begin{align}
    \begin{aligned}
     \Curl{\vector{B}}             &= \BeltramiParam \vector{B},
     &\qquad &\text{in } \Domain, \\
     \dotprod{\vector{B}}{\Normal} &= 0, &  &\text{on \Boundary}, \\
     \int\limits_{\CrossSection_i} \dotprod{\vector{B}}{\VecAreaElem}
     &= \Flux_i
     & &\text{for } i = 1, \dots, \Nsurf, 
    \end{aligned}
    \label{e:taylor-state}
\end{align}
where $\BeltramiParam$ is a constant real number called the Beltrami
parameter and $\Normal$ is the unit normal vector to~$\Boundary$ pointing
outward from $\Domain$. The surfaces~$\CrossSection_i$ are generally
chosen to capture all possible toroidal and poloidal fluxes of
interest inside~$\Domain$. In these flux
integrals,~$\VecAreaElem = \Normal \AreaElem$ with $\AreaElem$ being
the surface area element and $\Normal$ the oriented normal along the
cross-section $\CrossSection_i$.
In \cite{O_Neil_2018_Taylor}, we  applied the generalized Debye
representation of \cite{Epstein_2012} for computing Taylor states in
axisymmetric geometries.  We use the same integral formulation in the
present work for computing Taylor states in non-axisymmetric
geometries.

\begin{figure}[t!] 
  \centering
  \includegraphics[width=0.56\textwidth]{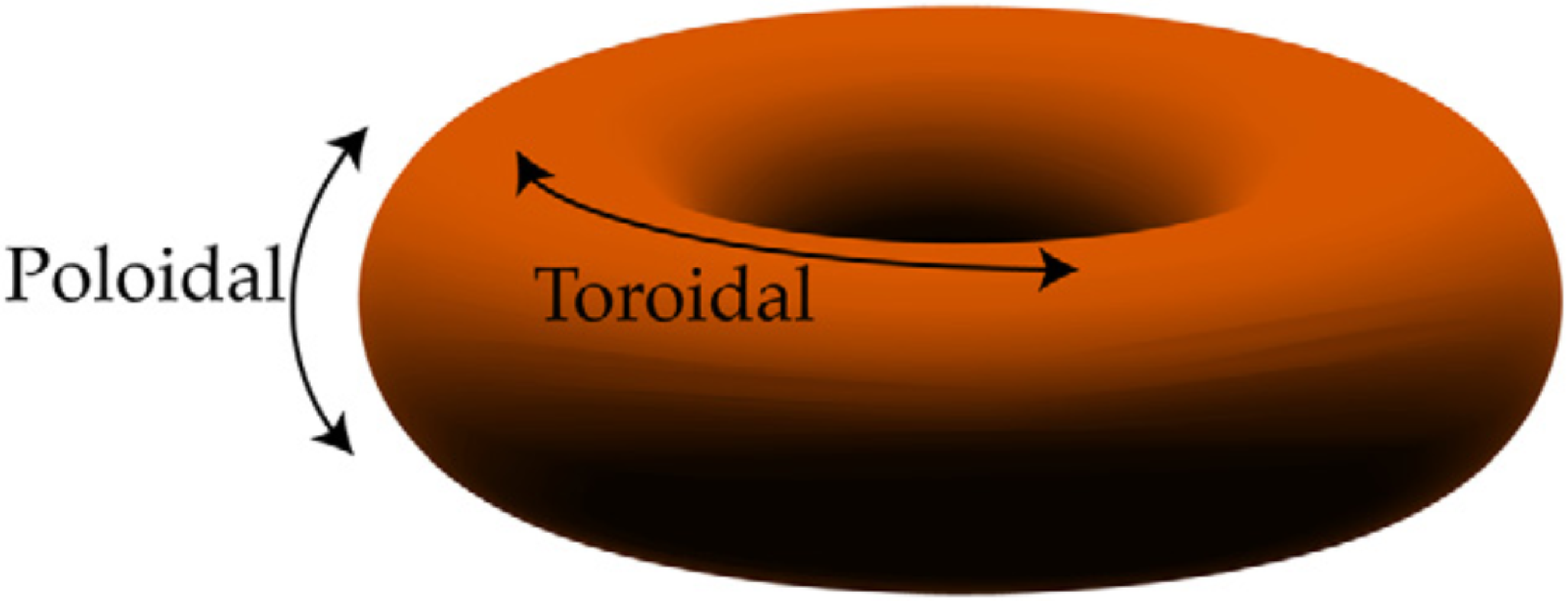}
  \includegraphics[width=0.43\textwidth]{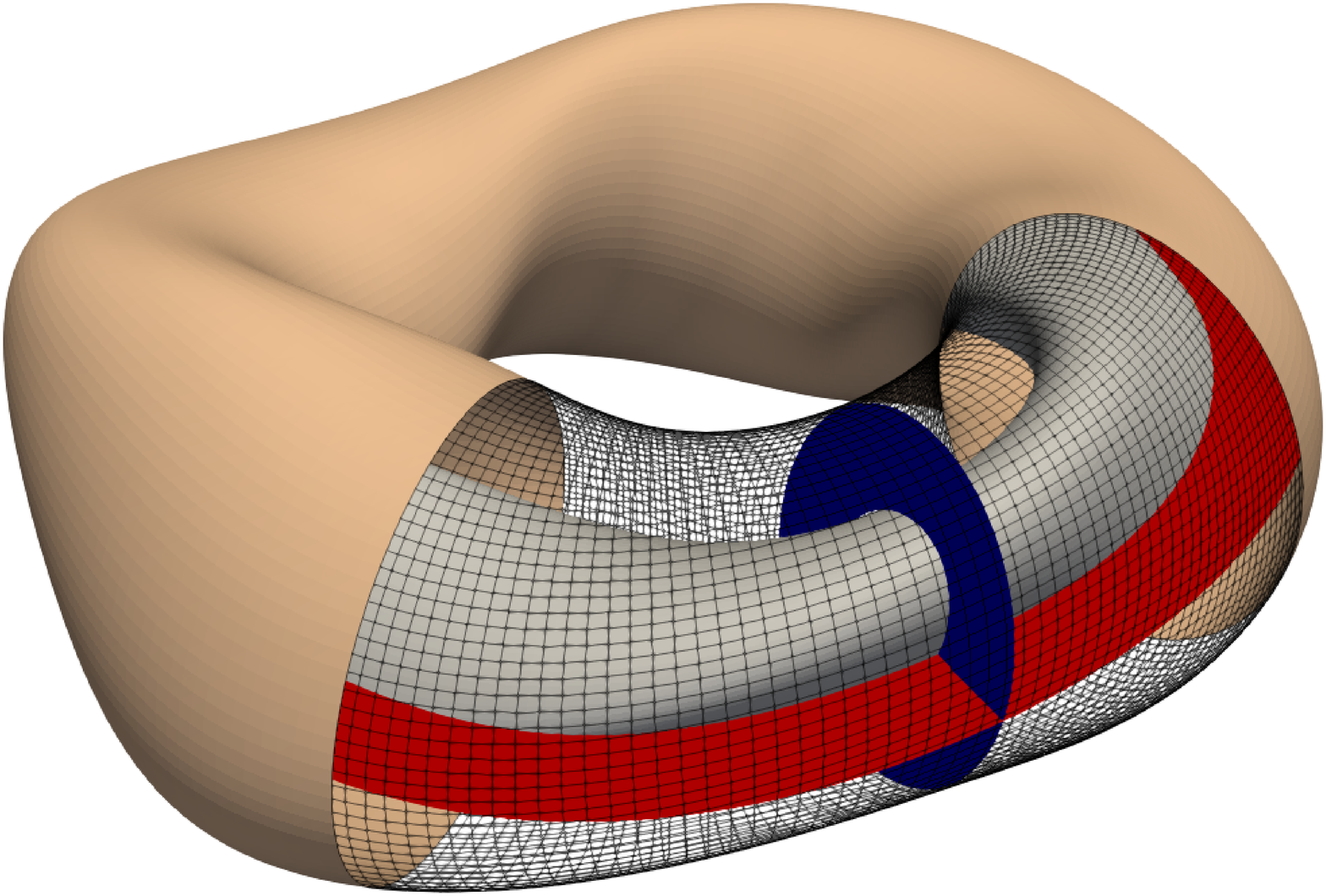}
  \caption{\label{f:geometry} Left: A toroidal surface with the
    toroidal and poloidal directions.  Right: A toroidal shell domain
    is the region between two nested toroidal geometries.  The
    cross-sections of the domain shown in blue and red are called the
    toroidal and poloidal cross-sections respectively.  The flux
    conditions for a Taylor state are given as the prescribed magnetic
    flux through these cross-sections.  }
\end{figure} 

As noted before, a vector field $\vct{B}$ which
satisfies the Beltrami condition~\mbox{$\Curl{\vector{B}} = \BeltramiParam
\vector{B}$} admits the pair
\mbox{$( \vector{E} = \Imag \vector{B}, \vector{H}=\vector{B} )$} which
automatically satisfies
the THME$(\lambda)$.  The vector field \vector{B} can therefore be
represented using the generalized Debye source
representation. Furthermore, due to the relation
$\vector{E} = \Imag \vector{H}$ and the symmetry of the generalized
Debye representation in \pr{e:debye-maxwell}, similar dependencies on
the vector potentials, scalar potentials, Debye currents, and Debye
sources can also be shown~\cite{epstein2015}:
\begin{equation}
\vct{A} = i \vct{Q}, \qquad u = \Imag v, \qquad \rho = i \sigma.
\end{equation}
As a result, the generalized Debye source representation for the Taylor
state~$\vector{B}$ is given as:
\begin{align}
  \vector{B} = \Imag \BeltramiParam \vector{Q} - \Grad{v} + \Imag \Curl{\vector{Q}},
  \label{e:debye-taylor}
\end{align}
where the vector potential~$\vector{Q}$ and the scalar potential $v$
are as defined before in~\pr{e:potentials}. In order to satisfy
$\Curl{\vector{B}} = \BeltramiParam \vector{B}$, the
consistency condition~~$\SurfDiv{\vector{m}} = \Imag \BeltramiParam
\sigma$ must be met.  With
an additional
constraint~~$\cross{\Normal}{\vector{m}} = -\Imag~\vector{m}$, the
surface vector field $\vector{m}$ is determined uniquely up to an
additive harmonic vector field,
\begin{align}\label{e:m-field}
  \begin{aligned}
    \vector{m} &= \vector{m}_0(\vector{\sigma}) + \vector{m}_H ,
  \end{aligned}  
\end{align}
where
\begin{align}
  \vector{m}_0(\vector{\sigma}) &= \Imag \BeltramiParam \left(
                                  \SurfGrad{\InvSurfLap{\sigma}} + \Imag~ \cross{\Normal}{\SurfGrad{\InvSurfLap{\sigma}}} \right)
                                  \label{e:m0-field}
\end{align}
and $\vector{m}_H$ is a harmonic surface vector field such
that:~${\SurfDiv{\vector{m}_H} = 0}$
and~${\cross{\Normal}{\vector{m}_H} = -\Imag~\vector{m}_H}$.  As
described at the beginning of this section, there are $\Nsurf$
linearly independent harmonic vector fields
$\{ \vector{m}^{1}_H, \cdots, \vector{m}^{\Nsurf}_H \}$ on
$\Boundary$ satisfying these conditions.
Therefore, we may represent $\vector{m}_H$ as:
\begin{align}
  \vector{m}_H(\vector{\alpha}) = \sum_{k=1}^{\Nsurf} \alpha_k~\vector{m}^{k}_{H}.
  \label{e:mh-field}
\end{align}
The unknowns $\sigma$ and
$\vector{\alpha} = \{\alpha_1, \cdots, \alpha_{\Nsurf}\}$ must be
determined using the boundary conditions and flux constraints in
\pr{e:taylor-state}.

\subsubsection{Boundary integral formulation} 
The generalized Debye representation for $\vector{B}$, as stated in
\prange{e:debye-taylor}{e:mh-field}, \emph{a priori} satisfies
$\Curl{\vector{B}} = \BeltramiParam \vector{B}$.  However, it does not
automatically satisfy the boundary
condition~$\dotprod{\vector{B}}{\Normal} = 0$ on $\Boundary$, nor the
non-trivial flux constraints.  For any scalar function~$\sigma$ and
coefficients vector~$\vector{\alpha}$, the
field~$\vector{B}$ can be evaluated directly on the
boundary~$\Boundary$ by taking the limit of its value from the
interior of~$\Domain$:
\begin{align}
  \vector{B}(\sigma, \vector{\alpha}) &= -\frac{\sigma}{2}\Normal + \Imag \frac{\cross{\Normal}{\vector{m}}}{2} + \Imag \BeltramiParam \SL{\BeltramiParam}[\vector{m}] - \Grad{\SL{\BeltramiParam}[\sigma]} + \Imag \Curl{\SL{\BeltramiParam}[\vector{m}]} & \text{on \Boundary},  
    \label{e:B-bdry-integ}
\end{align}
where $\vector{m}(\sigma, \vector{\alpha})$ is as defined in
\prange{e:m-field}{e:mh-field} and $\SL{\BeltramiParam}$ is the
single-layer potential
operator, 
  \begin{align*}
    \SL{\BeltramiParam}[f](x) = \int\limits_{\Boundary} \HelmKer{\BeltramiParam}(\vector{x} - \vector{x}') f(\vector{x}') \AreaElem(\vector{x}').
  \end{align*}
  The boundary condition $\dotprod{\vector{B}}{\Normal} = 0$ on $\Boundary$ results in a second-kind integral equation,
  \begin{align}
    -\frac{\sigma}{2} + \conv{K}[\sigma, \vector{\alpha}] &= 0,
    \label{e:taylor-second-kind}
  \end{align}
where $\conv{K}$ denotes the compact operator,
\begin{align}
  \conv{K}[\sigma, \vector{\alpha}] &= \Imag \BeltramiParam \dotprod{\Normal}{\SL{\BeltramiParam}[\vector{m}]} - \partial_{\Normal}{\SL{\BeltramiParam}[\sigma]} + \Imag \dotprod{\Normal}{\Curl{\SL{\BeltramiParam}[\vector{m}]}} .
    \label{e:K-op}
\end{align}

\subsubsection{Flux computation} 

The flux constraints as stated in \pr{e:taylor-state} are difficult to
impose in a boundary integral formulation since we only discretize the
boundary $\Boundary$ and not all of~$\Domain$.  However, using
Stokes' theorem and the fact that by construction
$\Curl{\vector{B}} = \BeltramiParam \vector{B}$, we can relate the
flux of $\vector{B}$ through a cross-section $\CrossSection$ with its
circulation on
$\partial \CrossSection = \CrossSection \Intersection \Boundary$.
Applying the Stokes' theorem and using \pr{e:B-bdry-integ},
\begin{align}
  \int\limits_{\CrossSection}
  \dotprod{\vector{B}(\sigma, \vector{\alpha})}{\D\vct{a}}
  &= \frac{1}{\BeltramiParam} \oint\limits_{\partial \CrossSection} \dotprod{\vector{B}(\sigma, \vector{\alpha})}{\VecLengthElem} \nonumber \\
                                                                                         &= \oint\limits_{\partial \CrossSection} \dotprod{\Imag \SL{\BeltramiParam}[\vector{m}_0 + \vector{m}_H]}{\VecLengthElem}
                                                                                                         + \frac{\Imag}{\BeltramiParam} \oint\limits_{\partial \CrossSection} \dotprod{\left( \frac{\cross{\Normal}{\vector{m}_0}}{2} + \Curl{\SL{\BeltramiParam}[\vector{m}_0]} \right)}{\VecLengthElem} \nonumber \\
  & \hspace{15em} + \frac{\Imag}{\BeltramiParam} \oint\limits_{\partial \CrossSection} \dotprod{\left( \frac{\cross{\Normal}{\vector{m}_H}}{2} + \Curl{\SL{\BeltramiParam}[\vector{m}_H]} \right)}{\VecLengthElem},                                                                                   \label{e:flux-eq}
\end{align}
where $\vector{m}_0(\sigma)$ is as defined in \pr{e:m0-field},
$\vector{m}_H(\vector{\alpha})$ is as defined in \pr{e:mh-field} and
$\VecLengthElem$ is the oriented unit arclength differential.  The first and
the second integral terms in \pr{e:flux-eq} remain bounded as
$\BeltramiParam \to  0$ since 
$\vector{m}_0 \sim \bigO{\BeltramiParam}$; however, computing the
last term becomes numerically unstable due to the $1/\BeltramiParam$
factor. In order to numerically stabilize this calculation, we begin
with the following lemma.

\begin{lem} \label{lem:vacuum-field-circulation} 
For a tangential vector field~$\vector{m}$ on~$\Boundary$ such
that~$\SurfDiv{\vector{m}} = 0$,
and an arbitrary cross section~$\CrossSection$ of the domain~$\Domain$,
\[
  \oint\limits_{\partial \CrossSection} \dotprod{\left(
      \frac{\cross{\Normal}{\vector{m}}}{2} +
      \Curl{\SL{0}[\vector{m}]} \right)}{\VecLengthElem} = 0.
\]
\end{lem} 
\begin{proof} 
  Let $\vector{V}  = \Curl{\SL{0}[\vector{m}]}$. Then, at every
  point in~$\Domain$, in particular for points
  on~$\CrossSection \subset \Domain$, we have that
  \begin{align*}
    \Curl{\vector{V}} &= \Curl{\Curl{\SL{0}[\vector{m}]}} \\
                      &= \Grad{\left(\Div{\SL{0}[\vector{m}]}\right)}
                        - {\Lap{\SL{0}}}[\vector{m}]\\
                      &= \Grad{\SL{0}[{\SurfDiv{\vector{m}}}]} \\
    &= 0.
  \end{align*}
  Furthermore, we have that the limiting value of~$\vct{V}$
  on~$\partial\CrossSection$ is
  \[
    \vct{V} = \frac{\cross{\Normal}{\vector{m}}}{2} +  \nabla \times
    \SL{0}[\vector{m}],
  \]
  as in~\pr{e:B-bdry-integ}.
  Therefore, by Stokes' theorem 
  \[
    \begin{aligned}
      \oint\limits_{\partial \CrossSection} \dotprod{\left(
          \frac{\cross{\Normal}{\vector{m}}}{2} +
          \Curl{\SL{0}[\vector{m}]} \right)}{\VecLengthElem}
      &= \oint\limits_{\partial\CrossSection}
      \dotprod{\vct{V}}{\VecLengthElem} \\
      &= \int\limits_{\CrossSection} \nabla \times \vct{V} 
      \cdot \D\vct{a} \\
      &= 0.
    \end{aligned}
  \]
\end{proof} 

Returning to the previous flux calculation,
since~$\SurfDiv{\vector{m}_H}=0$ we can apply the above lemma to obtain:
\begin{equation}
{\frac{\Imag}{\BeltramiParam}\oint\limits_{\partial \CrossSection}
  \dotprod{\left( \frac{\cross{\Normal}{\vector{m}_H}}{2} + 
      \Curl{\SL{0}[\vector{m}_H]} \right)}{\VecLengthElem} = 0}.
\end{equation}
Subtracting this identity from the last term in
  \pr{e:flux-eq} gives us the following relation,
\begin{align}
  \frac{1}{\BeltramiParam} \oint\limits_{\partial \CrossSection} \dotprod{\left( \frac{\vector{m}_H}{2} + \Imag \Curl{\SL{\BeltramiParam}[\vector{m}_H]} \right)}{\VecLengthElem}
  =
  \oint\limits_{\partial \CrossSection} \dotprod{\Imag \Curl{\left( \frac{\SL{\BeltramiParam}-\SL{0}}{\BeltramiParam} \right) [\vector{m}_H]}}{\VecLengthElem},
    \label{e:flux-last-term}
  \end{align}
  where the operator
  $\left( \frac{\SL{\BeltramiParam}-\SL{0}}{\BeltramiParam} \right)$
  is equivalent to computing a convolution along the boundary
  with the following bounded kernel function:
  \begin{align*}
    \frac{\HelmKer{\BeltramiParam}(\vector{r}) - \HelmKer{0}(\vector{r})}{\BeltramiParam} &= - \frac{ \sin\left( {\BeltramiParam |\vector{r}|}/{2} \right) \sinc\left( {\BeltramiParam |\vector{r}|}/{2} \right) }{4 \pi} + \Imag~\frac{\sinc\left( \BeltramiParam |\vector{r}| \right) }{4 \pi}.
  \end{align*}
  The above expression is obtained straightforwardly using
  trigonometric identities, and is numerically stable and bounded in its
  evaluation, even as~$\lambda|\vct{r}|\to 0$.  Therefore,  the
  right hand side in \pr{e:flux-last-term} can be stably computed  for any
  value of~$\BeltramiParam$.  Using \pr{e:flux-eq,e:flux-last-term},
  the flux constraints can be re-written as
  \begin{multline}
    \oint\limits_{\partial \CrossSection_i} \dotprod{\Imag \SL{\BeltramiParam}[\vector{m}_0 + \vector{m}_H]}{\VecLengthElem}
    + \frac{1}{\BeltramiParam} \oint\limits_{\partial \CrossSection_i} \dotprod{\left( \frac{\vector{m}_0}{2} + \Imag \Curl{\SL{\BeltramiParam}[\vector{m}_0]} \right)}{\VecLengthElem}  
    + \oint\limits_{\partial \CrossSection_i} \dotprod{\Imag \Curl{\left( \frac{\SL{\BeltramiParam}-\SL{0}}{\BeltramiParam} \right) [\vector{m}_H]}}{\VecLengthElem}
    = \Flux_i,
    \label{e:circ-constr}
  \end{multline}
  for $i = 1, \dots, \Nsurf$.  The complete formulation for computing
  Taylor states is given by the boundary integral equation
  \pr{e:taylor-second-kind} and the flux conditions
  \pr{e:circ-constr}.  We will discuss how to discretize and solve
  these equations to obtain the required numerical solution in
  \pr{s:algo}.

\subsubsection{Vacuum fields \label{sss:vacuum-formulation}} 
We now briefly discuss the special case where $\BeltramiParam =
0$. The magnetic field $\vector{B}$ then satisfies
$\nabla\times\vector{B}=\vector{0}$, i.e. it is a vacuum field. In
this case, the boundary limit of the integral representation in
\pr{e:B-bdry-integ} for $\vector{B}$ simplifies to
\begin{align}
  \vector{B}(\sigma, \vector{\alpha}) &= -\frac{\sigma}{2}\Normal + \frac{\vector{m}_{H}}{2} - \Grad{\SL{0}[\sigma]} + \Imag \Curl{\SL{0}[\vector{m}_{H}]} & \text{on \Boundary},
  \label{e:B-bdry-integ-vacuum}
\end{align}
where $\vector{m}_{H}(\vector{\alpha})$ is as defined in \pr{e:mh-field}.
The boundary condition $\dotprod{\vector{B}}{\Normal} = 0$ on $\Boundary$
results in the following second-kind integral equation,
\begin{equation}
  -\frac{\sigma}{2} + \conv{K}[\sigma, \vector{\alpha}] = 0,
  \label{e:vacuum-second-kind}
\end{equation}
where $\conv{K}$ is a compact boundary integral operator given by,
\begin{align}
  \conv{K}[\sigma, \vector{\alpha}] &= -\partial_{\Normal}{\SL{0}[\sigma]}
     + \Imag \dotprod{\Normal}{\Curl{\SL{0}[\vector{m}_{H}]}} .
                                      \label{e:K-op-vacuum}
\end{align}
The flux constraints in \pr{e:circ-constr} cannot be used directly for
vacuum fields since we must first explicitly compute the limit
$\BeltramiParam \to 0$. To avoid this tedious calculation, we propose
a more straightforward method.
We begin by defining a surface vector
field~$\vector{j} = \cross{\Normal}\vector{B}$ along~$\Boundary$, and
then by using a Green's theorem for magnetostatic
fields~\cite{chew1999} we obtain that
$\vector{B} = \Curl{\SL{0}[\vector{j}}]$ in~$\Domain$.  Notice that
$\vct{j}$ is not the same as $\vct{m}_H$.  This relation allows us to
use Stokes' theorem to compute the flux of $\vector{B}$ through a
cross section $\CrossSection$ as the circulation of
$\SL{0}[\vector{j}]$ on $\partial\CrossSection$.  The flux constraints
can then be written as,
\begin{equation}
  \int\limits_{\CrossSection} \dotprod{\vector{B}}{\D\vct{a}}
  = \oint\limits_{\partial \CrossSection} \dotprod{\SL{0}[\vector{j}(\sigma, \vector{\alpha})]}{\VecLengthElem}
  = \Flux_i.
  \label{e:vacuum-flux-constr}
\end{equation}
We discretize and solve \pr{e:vacuum-second-kind,e:vacuum-flux-constr} for the
unknowns $\sigma$ and $\alpha$.
This is discussed in the next section.
The magnetic field $\vector{B}$ on $\Boundary$ can then be computed using
\pr{e:B-bdry-integ-vacuum}.  Notice that in this formulation we do not need to
solve a Laplace-Beltrami problem to evaluate the boundary integral
operator~$\conv{K}$, and the surface convolution operator $\SL{0}$ computes
convolutions with the single-layer Laplace kernel which is much less expensive
to compute than the Helmholtz kernel.

\section{A fully 3D Taylor-state solver \label{s:algo}}

In this section, we  describe a solver for computing Taylor
states in 3D based on the integral equation formulation of the
previous section.  We give a brief overview of the algorithm in
\pr{ss:overview-algo}, and then discuss its building blocks: the
singular quadrature algorithm in \pr{ss:quad-algo}, the spectral
Laplace-Beltrami solver in \pr{ss:lb-algo}, and the
computation of harmonic
vector fields in \pr{ss:harmonic-vec-algo}.  Finally, we summarize the
overall algorithm in \pr{ss:taylor-algo}.

\subsection{An overview of the algorithm\label{ss:overview-algo}} 
We now give a brief overview of our method.  We  discuss the
scheme used for discretizing the boundary data and define the
discretized operators acting on this data, and then discuss how
to solve the resulting discretized linear system.

\subsubsection{Discretization} 
One significant advantage that the boundary integral formulation of
the previous section has over standard PDE formulations is that only
the surface of the domain $\Domain$ needs to be discretized, yielding
an immediate dimension reduction in the number of unknowns.
As discussed in \pr{ss:formulation}, the
boundary is a disjoint union of $\Nsurf$ toroidal
(i.e. doubly-periodic) surfaces.  We
parameterize each surface $\Boundary_i$ by a pair of toroidal and
poloidal angles $(\TAngle,\PAngle) \in [0, 2\pi)^2$, and  each 
surface $\Boundary_i$ is parameterized using a doubly periodic function
$\SCoord_i(\TAngle,\PAngle) \in \Boundary_i$.  Similarly, any function
$f$ along the boundary $\Boundary_i$ can also be parameterized by the
toroidal and poloidal angles so that
$\scalar{f}(\TAngle,\PAngle) =
\scalar{f}(\SCoord_i(\TAngle,\PAngle))$.  We assume that the surface
geometry and the parametrization is oriented and non-degenerate such that
$\det{\linop{\MetricTensor}(\TAngle,\PAngle)} > 0$ for all
$(\TAngle,\PAngle) \in [0, 2\pi)^2$, where $\linop{\MetricTensor}$ is
the metric tensor. 

We discretize each surface by sampling on a uniform grid of points in
the parameter space.  The discretization of
a function~$f$ along a toroidal surface is denoted by~$\vectord{f}$
and is obtained by sampling the function on
an $\Nt \times \Np$ uniform grid,
\begin{align*}
  \scalard{f}\nm[i][j] = \scalar{f}(\TAngle_i, \PAngle_j)
  \qquad \text{for } i=0,\cdots,\Nt-1 \text{ and } j=0,\cdots,\Np-1,
\end{align*}
where $\TAngle_i = 2\pi i / \Nt$ and $\PAngle_j = 2\pi j / \Np$.
Spectrally accurate Fourier approximations of smooth functions~$f$ can
then be
obtained by computing the expansion
\begin{align*}
  \scalar{f}(\TAngle, \PAngle) &\approx \sum\limits_{n=0}^{\Nt-1} \sum\limits_{m=0}^{\Np-1} \fourier{\scalard{f}}\nm e^{\Imag (n \TAngle + m \PAngle)} ,
\end{align*}
where the coefficients $\fourier{\scalard{f}}\nm$ are computed from
the grid values through a discrete Fourier transform,
\begin{align*}
  \fourier{\scalard{f}}\nm &= \frac{1}{2 \pi \Nt \Np}
                             \sum\limits_{i=0}^{\Nt-1}
                             \sum\limits_{j=0}^{\Np-1}
                             \scalard{f}\nm[i][j] \, e^{\Imag (n \TAngle_i + m \PAngle_j)}
    .
\end{align*}
We will denote this Fourier transform operation by
$\fourier{\vectord{f}} = \FFT{\vectord{f}}$ and the inverse operation
by $\vectord{f} = \IFFT{\fourier{\vectord{f}}}$.  This scheme is used
to approximate the surface geometry $\SCoord(\TAngle,\PAngle)$,
discretize the unknown $\sigma(\TAngle,\PAngle)$ in our boundary
integral formulation, and obtain the final solution
$\vector{B}(\TAngle,\PAngle)$ on the boundary. 
We will use $\Nunknown$ to denote the total number of discretization
points across all surfaces
$\{ \Boundary_1, \cdots, \Boundary_{\Nsurf} \}$.  This discretization
scheme of our boundary integral formulation \pr{e:taylor-second-kind}
results in a pseudo-spectral Nystr\"om-like discretization scheme
where the unknowns are point-values of~$\sigma$ and the boundary
conditions are enforced point-wise at the $\Nunknown$ surface
discretization points, but intermediate operations (e.g. forming the
system matrix) are carried out spectrally.

\subsubsection{Pseudo-spectral differentiation} 
Our discretization scheme allows us to easily compute derivatives of
functions on toroidal surfaces through diagonal operators acting on
the Fourier coefficients at a modest loss of numerical accuracy.
Therefore, for a function
$\scalar{f}(\TAngle, \PAngle)$ with discretization $\vectord{f}$ on an
$\Nt \times \Np$ grid and with the Fourier coefficients given by
$\fourier{\vectord{f}} = \FFT{\vectord{f}}$, we can compute
\begin{align*}
  \left( \fourier{\vectord{f}_{\TAngle}} \right)\nm = \Imag~n~\fourier{\scalard{f}}\nm, &&
\left( \fourier{\vectord{f}_{\PAngle}} \right)\nm = \Imag~m~\fourier{\scalard{f}}\nm
\end{align*}
where $\fourier{\vectord{f}_{\TAngle}}$ and
$\fourier{\vectord{f}_{\PAngle}}$ are the Fourier coefficients for
$\partial\scalar{f}/\partial{\TAngle}$ and
$\partial\scalar{f}/\partial{\PAngle}$ respectively.  The 
derivatives on the regular $\Nt \times \Np$ grid can then be evaluated using the inverse Fourier transform,
$\vectord{f}_{\TAngle} = \IFFT{\fourier{\vectord{f}_{\TAngle}}}$ and
$\vectord{f}_{\PAngle} = \IFFT{\fourier{\vectord{f}_{\PAngle}}}$.  We
will denote the spectral differentiation operators by the
notation $\vectord{f}_{\TAngle} = \fftdiff{\vectord{f}}{\TAngle}$ and
$\vectord{f}_{\PAngle} = \fftdiff{\vectord{f}}{\PAngle}$.  When
implemented using the Fast Fourier Transform (\abbrev{FFT}), these
operators require $\bigO{\Nunknown \log \Nunknown}$ work for
discretizations on grids with $\Nunknown$ points.  In our
implementation, we use the multithreaded \abbrev{FFTW} library
\cite{FFTW3_2005} to efficiently compute the Fourier transform and its
inverse.

    For a boundary $\Boundary$ with the discretized surface points given by \vectord{\SCoord}, the pseudo-spectral differentiation operators can be used to compute tangent vectors evaluated at the discretization points,
    \begin{align*}
      \vectord{\SCoord}_\TAngle = \fftdiff{\vectord{\SCoord}}{\TAngle}, &&
      \text{and} &&
      \vectord{\SCoord}_\PAngle = \fftdiff{\vectord{\SCoord}}{\PAngle}
      .
    \end{align*}
    We can then also compute the normal vector $\vectord{\Normal}$ and the metric tensor $\vectord{\MetricTensor}$ at each discretization point on the surface. We start from the continuous expressions,
    \begin{align}
      \Normal = \frac{\cross{\SCoord_\TAngle}{\SCoord_\PAngle}}{\|\cross {\SCoord_\TAngle}{\SCoord_\PAngle}\|} &&
      \text{and} &&
      \MetricTensor =
      \begin{bmatrix}
        \|\SCoord_\TAngle\|^2  & \SCoord_\TAngle\cdot \SCoord_\PAngle \\[1ex]
        \SCoord_\TAngle\cdot\SCoord_\PAngle & \|\SCoord_\PAngle\|^2 \\
      \end{bmatrix}
      ,
      \label{e:normal-metrictensor}
    \end{align}
    where $\SCoord_{\TAngle} = \pderiv{\SCoord}{\TAngle}$ and
    $\SCoord_{\PAngle} = \pderiv{\SCoord}{\PAngle}$ are the continuous surface
    tangent vectors. Evaluating the expressions above pointwise for each discretization grid point and for the corresponding components of $\vectord{\SCoord}_\TAngle$ and $\vectord{\SCoord}_\PAngle$ gives the discrete vector $\vectord{\Normal}$ and discrete metric tensor $\vectord{\MetricTensor}$. If necessary, we flip the direction of the normal vectors so that
    they always point outward from the domain $\Domain$ (this is
     not necessary if the parameterization of~$\Boundary$ was properly
     oriented).
    We will use $\vectord{\Normal}$ and $\vectord{\MetricTensor}$ to construct the discretized operators required in our boundary integral formulation.

\subsubsection{Discretized operators} 

The boundary integral formulation of this work requires that we
discretize the linear compact operator $\conv{K}$ defined in
\pr{e:K-op} for Taylor states and in \pr{e:K-op-vacuum} for vacuum
fields.  To do this, we need to numerically apply the layer-potential
operators ~$\SL{\BeltramiParam}$,
~$\partial_{\Normal}\SL{\BeltramiParam}$ ~and
~$\Curl{\SL{\BeltramiParam}}$.
    These operators compute singular integrals on the boundary and therefore require special quadratures.
    We will discuss the details of their implementation in \pr{ss:quad-algo}.
    In addition, for Taylor states, we also need to compute the surface vector field $\vector{m}$, as described in \pr{e:m-field,e:m0-field}.
    This requires the discrete surface gradient operator $\SurfGrad{}$, the inverse surface Laplacian $\InvSurfLap{}$ and the harmonic vector fields $\{ \vector{m}^{1}_H, \cdots, \vector{m}^{\Nsurf}_H \}$.
    We will discuss the implementation of operators $\SurfGrad{}$ and $\InvSurfLap{}$ in \pr{ss:lb-algo} and the computation of harmonic vector fields in \pr{ss:harmonic-vec-algo}.
    Constructing the discrete operator $\conv{K}$ from these building
    blocks is then straightforward.

    To discretize the flux constraints in \pr{e:circ-constr} and   
    \pr{e:vacuum-flux-constr}, we again
    need the layer-potential operators discussed above. The
    circulation integral is evaluated using a trapezoidal quadrature
    rule, which is spectrally accurate for this smooth periodic data.

  \subsubsection{Linear solver} 
  After discretization, the linear system can be written in a block
  matrix form as,
    \begin{align}
      \begin{bmatrix}
        \linop{A_{11}} & \linop{A_{12}} \\
        \linop{A_{21}} & \linop{A_{22}} \\
      \end{bmatrix}
      \begin{bmatrix}
        \vector{\sigma} \\
        \vector{\alpha} \\
      \end{bmatrix}
      =
      \begin{bmatrix}
        \vector{0} \\
        \vector{\Flux}\\
      \end{bmatrix}
      ,
      \label{e:block-matrix}
    \end{align}
    where the first block-row corresponds to the boundary condition in
    \pr{e:taylor-second-kind} and the second block-row corresponds to
    the flux constraints in \pr{e:circ-constr} for Taylor states; and likewise
    for the vacuum field formulation \pr{e:vacuum-second-kind,e:vacuum-flux-constr}.  Let $\Nsurf$ be the
    number of toroidal surfaces and $\Nunknown$ be the total number of
    surface discretization points.  Then, the RHS boundary condition
    $\vector{0}$ is a column vector of length $\Nunknown$ and the
    prescribed flux conditions $\vector{\Flux}$ is a column vector of
    length $\Nsurf$.  Similarly, the unknown $\vector{\sigma}$ is a
    column vector of length $\Nunknown$ and $\vector{\alpha}$ is a
    column vector of length $\Nsurf$.

    For the problems considered in this work, $\Nsurf$ is small
    ($\Nsurf \leq 3$); however, the number of discretization points
    $\Nunknown$ can be as large as $\bigO{1\pexp5}$.  It is
    computationally inefficient to directly construct the boundary
    integral operator $\linop{A_{11}}$, with dimensions
    $\Nunknown \times \Nunknown$, and it is instead implemented in a
    matrix-free form.  Then, to solve \pr{e:block-matrix}, we first
    solve the following linear system for $\linop{D}$,
    \begin{align}
      \linop{A_{11}} \linop{D} = -\linop{A_{12}}
      \label{e:block-solve0}
    \end{align}
    where $\linop{D}$ is a matrix of size $\Nunknown \times \Nsurf$.
    We solve \pr{e:block-solve0} using \abbrev{GMRES} and this requires one solve for each of the $\Nsurf$ columns of $\linop{D}$.
    Then, we compute the unknowns $\vector{\sigma}$ and $\vector{\alpha}$ as follows,
    \begin{align}
      \vector{\alpha} &= \left( \linop{A_{21}} \linop{D} + \linop{A_{22}} \right)^{-1} \vector{\Flux}, \label{e:block-solve1} \\
      \vector{\sigma} &= \linop{D} \vector{\alpha}, \label{e:block-solve2}
    \end{align}
    where the inverse of the $\Nsurf \times \Nsurf$ matrix in
    \pr{e:block-solve1} is computed directly.

\subsubsection{Evaluating \vector{B}} 

Once the unknowns $\vector{\sigma}$ and $\vector{\alpha}$ have been
computed, we can evaluate the field $\vector{B}$ on the boundary
$\Boundary$ by discretizing and evaluating \pr{e:B-bdry-integ}.  This
is analogous to the discretization of the boundary integral operator
discussed above.
To evaluate $\vector{B}$ at off-surface points, we discretize and
evaluate the representation in \pr{e:debye-taylor}.  In our current
implementation, we evaluate the layer-potentials at off-surface points
using trapezoidal quadratures on an upsampled mesh.  If the
off-surface points are close to the surface, then the upsample factor
must be large and this scheme becomes very expensive.  In the future,
for such cases, we plan to use specialized quadratures for
near-singular evaluation such as Quadrature by Expansion
(QBX)~\cite{Kl_ckner_2013,Epstein_2013,Rachh_2017}.


\subsection{High-order singular quadrature along surfaces\label{ss:quad-algo}} 
\newcommand{\Kernel}{\mathcal{G}}
\newcommand{\polar}[1]{\ensuremath{#1^{p}}}
\newcommand{\PInterp}{\ensuremath{\linopd{R}}}
\newcommand{\Density}{\ensuremath{f}}
\newcommand{\BoundaryIntegral}{\ensuremath{\mathsf{U}}}
\newcommand{\SmoothIntegral}{\ensuremath{\mathsf{U_G}}}
\newcommand{\SingularIntegral}{\ensuremath{\mathsf{U_L}}}
\newcommand{\SingularQuad}{\ensuremath{v}}
\newcommand{\PQuadWt}{\ensuremath{\scalard{w}}}
We now describe a quadrature rule to evaluate layer-potentials due to
singular kernel functions.  Our algorithm is adapted from the work of
\cite{Bruno_2001a,Bruno_2001b} for the Helmholtz kernel and the work
of \cite{Ying_2006} for the Stokes kernel.
\pr{f:sing-quad} gives a brief overview of the algorithm.

We consider a single toroidal surface $\Boundary$, parameterized by
the toroidal and poloidal angles $(\TAngle,\PAngle) \in [0,2\pi)^2$
and discretized on a uniform $\Nt \times \Np$ grid as discussed in
\pr{ss:overview-algo}.  The surface position is given by the function
$\SCoord(\TAngle,\PAngle)$; its discretization along the $N$-point
grid is denoted by~$\vectord{\SCoord}$.  And as before, given a
function $\scalar{\Density}(\TAngle,\PAngle)$ along the grid, its
discretization is denoted by~$\vectord{\Density}$.  The
layer-potential due to a kernel function $\Kernel$ at a target point
$\SCoord_0 = \SCoord(\TAngle_0,\PAngle_0)$ on $\Boundary$ is
given by
    \begin{align}
      \Kernel[f](\SCoord_0)
      =
      \int\limits_{[0,2\pi)^2} \!\! \Kernel(\SCoord_0 - \SCoord(\TAngle,\PAngle)) ~ \scalar{\Density}(\TAngle,\PAngle) ~ \sqrt{\det{\linop{\MetricTensor}(\TAngle,\PAngle)}} \D{\TAngle} \D{\PAngle}
      \label{e:sin-bdry-int}
    \end{align}
    where $\linop{\MetricTensor}$ is the metric tensor.  The kernel
    function may be weakly-singular such as the single-layer Helmholtz
    kernel $\HelmKer{\BeltramiParam}$ or of principal-value type such as the
    gradient of the single-layer Helmholtz kernel
    $\Grad{\HelmKer{\BeltramiParam}}$.  To evaluate the
    layer-potential due to such kernels, we use a partition of unity
    to split the boundary integral into two integrals
    $\Kernel[f](\SCoord_0) = \SmoothIntegral + \SingularIntegral$ such
    that $\SmoothIntegral$ is a smooth global integral over the entire
    surface and $\SingularIntegral$ is a singular local integral.  We
    define $\SmoothIntegral$ and $\SingularIntegral$ as
    \begin{align}
      \SmoothIntegral &= \int\limits_{[0,2\pi)^2} \!\! \left( 1 - \pou(\TAngle,\PAngle) \right) ~ \Kernel(\SCoord_0 - \SCoord) ~ \scalar{\Density}(\TAngle,\PAngle) ~ \sqrt{\det{\linop{\MetricTensor}}} \D{\TAngle} \D{\PAngle} \label{e:smooth-int} \\
      \SingularIntegral &= \int\limits_{\fnsupp{\pou}}              \pou(\TAngle,\PAngle)         ~ \Kernel(\SCoord_0 - \SCoord) ~ \scalar{\Density}(\TAngle,\PAngle) ~ \sqrt{\det{\linop{\MetricTensor}}} \D{\TAngle} \D{\PAngle} \label{e:singular-int}
    \end{align}
    where $\pou(\TAngle,\PAngle)$ is called a floating partition of
    unity centered at $(\TAngle_0, \PAngle_0)$ and $\fnsupp{\pou}$ is
    the support of $\pou$.  In order for the first integrand to be
    smooth and the second integrand to have compact support, we
    require that the function $\pou$ be smooth, have compact support
    \fnsupp{\pou}, and that $\pou(\TAngle,\PAngle)=1$ in a neighborhood
    around $(\TAngle_0,\PAngle_0)$.  To realize such a function, we
    define $\pou$ as,
    \begin{align}
      \pou(\TAngle,\PAngle) &= \chi\left( \frac{2}{\patchdim} \sqrt{ \left( \frac{\TAngle-\TAngle_0}{\htor} \right)^2 + \left( \frac{\PAngle - \PAngle_0}{\hpol} \right)^2 } \right)
      \label{e:pou-def}
    \end{align}
    where~$\htor = 2\pi/\Nt$~ and ~$\hpol = 2\pi/\Np$~ are the
    discretization grid spacing in $\TAngle$ and $\PAngle$,
    respectively, and ~$\chi:[0,\infty) \rightarrow [0,1]$ is a smooth
    function such that $\chi(\rho)=1$ in a neighborhood of zero and
    $\chi(\rho)=0$ for $\rho \geq 1$.  The parameter $\patchdim$ is used
    to control the width of $\pou$ so that its support $\fnsupp{\pou}$
    lies on an $\patchdim \times \patchdim$ subset of the original
    $\Nt \times \Np$ discretized grid.  We will next discuss
    quadratures to numerically evaluate the smooth integral
    $\SmoothIntegral$ and the singular integral $\SingularIntegral$.

\begin{figure}[t] 
    \centering
    \includegraphics[width=0.64\textwidth]{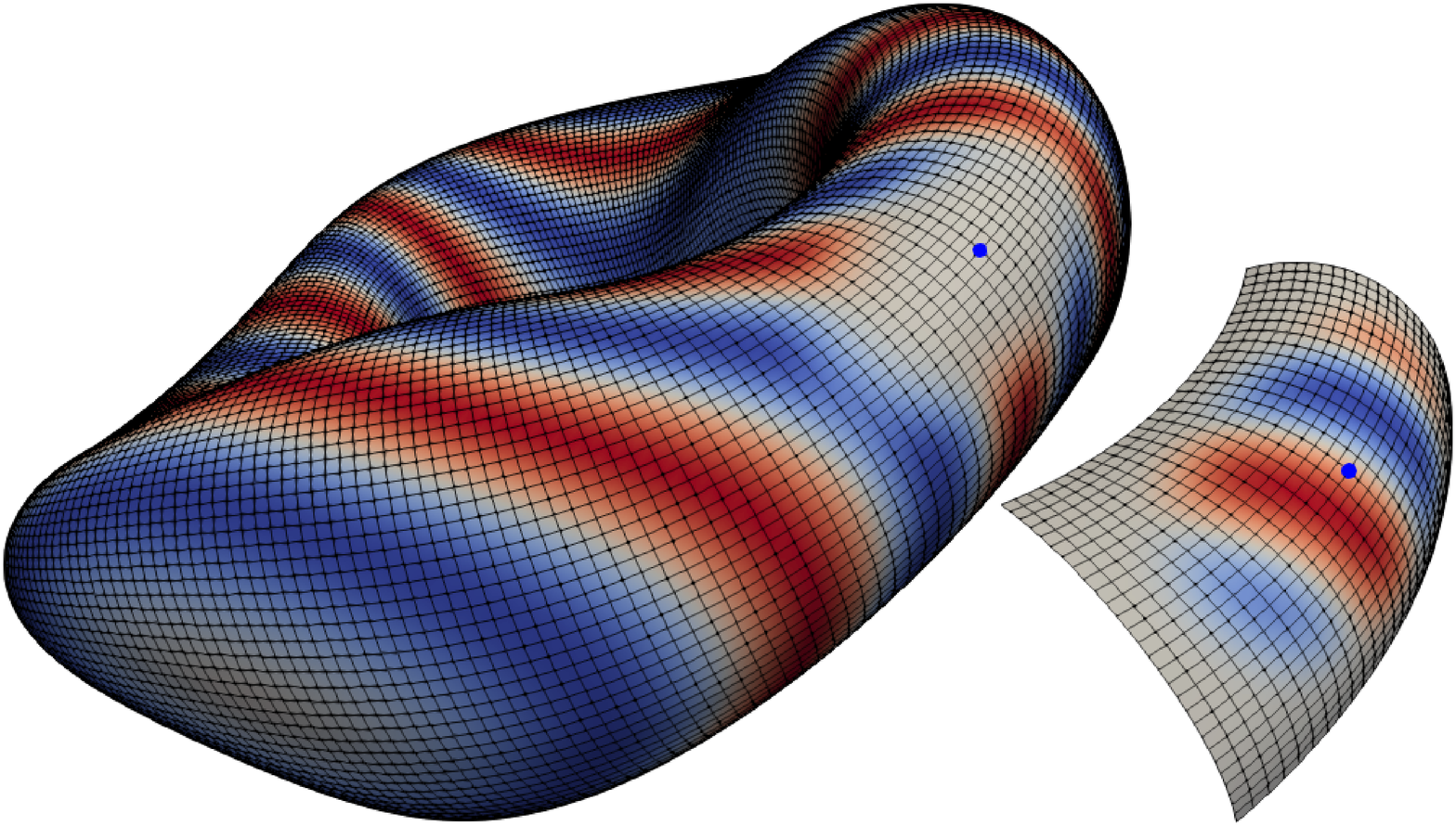}
    \resizebox{0.35\textwidth}{!}{\begin{tikzpicture}[scale=7] 

  \newcommand*\spherex[2]{(sin(#1)*sin(#2))}
  \newcommand*\spherey[2]{(sin(#1)*cos(#2)*cos(15)+sin(15)*cos(#1))}
  \newcommand*\spherez[2]{(cos(#1)*cos(15)+sin(15)*sin(#1)*cos(#2))}

  \newcommand*{\PLat}{60}
  \newcommand*{\PLon}{35}
  \newcommand*{\dLat}{12}%
  \newcommand*{\dLon}{8}%
  \newcommand*{\XX}[2]{(sin(#2)*0.5*(0.5+0.5*cos(#1)))}
  \newcommand*{\YY}[2]{(cos(#2)*0.5*(0.5+0.5*cos(#1)))}

  \newcommand*\drawpatch[3]{ 
    \foreach \y in {-0.5,-0.45,...,0.55} {
      \foreach \x in {-0.5,-0.45,...,0.5} {
        \draw[black!100] ({\spherex{#1+\x*#3}{#2+\y*#3}},{\spherez{#1+\x*#3}{#2+\y*#3}}) --  ({\spherex{#1+(\x+0.05)*#3}{#2+\y*#3}},{\spherez{#1+(\x+0.05)*#3}{#2+\y*#3}}) ;
        \draw[black!100] ({\spherex{#1+\y*#3}{#2+\x*#3}},{\spherez{#1+\y*#3}{#2+\x*#3}}) --  ({\spherex{#1+\y*#3}{#2+(\x+0.05)*#3}},{\spherez{#1+\y*#3}{#2+(\x+0.05)*#3}}) ;
      }
    }

    \foreach \p in {10,30,...,370} {
      \foreach \t in {20,25,...,180} {
        \draw[red!100,line width=0.15mm] ({\spherex{#1+\XX{\t}{\p}*#3}{#2+\YY{\t}{\p}*#3}},{\spherez{#1+\XX{\t}{\p}*#3}{#2+\YY{\t}{\p}*#3}}) --  ({\spherex{#1+(\XX{(\t+5.00)}{\p})*#3}{#2+\YY{(\t+5.00)}{\p}*#3}},{\spherez{#1+(\XX{(\t+5.00)}{\p})*#3}{#2+\YY{(\t+5.00)}{\p}*#3}}) ;
      }
    }
    \foreach \p in {10,15,...,370} {
      \foreach \t in {20,40,...,180} {
        \draw[red!100,line width=0.15mm] ({\spherex{#1+\XX{\t}{\p}*#3}{#2+\YY{\t}{\p}*#3}},{\spherez{#1+\XX{\t}{\p}*#3}{#2+\YY{\t}{\p}*#3}}) --  ({\spherex{#1+(\XX{(\t)}{\p+5.00})*#3}{#2+\YY{(\t)}{\p+5.00}*#3}},{\spherez{#1+(\XX{(\t)}{\p+5.00})*#3}{#2+\YY{(\t)}{\p+5.00}*#3}}) ;
      }
    }
    \draw[red,fill=red] ({\spherex{#1+\XX{20}{130}*#3}{#2+\YY{20}{130}*#3}},{\spherez{#1+\XX{20}{130}*#3}{#2+\YY{20}{130}*#3}}) circle [radius=0.2pt];

    \foreach \y in {0.25,0.30,...,0.55} {
      \foreach \x in {-0.45,-0.40,...,-0.20} {
        \draw[black!100,line width=0.2mm] ({\spherex{#1+\y*#3}{#2+\x*#3}},{\spherez{#1+\y*#3}{#2+\x*#3}}) --  ({\spherex{#1+\y*#3}{#2+(\x+0.05)*#3}},{\spherez{#1+\y*#3}{#2+(\x+0.05)*#3}}) ;
      }
    }
    \foreach \y in {0.25,0.30,...,0.5} {
      \foreach \x in {-0.45,-0.40,...,-0.15} {
        \draw[black!100,line width=0.2mm] ({\spherex{#1+\y*#3}{#2+\x*#3}},{\spherez{#1+\y*#3}{#2+\x*#3}}) --  ({\spherex{#1+(\y+0.05)*#3}{#2+\x*#3}},{\spherez{#1+(\y+0.05)*#3}{#2+\x*#3}}) ;
      }
    }
    \foreach \y in {0.25,0.30,...,0.55} {
      \foreach \x in {-0.45,-0.40,...,-0.15} {
        \draw[black,fill=black] ({\spherex{#1+\y*#3}{#2+\x*#3}},{\spherez{#1+\y*#3}{#2+\x*#3}}) circle [radius=0.2pt];
      }
    }

    \draw[blue,fill=blue] ({\spherex{#1}{#2}},{\spherez{#1}{#2}}) circle [radius=0.3pt];

    \node (A) at ({\spherex{#1-0.50*#3}{#2-0.60*#3}},{\spherez{#1-0.50*#3}{#2-0.60*#3}}) {};
    \node (B) at ({\spherex{#1+0.50*#3}{#2-0.55*#3}},{\spherez{#1+0.50*#3}{#2-0.55*#3}}) {};
    \node (C) at ({\spherex{#1+0.55*#3}{#2-0.50*#3}},{\spherez{#1+0.55*#3}{#2-0.50*#3}}) {};
    \node (D) at ({\spherex{#1+0.55*#3}{#2+0.50*#3}},{\spherez{#1+0.55*#3}{#2+0.50*#3}}) {};

  } 
  \drawpatch{\PLat}{\PLon}{70}

\end{tikzpicture} 
}
    \caption{\label{f:sing-quad}
      To evaluate the layer potential $\BoundaryIntegral=\Kernel[f]$ at the blue target point on the surface, we use a partition of unity function $\pou(\TAngle,\PAngle)$ (centered at the target point) to separate the boundary integral into a smooth integral $\SmoothIntegral = \Kernel[(1-\pou)f]$ over the entire surface and a singular integral $\SingularIntegral = \Kernel[\pou f]$ over the support of $\pou$.
      We evaluate the smooth integral $\SmoothIntegral$ using trapezoidal quadrature rule.
      We compute the singular integral $\SingularIntegral$ in polar coordinates using a trapezoidal rule in the angular direction and a Gauss-Legendre quadrature rule in the radial direction.
      The transformation from the regular grid discretization to a polar grid discretization is done using Lagrange interpolation on a $12\times12$ grid.
    }
  \end{figure} 

  \subsubsection{Quadrature for the smooth integrand} 
  The integral in \pr{e:smooth-int} is smooth, and we can therefore
  use standard quadratures for smooth functions to evaluate this
  integral with high accuracy.  Since we already have a uniform
  discretization of periodic functions on an $\Nt \times \Np$ grid,
  the trapezoidal rule is applicable and spectrally convergent:
    \begin{align}
      \SmoothIntegral &~\approx~ \sum\limits_{i,j} \left( 1 - \pou\nm[i][j] \right) \scalard{\Kernel}\nm[i][j] \scalard{\Density}\nm[i][j] \scalard{w}\nm[i][j]
      %
      \label{e:smooth-quad}
    \end{align}
    where ~~$\pou\nm[i][j] = \pou(\TAngle_i,\PAngle_j)$,
    ~~$\scalard{\Kernel}\nm[i][j] = \Kernel(\SCoord_0 -
    \vectord{\SCoord}\nm[i][j])$ ~~and
    ~~$\scalard{w}\nm[i][j] = \frac{4 \pi^2}{\Nt \Np}
    \sqrt{\det{\vectord{\MetricTensor}\nm[i][j]}}$ are the trapezoidal
    quadrature weights times the differential area element.  When the
    discretization is sufficiently fine to resolve the geometry, the
    area element $\sqrt{\det{\linop{\MetricTensor}}}$ and the density
    $\scalar{\Density}$, the quadrature error is then determined by
    the smoothness of the term $\left( 1 - \pou \right) \Kernel$.
    This depends on how well the floating partition of unity $\pou$
    can be resolved on an $\patchdim \times \patchdim$ grid of points
    and the effectiveness of $(1-\pou)$ at screening the kernel
    singularity.  We will discuss the choice of $\pou$ later in this
    section.

    When computing the potential at all grid points, the summation in
    \pr{e:smooth-quad} requires $\bigO{\Nunknown^2}$ total cost for
    $\Nunknown$ grid points.  This cost can be improved to
    $\bigO{\Nunknown}$ by using fast multipole method (FMM)
    acceleration \cite{Greengard_1987,Cheng_2006,Ying_2004}.  However,
    we have not used FMM acceleration in the present work since the
    performance improvement would not be significant for the problem
    sizes considered here.

  \subsubsection{Quadrature for the singular integrand} 
  Since the integral kernel $\Kernel$
  has a singularity, we can not compute
  $\SingularIntegral$ directly using quadrature rules for smooth
  functions.  To evaluate this integral, we apply a change of
  variables and compute the integral in polar coordinates.  The
  coordinate transform removes the leading order kernel singularity
  allowing us to use standard quadratures and achieve high-order
  accuracy.
  In particular, we apply
  the following change of variables,
    \begin{align*}
      \TAngle = \TAngle_0 + (\patchdim / 2) \htor \rho \sin \omega,
      && \text{and} &&
      \PAngle = \PAngle_0 + (\patchdim / 2) \hpol \rho \cos \omega,
    \end{align*}
    where $\rho$ and $\omega$ are the radial and angular coordinate variables in the polar coordinate space and the pole is at the target $\SCoord_0$.
    The boundary integral in \pr{e:singular-int} can then be written as,
    \begin{align}
      \SingularIntegral = \int\limits_{0}^{\pi} \int\limits_{-1}^{1} \chi(\rho) ~ \Kernel(\SCoord_0 - \SCoord) ~ \scalar{\Density}(\rho, \omega) ~ \sqrt{\det{\linop{\MetricTensor}}} ~ \frac{\patchdim^2 \htor \hpol}{4} |\rho| \D{\rho} \D{\omega}
      \label{e:polar-trans}
    \end{align}
    where $\frac{\patchdim^2 \htor \hpol}{4} |\rho|$ is the Jacobian of the transformation.
    For weakly-singular kernels (such as $\HelmKer{\BeltramiParam}$)  with $1/|\SCoord_0 - \SCoord|$ singularity, the Jacobian cancels the singularity making the integrand bounded.
    In fact, in \cite{Bruno_2001a} it was shown that when the kernel $\Kernel$ is the real part of $\HelmKer{\BeltramiParam}$, then the integrand is smooth and periodic in $\rho$.
    This allowed them to use trapezoidal quadratures in both $\rho$ and $\omega$.
    For integral operators such as
    $\Grad_\Gamma \mathcal S_\lambda $, the integrand in
    \pr{e:polar-trans} is still singular and the integral must be
    understood in the Cauchy principal value sense.  In
    \cite{Ying_2006}, the authors showed that for the double-layer
    Stokes pressure kernel (a hypersingular kernel), the integral in
    $\rho$ can be computed using a quadrature rule for smooth
    integrands if the quadrature is symmetric about the origin.  The
    proof relies on showing that the integrand can be written as the
    sum of a smooth function and an antisymmetric singular function.
    The smooth integral can be computed using standard quadratures and
    the singular integral (which has zero principal value) evaluates
    to zero due to the symmetry of the quadrature rule.  The same
    proof applies to many other  kernels, including
    $\Grad{\HelmKer{\BeltramiParam}}$.
    For bounded kernel functions (such as the imaginary part of $\HelmKer{\BeltramiParam}$), the integrand in \pr{e:polar-trans} is only $C^0$ continuous at the origin.
    To avoid having to deal with the real and the imaginary parts of $\HelmKer{\BeltramiParam}$ separately, we compute the integral in $\rho$ separately in each of the two intervals $[-1,0]$ and $[0,1]$ using a Gauss-Legendre quadrature rule.
    For the integral in $\omega$, we use trapezoidal quadrature rule.
    %

    To apply the quadrature in polar coordinates, we need to first
    evaluate the surface position $\SCoord$ and the density $\Density$
    at a new set of quadrature nodes.  This requires interpolation
    from the uniform Cartesian grid discretization in $\TAngle$ and
    $\PAngle$ to a polar grid discretization around the target point.
    For each interpolation point in the polar grid, we select a
    $12 \times 12$ array of values (surrounding the interpolation
    point) in the original discretization and use Lagrange
    interpolation to approximate the value at the interpolation point.
    Since we evaluate the polynomial interpolation close to the center of the interpolation grid, we do not suffer from Runge's phenomenon and can use high-order interpolation.
    Using $\ordinal{12}$ order interpolation provides sufficient accuracy without being too expensive.
    We use $\PInterp$ to denote the interpolation operator that computes the polar grid discretization from the uniform discretization.
    Then, at the polar grid points, the surface position is given by $\polar{\vectord{\SCoord}} = \PInterp \vectord{\SCoord}$ and the density is given by $\polar{\vectord{\Density}} = \PInterp \vectord{\Density}$.
    For the surface normal $\polar{\vectord{\Normal}}$ and the metric tensor $\polar{\vectord{\MetricTensor}}$, interpolating directly from values on the regular grid leads to poor accuracy.
    Instead, we first interpolate the tangent vectors
    $\polar{\vectord{\SCoord}_{\TAngle}} = \PInterp \vectord{\SCoord}_{\TAngle}$ and $\polar{\vectord{\SCoord}_{\PAngle}} = \PInterp \vectord{\SCoord}_{\PAngle}$;
    and then use these to compute $\polar{\vectord{\Normal}}$ and $\polar{\vectord{\MetricTensor}}$ as was done in \pr{e:normal-metrictensor}.

    The polar grid quadrature nodes are given by a tensor-product rule with trapezoidal quadrature of order $\quadorder$ in $\omega$ and Gauss-Legendre quadrature of order $\quadorder$ in $\rho$ in the intervals $[-1,0]$ and $[0,1]$.
    This gives us the quadrature nodes $(\rho_i, \omega_i)$ and the quadrature weights $\PQuadWt_i$ for $i=0, \cdots, 2\quadorder^2-1$.
    Once the values $\polar{\vectord{\SCoord}}$, $\polar{\vectord{\Normal}}$ and $\polar{\vectord{\MetricTensor}}$ have been evaluated at the polar quadrature nodes through the interpolation process discussed above, we can apply the quadrature rule to approximate the integral in \pr{e:polar-trans} as
    ~~$\SingularIntegral \approx \Transpose{\vectord{\SingularQuad}} \polar{\vectord{\Density}}$~
    where for $i=0, \cdots, 2\quadorder^2-1$,
    \begin{align*}
      \scalard{\SingularQuad}_i = \chi(\rho_i) ~ \Kernel(\SCoord_0 -
      \polar{\vectord{\SCoord}}_i) ~ \sqrt{\det
      \vectord{\MetricTensor}_i} ~ \frac{\patchdim^2 \htor \hpol}{4}
      |\rho_i| ~ \PQuadWt_i .
    \end{align*}
    Since the interpolation operator $\PInterp$ is a sparse matrix,
    for a given surface geometry we can easily precompute the
    quantity $\Transpose{\vectord{\SingularQuad}} \PInterp$ and then
    compute
    $\SingularIntegral \approx (\Transpose{\vectord{\SingularQuad}}
    \PInterp) \vectord{\Density}$ several times for different
    densities $\vectord{\Density}$.  In the work of \cite{Bruno_2001a,
      Bruno_2001b} and \cite{Ying_2006}, the interpolation was done by
    first upsampling the Cartesian discretization using \abbrev{FFT}
    and then using low-order polynomial interpolation.  While this may
    make the interpolation less expensive, it results in a dense
    interpolation operator and therefore makes the precomputation
    infeasible.  Our scheme requires $\bigO{\quadorder^2}$ cost per
    target for the precomputation and then the quadrature can be
    applied several times (with $\bigO{\patchdim^2}$ cost per target)
    without the expensive kernel evaluations and interpolation of the
    density $\vectord{\Density}$ in each application.
    For a specified quadrature accuracy and a given surface, the
    optimal choice of the parameters $\patchdim$ and $\quadorder$ must
    be determined empirically.  From~\cite{Ying_2006}, to obtain
    convergence as the mesh is refined, a good rule of thumb for scaling
    the parameters is given by ~$\patchdim = \bigO{\Nunknown^{1/4}}$~
    and ~$\quadorder = \bigO{\Nunknown^{1/4}}$, where
    $\Nunknown = \Nt \Np$ is the number of surface mesh points.  From
    the error analysis in \cite{Ying_2006}, we expect spectral
    convergence with mesh refinement as we also appropriately scale
    the parameters $\patchdim$ and $\quadorder$.

  \begin{figure} 
    \centering
    \resizebox{0.46\textwidth}{!}{\begin{tikzpicture} 
    	\begin{axis}[ xlabel={$\vector{r}$}, ylabel={$\chi(\vector{|r|})$} ]
      \addplot[line width=0.25mm, color=blue] coordinates { 
        (-1.00, 2.3195e-16)
        (-0.98, 4.9964e-14)
        (-0.96, 5.2646e-12)
        (-0.94, 2.9487e-10)
        (-0.92, 9.4630e-09)
        (-0.90, 1.8613e-07)
        (-0.88, 2.3837e-06)
        (-0.86, 2.0982e-05)
        (-0.84, 1.3322e-04)
        (-0.82, 6.3689e-04)
        (-0.80, 2.3820e-03)
        (-0.78, 7.2091e-03)
        (-0.76, 1.8189e-02)
        (-0.74, 3.9278e-02)
        (-0.72, 7.4278e-02)
        (-0.70, 1.2552e-01)
        (-0.68, 1.9286e-01)
        (-0.66, 2.7358e-01)
        (-0.64, 3.6302e-01)
        (-0.62, 4.5565e-01)
        (-0.60, 5.4626e-01)
        (-0.58, 6.3064e-01)
        (-0.56, 7.0597e-01)
        (-0.54, 7.7083e-01)
        (-0.52, 8.2493e-01)
        (-0.50, 8.6882e-01)
        (-0.48, 9.0353e-01)
        (-0.46, 9.3037e-01)
        (-0.44, 9.5068e-01)
        (-0.42, 9.6574e-01)
        (-0.40, 9.7668e-01)
        (-0.38, 9.8447e-01)
        (-0.36, 9.8990e-01)
        (-0.34, 9.9359e-01)
        (-0.32, 9.9605e-01)
        (-0.30, 9.9764e-01)
        (-0.28, 9.9864e-01)
        (-0.26, 9.9925e-01)
        (-0.24, 9.9960e-01)
        (-0.22, 9.9980e-01)
        (-0.20, 9.9991e-01)
        (-0.18, 9.9996e-01)
        (-0.16, 9.9998e-01)
        (-0.14, 9.9999e-01)
        (-0.12, 1.0000e+00)
        (-0.10, 1.0000e+00)
        (-0.08, 1.0000e+00)
        (-0.06, 1.0000e+00)
        (-0.04, 1.0000e+00)
        (-0.02, 1.0000e+00)
        ( 0.00, 1.0000e+00)
        ( 0.02, 1.0000e+00)
        ( 0.04, 1.0000e+00)
        ( 0.06, 1.0000e+00)
        ( 0.08, 1.0000e+00)
        ( 0.10, 1.0000e+00)
        ( 0.12, 1.0000e+00)
        ( 0.14, 9.9999e-01)
        ( 0.16, 9.9998e-01)
        ( 0.18, 9.9996e-01)
        ( 0.20, 9.9991e-01)
        ( 0.22, 9.9980e-01)
        ( 0.24, 9.9960e-01)
        ( 0.26, 9.9925e-01)
        ( 0.28, 9.9864e-01)
        ( 0.30, 9.9764e-01)
        ( 0.32, 9.9605e-01)
        ( 0.34, 9.9359e-01)
        ( 0.36, 9.8990e-01)
        ( 0.38, 9.8447e-01)
        ( 0.40, 9.7668e-01)
        ( 0.42, 9.6574e-01)
        ( 0.44, 9.5068e-01)
        ( 0.46, 9.3037e-01)
        ( 0.48, 9.0353e-01)
        ( 0.50, 8.6882e-01)
        ( 0.52, 8.2493e-01)
        ( 0.54, 7.7083e-01)
        ( 0.56, 7.0597e-01)
        ( 0.58, 6.3064e-01)
        ( 0.60, 5.4626e-01)
        ( 0.62, 4.5565e-01)
        ( 0.64, 3.6302e-01)
        ( 0.66, 2.7358e-01)
        ( 0.68, 1.9286e-01)
        ( 0.70, 1.2552e-01)
        ( 0.72, 7.4278e-02)
        ( 0.74, 3.9278e-02)
        ( 0.76, 1.8189e-02)
        ( 0.78, 7.2091e-03)
        ( 0.80, 2.3820e-03)
        ( 0.82, 6.3689e-04)
        ( 0.84, 1.3322e-04)
        ( 0.86, 2.0982e-05)
        ( 0.88, 2.3837e-06)
        ( 0.90, 1.8613e-07)
        ( 0.92, 9.4630e-09)
        ( 0.94, 2.9487e-10)
        ( 0.96, 5.2646e-12)
        ( 0.98, 4.9964e-14)
        ( 1.00, 2.3195e-16)
      }; 
      \addplot[line width=0.25mm, color=red] coordinates { 
        (-1.00, 0.00000)
        (-0.98, 0.00000)
        (-0.96, 0.00000)
        (-0.94, 0.00001)
        (-0.92, 0.00022)
        (-0.90, 0.00138)
        (-0.88, 0.00475)
        (-0.86, 0.01149)
        (-0.84, 0.02235)
        (-0.82, 0.03756)
        (-0.80, 0.05698)
        (-0.78, 0.08026)
        (-0.76, 0.10694)
        (-0.74, 0.13650)
        (-0.72, 0.16845)
        (-0.70, 0.20237)
        (-0.68, 0.23783)
        (-0.66, 0.27450)
        (-0.64, 0.31208)
        (-0.62, 0.35029)
        (-0.60, 0.38892)
        (-0.58, 0.42777)
        (-0.56, 0.46665)
        (-0.54, 0.50542)
        (-0.52, 0.54390)
        (-0.50, 0.58197)
        (-0.48, 0.61946)
        (-0.46, 0.65624)
        (-0.44, 0.69214)
        (-0.42, 0.72699)
        (-0.40, 0.76062)
        (-0.38, 0.79283)
        (-0.36, 0.82341)
        (-0.34, 0.85213)
        (-0.32, 0.87878)
        (-0.30, 0.90310)
        (-0.28, 0.92487)
        (-0.26, 0.94390)
        (-0.24, 0.96002)
        (-0.22, 0.97315)
        (-0.20, 0.98330)
        (-0.18, 0.99062)
        (-0.16, 0.99541)
        (-0.14, 0.99816)
        (-0.12, 0.99945)
        (-0.10, 0.99990)
        (-0.08, 0.99999)
        (-0.06, 1.00000)
        (-0.04, 1.00000)
        (-0.02, 1.00000)
        ( 0.00, 1.00000)
        ( 0.02, 1.00000)
        ( 0.04, 1.00000)
        ( 0.06, 1.00000)
        ( 0.08, 0.99999)
        ( 0.10, 0.99990)
        ( 0.12, 0.99945)
        ( 0.14, 0.99816)
        ( 0.16, 0.99541)
        ( 0.18, 0.99062)
        ( 0.20, 0.98330)
        ( 0.22, 0.97315)
        ( 0.24, 0.96002)
        ( 0.26, 0.94390)
        ( 0.28, 0.92487)
        ( 0.30, 0.90310)
        ( 0.32, 0.87878)
        ( 0.34, 0.85213)
        ( 0.36, 0.82341)
        ( 0.38, 0.79283)
        ( 0.40, 0.76062)
        ( 0.42, 0.72699)
        ( 0.44, 0.69214)
        ( 0.46, 0.65624)
        ( 0.48, 0.61946)
        ( 0.50, 0.58197)
        ( 0.52, 0.54390)
        ( 0.54, 0.50542)
        ( 0.56, 0.46665)
        ( 0.58, 0.42777)
        ( 0.60, 0.38892)
        ( 0.62, 0.35029)
        ( 0.64, 0.31208)
        ( 0.66, 0.27450)
        ( 0.68, 0.23783)
        ( 0.70, 0.20237)
        ( 0.72, 0.16845)
        ( 0.74, 0.13650)
        ( 0.76, 0.10694)
        ( 0.78, 0.08026)
        ( 0.80, 0.05698)
        ( 0.82, 0.03756)
        ( 0.84, 0.02235)
        ( 0.86, 0.01149)
        ( 0.88, 0.00475)
        ( 0.90, 0.00138)
        ( 0.92, 0.00022)
        ( 0.94, 0.00001)
        ( 0.96, 0.00000)
        ( 0.98, 0.00000)
        ( 1.00, 0.00000)
      }; 
      \end{axis}
    \end{tikzpicture}} 
    \hfill
    \resizebox{0.48\textwidth}{!}{\begin{tikzpicture} 
      \begin{semilogyaxis}[ xlabel={$k$}, ylabel={$|~\fourier{\chi}_k~|$} ]
      \addplot[line width=0.25mm, color=blue] coordinates { 
        ( 0, 1.2034e+03)
        ( 4, 8.5791e+01)
        ( 8, 2.8023e+00)
        (12, 2.3454e-01)
        (16, 3.1457e-02)
        (20, 3.0184e-05)
        (24, 5.7921e-04)
        (28, 9.6271e-05)
        (32, 5.9175e-06)
        (36, 3.0748e-07)
        (40, 5.6965e-08)
        (44, 9.1146e-10)
        (48, 3.3848e-10)
        (52, 1.3372e-11)
        (56, 1.2160e-12)
        (60, 1.3962e-13)
        (64, 3.8051e-15)
      }; 
      \addplot[line width=0.25mm, color=red] coordinates { 
        (  0, 1.0757e+03)
        (  4, 2.2502e+00)
        (  8, 1.1128e+00)
        ( 12, 1.3686e-01)
        ( 16, 7.0827e-02)
        ( 20, 7.3460e-03)
        ( 24, 4.6009e-03)
        ( 28, 3.0312e-03)
        ( 32, 9.2658e-04)
        ( 36, 3.6808e-05)
        ( 40, 1.5348e-04)
        ( 44, 1.1759e-04)
        ( 48, 5.4795e-05)
        ( 52, 1.5493e-05)
        ( 56, 1.0998e-06)
        ( 60, 5.1995e-06)
        ( 64, 4.4299e-06)
        ( 68, 2.6194e-06)
        ( 72, 1.1627e-06)
        ( 76, 3.1283e-07)
        ( 80, 6.8646e-08)
        ( 84, 1.7908e-07)
        ( 88, 1.6612e-07)
        ( 92, 1.1489e-07)
        ( 96, 6.4758e-08)
        (100, 2.8744e-08)
        (104, 7.4657e-09)
        (108, 2.8062e-09)
        (112, 6.2896e-09)
        (116, 6.2682e-09)
        (120, 4.8499e-09)
        (124, 3.1858e-09)
        (128, 1.7889e-09)
        (132, 8.0567e-10)
        (136, 2.0599e-10)
        (140, 1.0450e-10)
        (144, 2.2582e-10)
        (148, 2.3920e-10)
        (152, 2.0139e-10)
        (156, 1.4727e-10)
        (160, 9.5518e-11)
        (164, 5.4113e-11)
        (168, 2.4835e-11)
        (172, 6.3556e-12)
        (176, 3.8260e-12)
        (180, 8.3022e-12)
        (184, 9.2539e-12)
        (188, 8.3153e-12)
        (192, 6.5820e-12)
        (196, 4.7144e-12)
        (200, 3.0547e-12)
        (204, 1.7487e-12)
        (208, 8.1635e-13)
        (212, 2.1128e-13)
        (216, 1.4078e-13)
        (220, 3.1078e-13)
        (224, 3.6170e-13)
        (228, 3.4504e-13)
        (232, 2.8892e-13)
        (236, 2.2237e-13)
        (240, 1.6092e-13)
        (244, 1.0137e-13)
        (248, 5.7956e-14)
        (252, 2.8744e-14)
        (256, 6.9542e-15)
      }; 
      \legend{current work, {Bruno and Kunyansky}}
      \end{semilogyaxis}
    \end{tikzpicture}} 
  \caption{\label{f:pou-comparison} We compare the function $\chi$ as
    defined in this work with the one used in \cite{Bruno_2001b}.  The
    partition of unity function used in this work is close to $1$ in a
    larger region around the origin, and is more effective in
    separating the singular and smooth parts of the kernel function.
    Furthermore, the periodic Fourier coefficients
    $\fourier{\vector{\chi}}$ decay more rapidly, yielding a higher
    resolution splitting for the same number of grid points.}
  \end{figure}
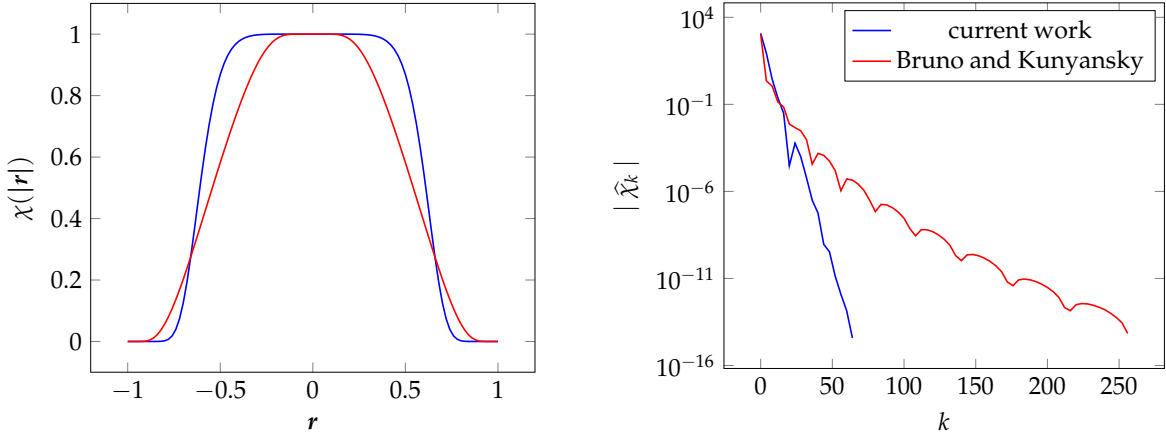 
  \subsubsection{Partition of unity function} 
  In \pr{e:pou-def}, the floating partition of unity $\pou$ is defined
  in terms of another function $\chi$ that we define as:
    \begin{align}
      \chi(\rho) = \exp\left( - 36 |\rho|^{8} \right) .
      \label{e:pou-function}
    \end{align}
    Notice that $\chi$ as defined above does not strictly satisfy the
    condition $\chi(\rho)=0$ for $\rho \geq 0$; however, for computations
    in double-precision it is sufficiently small for $\rho \geq 1$.
    In \pr{f:pou-comparison}, we compare $\chi$ as defined here with
    the definition in \cite{Bruno_2001b}.
    We show $\chi$ in the plot
    on the left and its periodic Fourier coefficients in the plot on the right.
    Comparing the two choices, we observe that our definition of
    $\chi$ is close to $1$ in a larger neighborhood around the origin
    and therefore, it is more effective at screening the kernel
    singularity in \pr{e:smooth-int}.  We also note that the Fourier
    coefficients $\fourier{\chi}$ decay much more rapidly for our
    $\chi$ and it therefore requires fewer grid points for the same
    resolution.

  \subsubsection{Performance optimizations} 
  Kernel function evaluations for computing the smooth integrals is
  one of the most expensive parts of our quadrature scheme.  For our
  application, we need to evaluate the potentials from the
  single-layer Helmholtz kernel function ~$\HelmKer{\BeltramiParam}$~
  and its gradient ~$\Grad{\HelmKer{\BeltramiParam}}$.  We have
  optimized kernel evaluations using shared-memory parallelism and
  through the use of \abbrev{AVX} vector instructions for x86
  processors.
  The inverse square root operation in the Laplace and the Helmholtz kernels
  is implemented using the fast approximate inverse square root instruction
  along with Newton iterations for additional accuracy as described in \cite{Malhotra_2015}.
  We use $\ordinal{13}$ order Taylor series approximation
  to evaluate ~$\sin t$~ to 12-digit accuracy in the interval
  ~$[-\pi/4,\pi/4]$ and then evaluate ~$\cos t=\sqrt{1-\sin^2 t}$~
  using a procedure similar to the inverse square root operation.
  To evaluate $\sin$ and
  $\cos$ functions outside the range ~$[-\pi/4, \pi/4]$, we use
  rotations in the complex plane by multiplying with
  ~$\Imag = \sqrt{-1}$.
  Our $\sin$ and $\cos$ evaluation algorithm performed slightly better than Intel's SVML (Short Vector Math Library) in our experiments;
  however, we do not achieve full double-precision accuracy using $\ordinal{13}$ order Taylor series approximation.
  In \pr{t:ker-opt}, we present timing results to show the performance improvement achieved using these optimizations.
    
  \begin{table}[b!] 
    \centering
    \caption{\label{t:ker-opt} Timing results in seconds on a single
      Intel Ivy Bridge CPU core (running at $2.5\ghz$) for $1\pexp9$
      kernel evaluations of the Laplace and the Helmholtz single-layer
      kernel functions compiled using GCC and Intel compilers, and
      comparing the un-vectorized and the vectorized implementations
      in each case.  With GCC, we get $4.5\times$ speedup for the
      Laplace kernel and $6.5\times$ speedup for the Helmholtz kernel.
      The Intel compiler is able to auto-vectorize our un-vectorized
      code and therefore the speedups are not as large as with GCC;
      however, we still achieve over $2\times$ speedup using
      explicit vectorization and our optimized implementation of $\sin$ and
      $\cos$ functions.  }
    \begin{tabular}{l | r r | r r}
      \hline 
                        & \multicolumn{2}{c|}{Laplace kernel} & \multicolumn{2}{c}{Helmholtz kernel}\\
      Compiler          &   Un-vectorized    &   Vectorized   &   Un-vectorized    &   Vectorized   \\
      \hline 
      GCC-8.2           &          $9.21$    &       $2.05$   &         $92.65$    &      $14.12$   \\
      Intel-17.0.2      &          $4.58$    &       $2.12$   &         $28.71$    &      $14.27$   \\
      \hline 
    \end{tabular}
  \end{table} 


\subsection{A spectral Laplace-Beltrami solver\label{ss:lb-algo}} 
  Our spectral Laplace-Beltrami solver is based on previous work in \cite{Imbert_G_rard_2017, O_Neil_2018_LaplaceBeltrami}.
  In \cite{Imbert_G_rard_2017}, the surface Laplacian was implemented using Fourier pseudo-spectral differentiation and preconditioned using the exact inverse for the flat plane.
  The discrete Fourier transforms were computed using \abbrev{FFTPACK} \cite{swarztrauber1985fftpack} and the linear system was solved using BiCGSTAB.
  In our current work, we have improved performance by using the \abbrev{FFTW} library \cite{FFTW3_2005} for computing the discrete Fourier transforms and using \abbrev{GMRES} for the linear solve.
  A layer potential preconditioner for the Laplace-Beltrami equation was presented in \cite{O_Neil_2018_LaplaceBeltrami} for high-order curvilinear triangular meshes.
  We have implemented this preconditioner using the quadratures discussed in \pr{ss:quad-algo}.
  For completeness, we now briefly summarize the Laplace-Beltrami solver.

    For a single toroidal surface $\Boundary$, the surface gradient, divergence and Laplace-Beltrami operators are given by,
    \begin{align} 
      %
      %
      \SurfGrad{f} &= \begin{bmatrix} \SCoord_{\TAngle} & \SCoord_{\PAngle} \end{bmatrix} \linop{\MetricTensor}^{-1} \begin{bmatrix} \pderiv{}{\TAngle} \\ \pderiv{}{\PAngle} \end{bmatrix} f , \label{e:surf-grad} \\
      \SurfDiv{\vector{v}} &= \frac{1}{\sqrt{\detMetricTensor}} \begin{bmatrix} \pderiv{}{\TAngle} & \pderiv{}{\PAngle} \end{bmatrix} \sqrt{\detMetricTensor}~ \vector{v} , \label{e:surf-div} \\
      \SurfLap{f} &= \SurfDiv{\SurfGrad{f}} = \frac{1}{\sqrt{\detMetricTensor}} \begin{bmatrix} \pderiv{}{\TAngle} & \pderiv{}{\PAngle} \end{bmatrix} \sqrt{\detMetricTensor}~ \linop{\MetricTensor}^{-1} \begin{bmatrix} \pderiv{}{\TAngle} \\ \pderiv{}{\PAngle} \end{bmatrix} f , \label{e:surf-lap}
    \end{align} 
    where $\linop{\MetricTensor}$ is the metric tensor, $f$
    is a scalar function on the surface and $\vector{v}$ is a
    tangential vector field defined with respect to the tangent
    vectors $\SCoord_{\TAngle}$ and $\SCoord_{\PAngle}$ such that
    $\vector{v}(\TAngle, \PAngle)=v_{1}(\TAngle,
    \PAngle)~\SCoord_{\TAngle} + v_{2}(\TAngle,
    \PAngle)~\SCoord_{\PAngle}$.
    The discrete surface gradient, divergence and Laplace-Beltrami
    operators are obtained by replacing $\pderiv{}{\TAngle}$ and
    $\pderiv{}{\PAngle}$ with the discrete spectral differentiation
    operators $\fftdiff{}{\TAngle}$ and $\fftdiff{}{\PAngle}$
    respectively in \pr{e:surf-grad,e:surf-div,e:surf-lap},
    \begin{align} 
      \dSurfGrad{\vectord{f}} &= \begin{bmatrix} \vectord{\SCoord}_{\TAngle} & \vectord{\SCoord}_{\PAngle} \end{bmatrix} \linopd{\MetricTensor}^{-1} \begin{bmatrix} \fftdiff{}{\TAngle} \\ \fftdiff{}{\PAngle} \end{bmatrix} \vectord{f} , \label{e:dsurf-grad} \\
      \dSurfDiv{\vectord{v}} &= \frac{1}{\sqrt{\detMetricTensord}} \begin{bmatrix} \fftdiff{}{\TAngle} & \fftdiff{}{\PAngle} \end{bmatrix} \sqrt{\detMetricTensord}~ \vectord{v} , \label{e:dsurf-div} \\
      \dSurfLap{\vectord{f}} &= \dSurfDiv{\dSurfGrad{f}} = \frac{1}{\sqrt{\detMetricTensord}} \begin{bmatrix} \fftdiff{}{\TAngle} & \fftdiff{}{\PAngle} \end{bmatrix} \sqrt{\detMetricTensord}~ \linopd{\MetricTensor}^{-1} \begin{bmatrix} \fftdiff{}{\TAngle} \\ \fftdiff{}{\PAngle} \end{bmatrix} \vectord{f} , \label{e:dsurf-lap}
    \end{align} 
    where $\vectord{f}$, $\vectord{v}$ and $\linopd{\MetricTensor}$
    are the discretizations of $\scalar{f}$,
    $\vector{v}$ and
    $\linop{\MetricTensor}$ respectively.
    The Laplace-Beltrami operator has a nullspace of dimension one
    containing the space of constant functions.  We use a rank-one
    update to make the operator full-rank and therefore uniquely
    invertible.  The modified problem and the corresponding discrete
    problem are then given by:
    \begin{align}
      \SurfLap{\scalar{u}} + \int\limits_{\Boundary} \scalar{u}
      \AreaElem = \scalar{f}  \qquad 
      \text{and} \qquad 
      \dSurfLap{\vectord{u}} + \dotprod{\vectord{w}}{\vectord{u}} = \vectord{f} ,
      \label{e:disc-lap-beltrami}
    \end{align}
    where $\scalar{f}$ is the given RHS, $\scalar{u}$ is the unknown,
    $\vectord{f}$ and $\vectord{u}$ are the discretizations of
    $\scalar{f}$ and $\scalar{u}$, respectively, and
    $\scalard{w}\nm[i][j] = \frac{4 \pi^2}{\Nt \Np}
    \sqrt{\det{\vectord{\MetricTensor}\nm[i][j]}}$ is the trapezoidal
    quadrature weight times the differential area element.  We will
    next discuss two preconditioners for this problem.  The final
    preconditioned problem is then solved using a \abbrev{GMRES} code
    with modified Gram-Schmit orthogonalization from the PETSc
    library~\cite{petsc-web-page,petsc-user-ref}.

  \subsubsection{Spectral preconditioner} 
    \newcommand{\InvFlatSurfLap}[1]{\ensuremath{\Delta^{-1}_{I_h}{#1}}}
    \newcommand{\FourierInvLap}[1]{\ensuremath{\linopd{L}{#1}}}
    In \cite{Imbert_G_rard_2017}, the inverse of the Laplace-Beltrami operator for a flat surface was used to precondition \pr{e:disc-lap-beltrami}.
    The pseudo-spectral inverse surface Laplacian can be constructed as follows,
    \begin{align}
      \InvFlatSurfLap{ ~\vectord{f}} = \IFFT{ \FourierInvLap{ ~\FFT{ ~\vectord{f}} } }
    \end{align}
    where $\FourierInvLap{}$ is a diagonal operator applied to the Fourier coefficients,
    \begin{align}
      \left( \FourierInvLap{ ~\fourier{\vectord{f}} } \right)\nm &=
      \begin{cases}
        0,                                                                                                  &\text{if } m=n=0 \\
        \left( \frac{n^2}{L_{\TAngle}^2} + \frac{m^2}{L_{\PAngle}^2} \right)^{-1} \fourier{\scalard{f}}\nm, &\text{otherwise}.
      \end{cases}
    \end{align}
    The parameters $L_{\TAngle}$ and $L_{\PAngle}$ are the average
    surface lengths in the toroidal and poloidal direction
    respectively.  They account for the anisotropy in the surface
    parameterization.  The final preconditioned system is then
    obtained by left-preconditioning \pr{e:disc-lap-beltrami} with
    $(\InvFlatSurfLap{} + \vectord{1} \Transpose{\vectord{1}})$ as
    follows,
    \begin{align}
      \left( \InvFlatSurfLap{} + \vectord{1} \Transpose{\vectord{1}} \right) ~\left( \dSurfLap{} + \vectord{1} \Transpose{\vectord{w}} \right) \vectord{u} ~=~ \left( \InvFlatSurfLap{} + \vectord{1} \Transpose{\vectord{1}} \right) ~\vectord{f} ,
    \end{align}
    where we have added $\vectord{1} \Transpose{\vectord{1}}$ to $\InvFlatSurfLap{}$ to make it full rank.

  \subsubsection{Layer-potential preconditioner} 
  In \cite{O_Neil_2018_LaplaceBeltrami}, it was shown that preconditioning \pr{e:disc-lap-beltrami} symmetrically with the Laplace single-layer potential operator $\SL{0}$ results in a second-kind Fredholm equation.
    The preconditioned system is then given by:
    \begin{align}
      \SL{0}~ \left( \dSurfLap{} + \vectord{1} \Transpose{\vectord{w}} \right) ~\SL{0}~ \vectord{v} &= \SL{0}~ \vectord{f}, \label{e:layer-potential-precond0} \\
      \vectord{u} &= \SL{0}~ \vectord{v}, \label{e:layer-potential-precond1}
    \end{align}
    where we first solve \pr{e:layer-potential-precond0} for $\vectord{v}$ using \abbrev{GMRES} and then compute the final solution $\vectord{u}$ using \pr{e:layer-potential-precond1}.

  In \pr{ss:results-lb}, we present results comparing the performance of these two preconditioners.
  While the spectral preconditioner is extremely efficient to compute due to the use of \abbrev{FFT}, it can only be applied to uniform surface discretizations.
  On adaptive meshes, the inverse of the 2D Laplacian can be applied using a specialized fast multipole method \cite{Ethridge_2001} or other fast Poisson solvers; however, this is not straightforward.
  In addition, it is sensitive to the surface parameterization.
  On the other hand, the layer-potential preconditioner is not sensitive to the surface parameterization provided the surface quadratures are computed accurately.
  It can also be applied to adaptive meshes, provided appropriate quadratures are available to compute the layer-potentials.
  However, despite the performance optimizations discussed in \pr{ss:quad-algo}, the surface quadratures can still be expensive to compute.
  For our current application, the spectral preconditioner is both fast and sufficiently robust. 

\subsection{Computing harmonic vector fields\label{ss:harmonic-vec-algo}} 

On general smooth closed surfaces, harmonic vector fields are not
known analytically and must be solved for. Only in specialized
geometries, e.g. surfaces of revolution, do formulae
exist~\cite{Cerfon_2014,Epstein_2012}.  Recall that a harmonic vector
field $\vct{m}_H$ is one such that~$\SurfDiv{\vector{m}_H} = 0$
and~\mbox{$\SurfDiv \cross{\Normal}{\vector{m}_H} = 0$}. In our case,
when computing Taylor states using the generalized Debye
representation, we need to compute a basis for harmonic vector fields
subject to an additional constraint, namely that
$\Normal \times \vct{m}_H = -\Imag \vector{m}_H$.  Such harmonic
vector fields~$\vector{m}_H$ can be computed for each toroidal surface
and are unique up to a complex scaling factor.  We follow the scheme
described in~\cite{O_Neil_2018_LaplaceBeltrami} and start with a
smooth tangential vector field $\vector{v}$.  In our implementation,
we choose ~$\vector{v} = \SCoord_{\TAngle}$.  We first compute the non-harmonic terms of its
Hodge decomposition,
  \begin{align*}
    \vector{v} ~=~ \SurfGrad{\alpha} ~+~
    \cross{\Normal}{\SurfGrad{\beta}} ~+~ \vector{v}_H, 
  \end{align*}
  where $\alpha$ and $\beta$ are unknown scalar functions obtained by solving
  the following two Laplace-Beltrami problems:
  \begin{align*}
    \SurfLap{ ~\alpha} ~=~ \SurfDiv{\vector{v}} ,
    && \text{and} &&
    \SurfLap{ ~\beta}  ~=~ -\SurfDiv{\cross{\Normal}{\vector{v}}}.
  \end{align*}
  Then, we can compute~$\vector{v}_H = \vector{v} - \SurfGrad{\alpha} -
  \cross{\Normal}{\SurfGrad{\beta}}$~ and finally obtain the required
  harmonic vector field as
  ~$\vector{m}_H = \vector{v}_H + \Imag
  \cross{\Normal}{\vector{v}_H}$.

\subsection{Algorithm summary}
\label{ss:taylor-algo} 

In this section, we discussed the discretization of the boundary data
on a uniform spectral mesh.  We have presented a quadrature scheme
which is used to compute boundary convolutions with the
kernels~$\HelmKer{\BeltramiParam}(\vector{r})$
and~$\Grad{\HelmKer{\BeltramiParam}}(\vector{r})$.  These schemes are
used to construct the discretized convolution operators
$\SL{\BeltramiParam}$, $\Grad{\SL{\BeltramiParam}}$.  The operator
$\Curl{\SL{\BeltramiParam}}$ can similarly be constructed from the
result of $\Grad{\SL{\BeltramiParam}}$.  We discussed a spectral
Laplace-Beltrami solver along with two efficient preconditioners, and
showed how to compute harmonic vector fields using the
Laplace-Beltrami solver.  The harmonic vector fields are then used to
discretize our boundary integral operator $\conv{K}$ in
\pr{e:taylor-second-kind}.  The resulting system is then solved as
described in \pr{ss:overview-algo}.  This requires $\Nsurf$ solves of
the boundary integral operator using \abbrev{GMRES}.  We summarize the
total cost of different stages of our algorithm for computing Taylor
states in \pr{t:complexity}.  Since our quadrature scheme and the
Laplace-Beltrami solver are both spectrally accurate with mesh
refinement, we expect to see spectral convergence for the overall
solver as well.  In the next section, we present empirical numerical
results to show convergence and solve times for different stages of
our solver.

  \begin{table}[t!] 
    \centering
    \begin{tabular}{lllllr}
      \noalign{\vspace{3pt} \hrule height 1.1pt \vspace{3pt}} 
                                          & $\Tsetup ~~~= $ & $                          $ & $      $ & $\Nunknown   \quadorder^2$ &  \vspace{3pt} \\
      Layer potential quadrature          & $\Tsmooth  ~= $ & $\Nsurf \gmresiter         $ & $\times$ & $\Nunknown^2             $ &  \vspace{3pt} \\
                                          & $\Tsingular = $ & $\Nsurf \gmresiter         $ & $\times$ & $\Nunknown   \patchdim^2 $ &  \vspace{3pt} \\
                                          & $\Tquadeval = $ & \multicolumn{3}{l}{$\Tsmooth ~+~ \Tsingular$}                        &  \\
      \noalign{\vspace{3pt} \hrule height 0.5pt \vspace{3pt}} 
      Spectral Laplace-Beltrami           & $\Tlb ~~~~~~= $ & $\Nsurf \gmresiter \lbiter $ & $\times$ & $\Nunknown \log \Nunknown$ &  \\
      \noalign{\vspace{3pt} \hrule height 1.1pt \vspace{3pt}} 
    \end{tabular}
    \caption[Time complexity of different stages in the
    algorithm.]{\label{t:complexity} Cost complexity for different
      stages of the algorithm for $\Nsurf$ surfaces with $\Nunknown$
      total surface discretization points, $\gmresiter$ \abbrev{GMRES}
      iterations for each Boundary Integral Equation (BIE) solve, and
      $\lbiter$ iterations for each Laplace-Beltrami solve.  $\Tlb$ is
      the total cost for $\Nsurf \gmresiter$ Laplace-Beltrami solves
      using the spectral preconditioner.  The setup cost for the
      quadratures is $\Tsetup$, and the total cost of quadrature
      evaluation $\Tquadeval$ is the sum of cost of the singular
      correction $\Tsingular$ and the cost of the smooth quadrature
      $\Tsmooth$.  }
  \end{table} 


\section{Numerical experiments}\label{s:tests}

We now present numerical results to demonstrate convergence and
efficiency of the quadrature scheme for singular integrands detailed
in \pr{ss:results-quad}, the spectral Laplace-Beltrami solver in
\pr{ss:results-lb}, and the full solver for computing Taylor states
in \pr{ss:results-taylor}.  All numerical errors reported in this
section are relative errors.
Our experiments were performed on a Linux workstation with quad
15-core Intel Xeon E7-4880 v2 CPUs running at $2.5\ghz$ with $1.5\tb$
of RAM.  We use Intel compiler version 17.0.2 with O3 optimization
level and AVX vectorization enabled.

  \begin{figure}[t!] 
    \centering
    \includegraphics[width=0.5\textwidth]{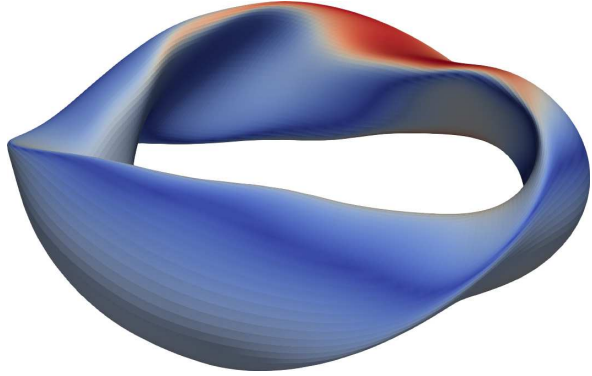}
    \caption{\label{f:conv-sing-quad}
      The surface and the magnitude of the reference potential $u$ for the convergence results presented in \pr{t:quad-conv}. The surface is the outer surface of the QAS3 stellarator~\cite{GarabedianQAS3}.
      The reference potential is generated by a single point source outside the bounded domain $\Domain$.
    }
  \end{figure} 

\subsection{Singular quadrature\label{ss:results-quad}} 
In this section, we present convergence results for our singular
surface quadrature scheme.  We use Green's third identity to
verify and evaluate the accuracy of our method.  For a solution $u$ to
the homogeneous Helmholtz equation in a bounded domain
$\Domain \subset \Real^3$, Green's identity states that,
  \begin{align*}
    u(\vector{x}) &=
                       \int\limits_{\Boundary} \HelmKer{\BeltramiParam}(\vector{x} - \vector{y}) \partial_{\Normal_y} u(\vector{y}) \AreaElem_y
                     - \int\limits_{\Boundary} u(\vector{y}) \partial_{\Normal_y} \HelmKer{\BeltramiParam}(\vector{x} - \vector{y}) \AreaElem_y,
                     & \text{ for } \vector{x} \in \Domain .
  \end{align*}
Taking the limit as $\vct{x} \to \Gamma$, with $\vct{x} \in \Omega$,
along the boundary $\Boundary$ we have
\begin{align}
  u(\vector{x}) &= \frac{u}{2}(\vector{x})
                  + \SL{\BeltramiParam}[\partial_{\Normal} u](\vector{x})
                  - \DL{\BeltramiParam}[u](\vector{x}),
  & \text{ for } \vector{x} \in \Boundary 
    \label{e:greens-rep}
\end{align}
where $\SL{\BeltramiParam}[\cdot]$ and $\DL{\BeltramiParam}[\cdot]$
are the single- and double-layer potential operators on $\Boundary$
which compute convolutions with $\HelmKer{\BeltramiParam}$ and
$\partial_{\Normal} \HelmKer{\BeltramiParam}$ respectively.  We
compute the RHS in \pr{e:greens-rep} using our quadrature scheme and
compare it to the reference potential~$u$ on~$\Boundary$.  The
reference potential $u$ is generated using a single source
$\vector{x_0}$ outside of the domain $\Domain$; \ie
~$u(\vector{x}) = \HelmKer{\BeltramiParam}(\vector{x} -
\vector{x_0})$.

The surface geometry and the magnitude of the reference potential $u$
in our experiments are visualized in \pr{f:conv-sing-quad}.  The
geometry corresponds to the QAS3 stellarator proposed by Paul
Garabedian, whose outer surface can be constructed from the
coefficients given in Table II of reference~\cite{GarabedianQAS3}. We
present convergence results for $\BeltramiParam = 0$ and
$\BeltramiParam = 1$ in \pr{t:quad-conv}.  For both sets of results,
we show convergence to about 8-digits in $L^\infty$-norm as we refine
the mesh from $\Nunknown = 70 \times 14$ unknowns to
$\Nunknown = 700 \times 140$ unknowns, increase the
dimension~$\patchdim$ of the square partition-of-unity patch, and
increase the order $\quadorder$ of the polar quadrature.  In each
case, we have selected the optimal values of the parameters
$\Nunknown$, $\patchdim$ and $\quadorder$ to minimize the computation
time for a given accuracy.  As expected from the error analysis, we
observe spectral convergence in each case.

  In \pr{t:quad-conv}, we also report a breakdown of the total time into the setup time $\Tsetup$ and the evaluation time $\Tquadeval = \Tsingular + \Tsmooth$ on 1-CPU core and 60-CPU cores.
  The setup stage is required only once for a fixed geometry shape;
  therefore, when solving boundary integral equations using iterative linear solvers, the setup cost is amortized.
  As we increase the problem size, the setup time $\Tsetup$ scales as $\bigO{\Nunknown \quadorder^2}$.
  For the evaluation phase $\Tquadeval$, the $\bigO{\Nunknown^2}$ complexity of the N-body computation is the dominant cost
  and the $\bigO{\Nunknown \patchdim^2}$ cost of singular correction accounts for only a small fraction of the evaluation time.
  Comparing the timings for 1-CPU core and 60-CPU cores, we observe that both the setup and evaluation stages scale well for sufficiently large problems,
  with parallel efficiency up to $90\%$.
  Notice that the evaluation time for $\BeltramiParam=0$ is much smaller than for $\BeltramiParam=1$; this is because
  the former uses the Laplace single-layer and double-layer kernel functions, which do not require expensive $\sin$ and $\cos$ function evaluations.


  \begin{table}[b!] 
    \centering
    \begin{tabular}{r c r r r c r | r r r | r r r} 
      \hline 
                        &~~&  $ $         &             &              &~~&                     &               \multicolumn{3}{c|}{1-core} &                \multicolumn{3}{c}{60-cores} \\
      $\BeltramiParam$  &~~&  $\Nunknown$ & $\patchdim$ & $\quadorder$ &~~&  $\norm[\infty]{e}$ &    $\Tsetup$ &  $\Tsingular$ & $\Tsmooth$ &    $\Tsetup$ &  $\Tsingular$ &   $\Tsmooth$ \\
      \hline 
                        &~~&  $9.8\pexp2$ &        $12$ &         $15$ &~~&         $2.0\nexp3$ &      $  0.7$ &       $0.001$ &    $0.006$ &       $0.03$ &       $0.001$ &      $0.003$ \\
                        &~~&  $3.9\pexp3$ &        $24$ &         $24$ &~~&         $3.9\nexp5$ &      $  6.9$ &        $0.02$ &     $0.09$ &       $0.16$ &       $0.002$ &       $0.01$ \\
      $0.0$             &~~&  $1.6\pexp4$ &        $30$ &         $30$ &~~&         $1.2\nexp6$ &      $ 42.9$ &        $0.09$ &     $ 1.3$ &       $ 0.9$ &       $0.008$ &       $0.09$ \\
                        &~~&  $4.8\pexp4$ &        $40$ &         $36$ &~~&         $4.3\nexp8$ &      $190.8$ &         $0.5$ &     $12.2$ &       $ 3.6$ &        $0.02$ &       $0.34$ \\
                        &~~&  $9.8\pexp4$ &        $50$ &         $42$ &~~&         $3.3\nexp9$ &      $543.5$ &         $1.4$ &     $50.5$ &       $10.3$ &        $0.07$ &       $1.14$ \\
     \hline 
                        &~~&  $9.8\pexp2$ &        $12$ &         $15$ &~~&         $2.6\nexp3$ &      $  1.1$ &       $0.004$ &     $0.03$ &       $0.04$ &       $0.002$ &      $0.003$ \\
                        &~~&  $3.9\pexp3$ &        $24$ &         $24$ &~~&         $1.5\nexp4$ &      $ 11.9$ &        $0.04$ &     $ 0.6$ &       $0.26$ &       $0.005$ &       $0.02$ \\
      $1.0$             &~~&  $1.6\pexp4$ &        $30$ &         $30$ &~~&         $4.0\nexp6$ &      $ 74.9$ &         $0.3$ &     $ 8.7$ &       $ 1.4$ &        $0.02$ &       $0.19$ \\
                        &~~&  $4.8\pexp4$ &        $40$ &         $36$ &~~&         $1.1\nexp7$ &      $337.8$ &         $1.3$ &     $80.7$ &       $ 6.1$ &        $0.05$ &       $1.53$ \\
                        &~~&  $9.8\pexp4$ &        $50$ &         $42$ &~~&         $7.2\nexp9$ &      $946.7$ &         $4.0$ &    $334.9$ &       $17.5$ &        $0.17$ &       $6.14$ \\
      \hline 
    \end{tabular} 
    \caption{\label{t:quad-conv}
      Convergence results for the singular quadrature scheme for the Laplace kernel $\BeltramiParam=0$
      and the Helmholtz kernel $\BeltramiParam=1$.
      \pr{f:conv-sing-quad} shows the surface and the reference solution used in these experiments.
      We show convergence in the relative $L^\infty$-norm of the error as we refine the mesh $\Nunknown$,
      increase the dimension of the partition-of-unity patch $\patchdim$
      and increase the order of the polar-quadrature $\quadorder$.
      We also report the setup time $\Tsetup$ and a breakdown of the evaluation time $\Tquadeval = \Tsingular + \Tsmooth$.
      To show scalability of our code, we present timing results on 1-CPU core and on 60-CPU cores.
    }
  \end{table} 

\subsection{Laplace-Beltrami solver\label{ss:results-lb}} 

We now show convergence results for our spectral Laplace-Beltrami
solver.  In the following experiments, we solve
$\SurfLap{\scalar{\phi}} = \scalar{f}$ on $\Boundary$.  We define the
function $\scalar{f}$ in terms of a 3D harmonic potential $\scalar{u}$
as follows,
\begin{align*}
    \scalar{f} = \restr{\SurfLap{\scalar{u}}}{\Boundary}
    = -2\MeanCurv \frac{\partial \scalar{u}}{\partial \Normal} - \frac{\partial^2 \scalar{u}}{\partial \Normal^2}
\end{align*}
where $\MeanCurv$ is the mean curvature function over $\Boundary$.
Then, the exact solution $\scalar{\phi}$ is given by projecting
$\scalar{u}$ to the space of mean-zero functions on $\Boundary$,
\begin{align*}
    \scalar{\phi} = \restr{\scalar{u}}{\Boundary} - \frac{1}{| \Boundary |} \int\limits_{\Boundary} \scalar{u} \AreaElem 
\end{align*}
where $| \Boundary |$ is the surface area of $\Boundary$. In our
experiments, we choose $\scalar{u}$ to be the Coulombic potential from
a single off-surface point charge.  The boundary geometry $\Boundary$
along with reference solution $\scalar{\phi}$ and the RHS $\scalar{f}$
are visualized in \pr{f:laplace-beltrami}.  For this test, we also
took $\Boundary$ to be the outer surface of the QAS3
stellarator~\cite{GarabedianQAS3}.

  \begin{table}[b] 
    \centering
    \begin{tabular}{r | r r r r r | r r r r r} 
      \hline 
       $ $         &                                                    \multicolumn{5}{c|}{spectral preconditioner} &                                    \multicolumn{5}{c}{layer-potential preconditioner} \\
       $\Nunknown$ &  $\gmrestol$ &    $\lbiter$ & $\norm[\infty]{e}$ &               $\Time_{p=1}$ & $\Time_{p=60}$ &  $\gmrestol$ &    $\lbiter$ & $\norm[\infty]{e}$ &     $\Time_{p=1}$ & $\Time_{p=60}$ \\
      \hline 
       $3.9\pexp3$ &  $3.5\nexp3$ &         $ 7$ &        $1.9\nexp2$ &                     $0.010$ &        $0.022$ & $8.8\nexp2$ &         $ 3$ &        $1.4\nexp1$ &            $0.30$ &        $0.018$ \\
       $8.8\pexp3$ &  $5.2\nexp4$ &         $11$ &        $2.8\nexp3$ &                     $0.029$ &        $0.050$ & $1.3\nexp2$ &         $ 6$ &        $7.2\nexp3$ &            $2.59$ &        $0.086$ \\
       $1.6\pexp4$ &  $6.0\nexp5$ &         $16$ &        $4.6\nexp4$ &                      $0.07$ &         $0.10$ & $1.5\nexp3$ &         $ 7$ &        $1.3\nexp3$ &            $8.89$ &         $0.28$ \\
       $3.5\pexp4$ &  $2.1\nexp6$ &         $24$ &        $2.0\nexp5$ &                      $0.26$ &         $0.37$ & $5.1\nexp5$ &         $10$ &        $1.6\nexp5$ &             $ 60$ &         $1.39$ \\
       $6.3\pexp4$ &  $8.5\nexp8$ &         $32$ &        $1.1\nexp6$ &                      $0.63$ &         $0.78$ & $2.1\nexp6$ &         $12$ &        $8.0\nexp7$ &             $222$ &         $4.76$ \\
       $9.8\pexp4$ &  $9.8\nexp9$ &         $37$ &        $9.1\nexp8$ &                      $1.35$ &         $1.42$ & $2.4\nexp7$ &         $13$ &        $6.3\nexp8$ &             $576$ &         $11.9$ \\
      \hline 
    \end{tabular} 
    \caption{\label{t:conv-laplace-beltrami}
      Convergence results for the Laplace-Beltrami solver for the problem shown in \pr{f:laplace-beltrami}, using the spectral preconditioner and the layer-potential preconditioner.
      We show convergence in the relative $L^\infty$ norm of the error $\norm[\infty]{e}$ as we increase the number of discretization grid points $\Nunknown$, and reduce the \abbrev{GMRES} tolerance $\gmrestol$.
      We also report the number of \abbrev{GMRES} iterations $\lbiter$ and the solve time $\Time_{p=1}$ on 1-CPU core and $\Time_{p=60}$ on 60-CPU cores (not including the setup time for quadratures).
    }
  \end{table} 

  In \pr{t:conv-laplace-beltrami}, we present convergence results using
the two preconditioners discussed in \pr{ss:lb-algo}.  In the spectral
preconditioner case, we use left preconditioning with the exact
$\InvSurfLap$ operator for the flat plane.  In the layer-potential
preconditioner case, we use symmetric preconditioning with the
single-layer potential potential operator associated with the
Laplace's equation.  Without preconditioning, the solver did not
converge to the desired accuracy within 300 \abbrev{gmres} iterations
and therefore we have omitted those results.  The results show
convergence in the relative $L^\infty$-norm of the error as we refine
the mesh from $\Nunknown = 70 \times 14$ unknowns to
$\Nunknown = 700 \times 140$ unknowns and reduce the \abbrev{GMRES}
tolerance $\gmrestol$.  The relative error $\norm[\infty]{e}$
(compared to the reference solution) converges spectrally to about
8-digits for both preconditioners.  The number of \abbrev{GMRES}
iterations $\lbiter$ grows only modestly with the desired precision.
We also report the solve time $\Time_{p=1}$ on 1-CPU core and
$\Time_{p=60}$ on 60-CPU cores.  From the solve times $\Time_{p=1}$,
we observe that the solves using the spectral preconditioner scale as
expected with a cost of $\bigO{\lbiter \Nunknown \log{\Nunknown}}$.
Similarly, the solves with the layer-potential preconditioner scale as
$\bigO{\lbiter \Nunknown^2}$.  This cost can be reduced to
$\bigO{\lbiter \Nunknown}$ by accelerating using the fast multipole
method (FMM); however, depending on the specific FMM implementation
and desired precision, this would only be advantageous for large
problems ($\Nunknown > 20K$).  Comparing $\Time_{p=1}$ and
$\Time_{p=60}$, we observe that for the spectral preconditioner, the
solve time does not improve from using multiple CPU-cores since the
\abbrev{FFT} is a memory bound operation.  For the layer-potential
preconditioner, we observe up to $48\times$ speedup for the largest
problem on 60-CPU cores since the quadrature evaluation is compute
bound.  Comparing the two preconditioners, the layer-potential
potential is more effective in reducing the number of \abbrev{GMRES}
iterations $\lbiter$; however, it is still much more expensive due to
the high cost of the quadratures.  Since the layer-potential
preconditioner has better parallel scalability, in parallel the
timings for the two schemes become comparable.

  \begin{figure}[t] 
    \centering
    \includegraphics[width=0.49\textwidth]{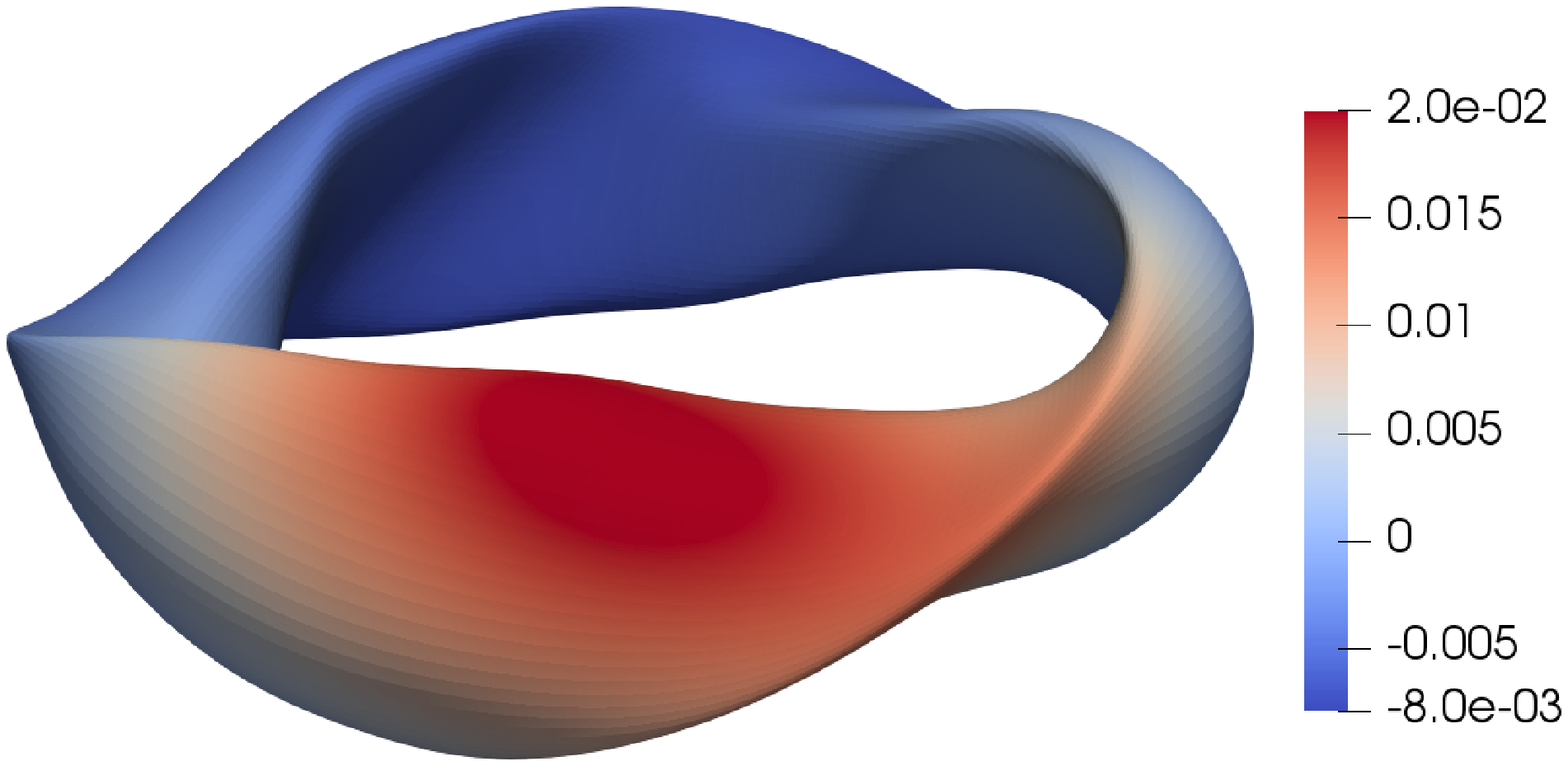}
    \includegraphics[width=0.49\textwidth]{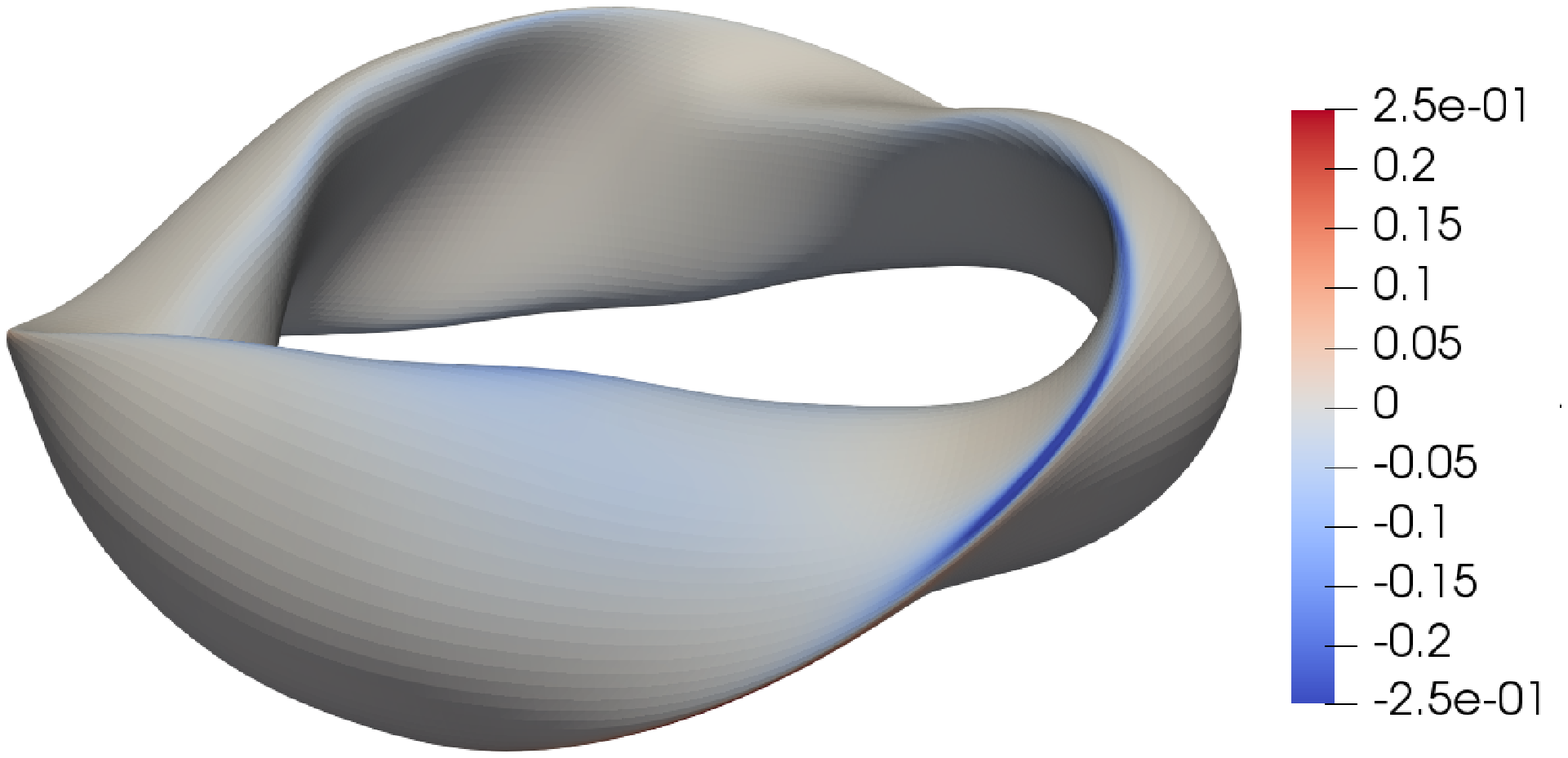}
    \caption{\label{f:laplace-beltrami}
      The surface geometry with the reference solution $\scalar{\phi}$ (left) and the input $\scalar{f}$ (right) to the Laplace-Beltrami solver for the convergence results in \pr{t:conv-laplace-beltrami}. The surface is the outer surface of the QAS3 stellarator~\cite{GarabedianQAS3}.
    }
  \end{figure} 


\subsection{Computing Taylor states\label{ss:results-taylor}} 

We now present numerical results to show convergence of our boundary
integral solver for Taylor states.  In \cite{Cerfon_2014}, a method
for constructing analytic solutions for Taylor states in axisymmetric
geometries was discussed and we used this to construct reference
solutions for testing our numerical scheme in
\cite{O_Neil_2018_Taylor}.  We are not aware of a general method for
generating analytic solutions for non-axisymmetric geometries.
Therefore, to evaluate the accuracy of our solver, we instead solve a
related boundary value problem discussed below.

We construct a reference Taylor state $\vector{B_0}$ given by,
  \begin{align}
    \begin{aligned}
      \vector{B_0}             &= \BeltramiParam \vector{Q_0} + \Curl{\vector{Q_0}}, \\
      \vector{Q_0}(\vector{x}) &= \oint\limits_{\CurrentLoop} \HelmKer{\BeltramiParam}(\vector{x},\vector{y})~ m_0 \VecLengthElem(\vector{y}),
    \end{aligned}
    \label{e:taylor-reference}
  \end{align}
  where $m_0$ is a constant generalized Debye current in a loop $\CurrentLoop$
  around the domain $\Domain$.  The generalized Debye current loop is shown in red in
  \pr{f:taylor-result}.  By construction, $\vector{B_0}$ satisfies
  $\Curl{\vector{B_0}} = \BeltramiParam \vector{B_0}$; however, in
  general, it does not satisfy the boundary condition
  $\dotprod{\vector{B_0}}{\Normal} = 0$ on $\Boundary$.  Therefore,
  instead of solving the original problem in \pr{e:taylor-state}, we
  solve the following more general problem where
  $\dotprod{\vector{B}}{\Normal}$ on the boundary is non-zero,
  \begin{equation}
    \begin{aligned}
      \Curl{\vector{B}} &= \BeltramiParam \vector{B},
      &\qquad &\text{in \Domain}, \\
      \dotprod{\vector{B}}{\Normal} &= \dotprod{\vector{B_0}}{\Normal},
      & &\text{on \Boundary},
      \label{e:taylor-state-modified}
    \end{aligned}
  \end{equation}
  while also matching the appropriate flux conditions (or equivalently
  the circulation as discussed in \pr{ss:formulation}) for
  $\vector{B}$ and $\vector{B_0}$.  Due to uniqueness, the numerically
  computed solution $\vector{B}$ should match the reference solution
  $\vector{B_0}$ at all points in the interior of the domain
  $\Domain$.  Below, we test our solver for the two domains shown in
  \pr{f:taylor-result}.

  \begin{figure}[t!] 
  \centering
  \includegraphics[width=6.8cm]{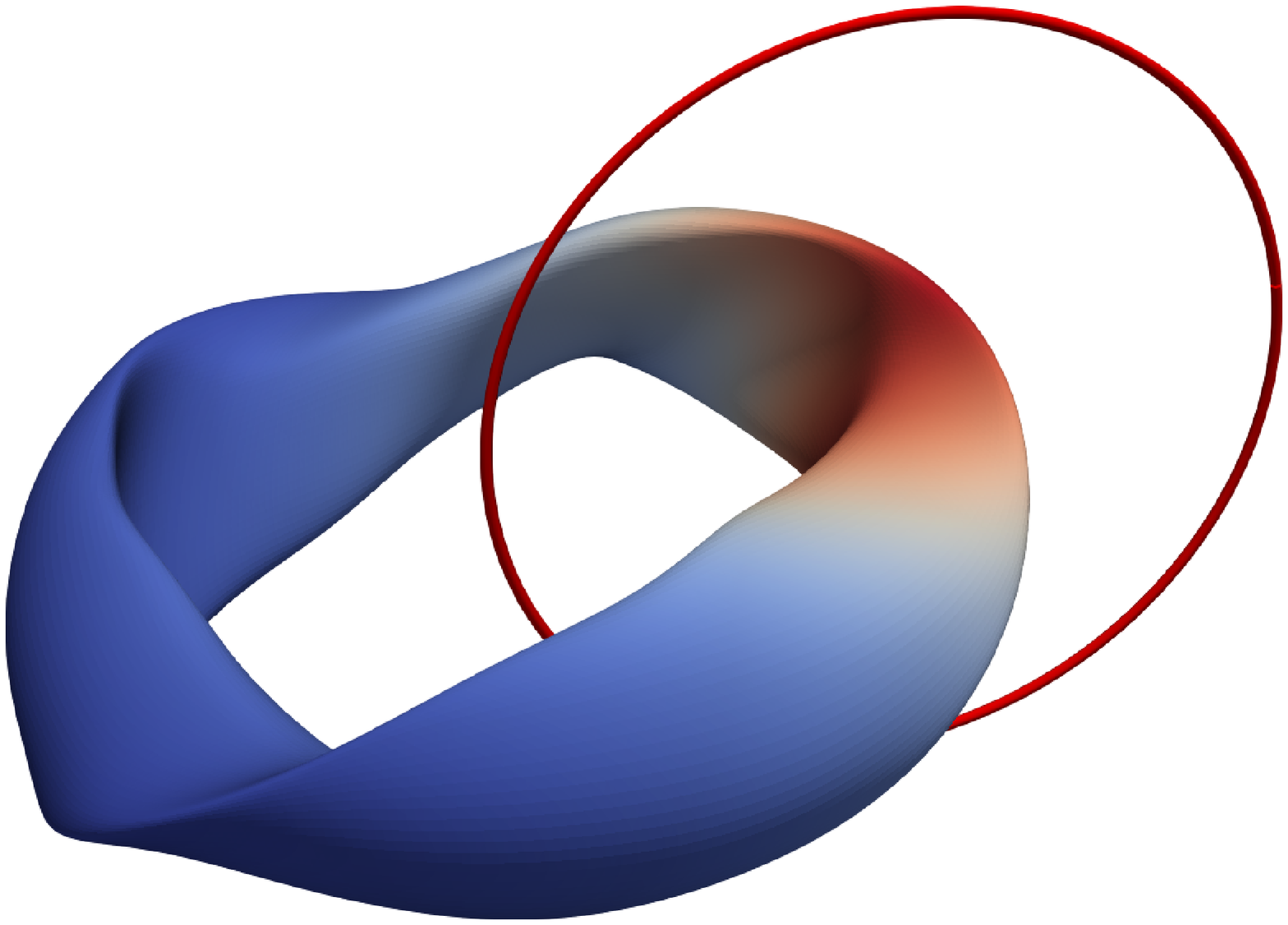}
  \includegraphics[width=6.8cm]{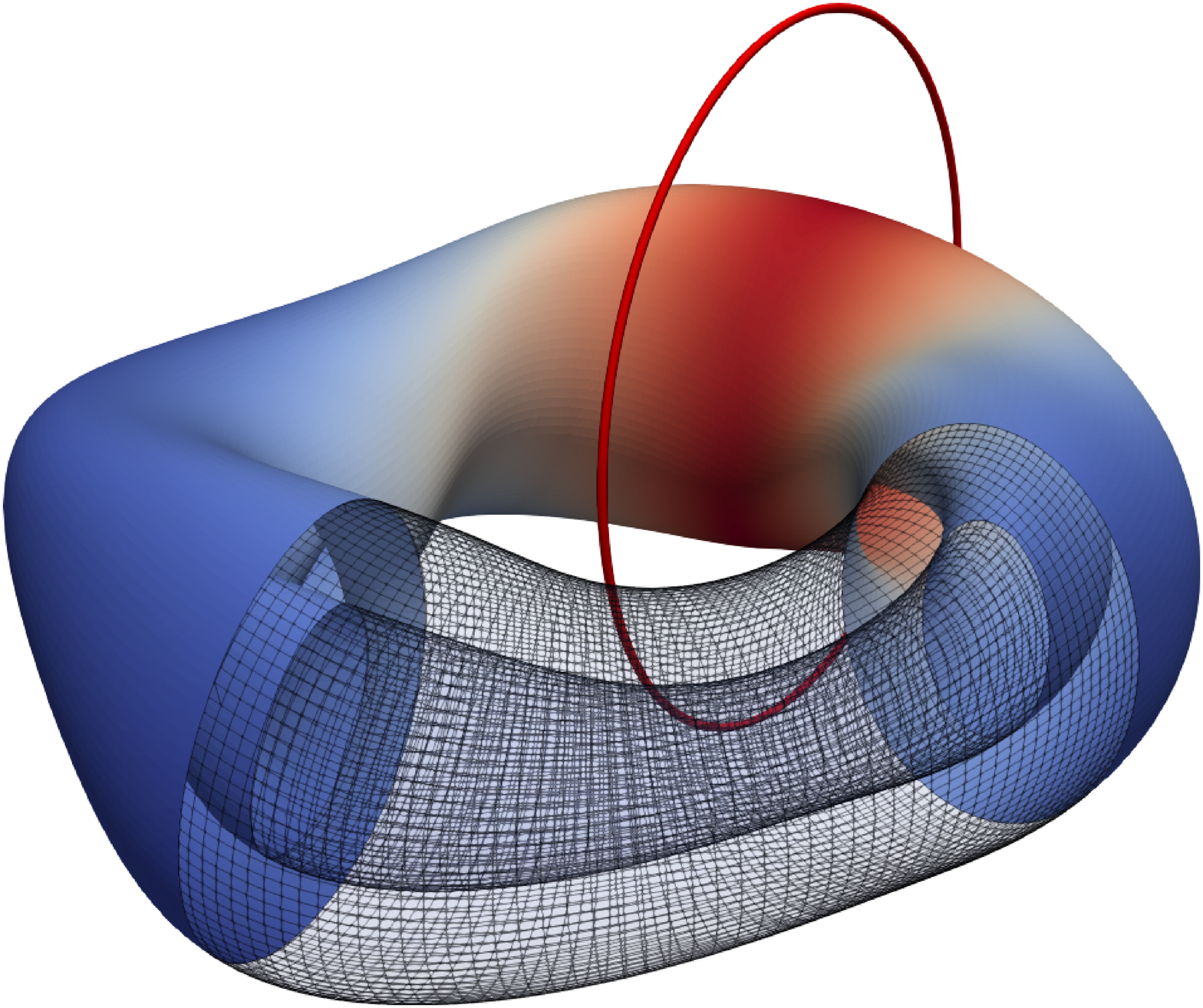}
  \caption{\label{f:taylor-result} Taylor states in toroidal (left)
    and toroidal-shell (right) domains.
    The first shape is the outer surface of the QAS3 stellarator~\cite{GarabedianQAS3}.
    In the second geometry, each surface is the surface of revolution a rotating ellipse as described in \pr{sss:shell-conv}.
    The reference solution
    $\vector{B_0}$ visualized here is generated by a generalized Debye current loop
    (shown in red) around the domain.  Convergence results for these
    geometries are presented in \pr{t:conv-taylor0,t:conv-taylor1}.}
\end{figure} 

  \subsubsection{Toroidal domain} 
  The domain along with the magnitude of the reference solution
  $\vector{B_0}$ are visualized on the left in \pr{f:taylor-result}. Once more, the surface is the outer surface of the QAS3 stellarator~\cite{GarabedianQAS3}.
  We discretize the current loop (shown in red) using $1\pexp3$
  equispaced points and compute the line integral in
  \pr{e:taylor-reference} using trapezoidal quadrature rule.  Since
  the domain $\Domain$ and the current loop $\CurrentLoop$ are
  sufficiently separated, the reference solution $\vector{B_0}$
  computed in this way is accurate to 16-digits at points on
  $\Domain$.  We evaluate $\dotprod{\vector{B_0}}{\Normal}$ on
  $\Boundary$ and the toroidal flux $\Flux^{tor}_{\vector{B_0}}$ in a
  cross section of $\Domain$.  Then, we setup the following
  discretized linear system,
    \begin{align*}
      \begin{bmatrix}
        \linop{A}      & \vector{u_1} \\
        \vector{v_1}^T &       c_{11} \\
      \end{bmatrix}
      \begin{bmatrix}
        \vector{\sigma} \\
        \alpha \\
      \end{bmatrix}
      =
      \begin{bmatrix}
        \dotprod{\vector{B_0}}{\Normal} \\
        \Flux^{tor}_{\vector{B_0}} \\
      \end{bmatrix}
      .
    \end{align*}
    We solve this block linear system with a procedure similar to the one described before in \prange{e:block-solve0}{e:block-solve2}, except that the $\dotprod{\vector{B_0}}{\Normal}$ is non-zero.
    The unknowns $\vector{\sigma}$ and $\alpha$ are given by
    \begin{align*}
      \alpha &= (\vector{v_1}^T \vector{d} + c)^{-1} (\Flux^{tor}_{\vector{B_0}} - \vector{v_1}^T \vector{w}) \\
      \vector{\sigma} &= \vector{w} - \vector{d} \alpha
    \end{align*}
    where $\vector{w} = \linop{A}^{-1} \dotprod{\vector{B_0}}{\Normal}$ and $\vector{d} = \linop{A}^{-1} \vector{u_1}$.
    The operator $\linop{A}^{-1}$ must be applied twice and this is accomplished using two \abbrev{GMRES} solves.
    Next, using $\vector{\sigma}$ and $\alpha$, we evaluate $\vector{B}$ at points in the interior of $\Domain$ and compare with the reference solution $\vector{B_0}$.

    In \pr{t:conv-taylor0}, we report the maximum relative error
    $\norm[\infty]{e} = {\norm[\infty]{\vector{B}-\vector{B_0}}} /
    {\norm[\infty]{\vector{B_0}}}$ in $\Domain$ for two sets of
    experiments with $\BeltramiParam = 0.5$ and $\BeltramiParam = 1$
    respectively.  In each case, we show convergence as we reduce the
    \abbrev{GMRES} tolerance $\gmrestol$, increase the mesh refinement
    $\Nunknown$, and correspondingly increase the quadrature accuracy.
    In both sets of results, we observe spectral convergence to about
    5-digits of accuracy
      We also report the total number of \abbrev{GMRES}
    iterations $(\Nsurf+1)~\gmresiter$ and the total solve time
    $\Tsolve$ on 60-CPU cores.  We observe that problems with larger
    $\BeltramiParam$ require a larger number of \abbrev{GMRES}
    iterations and proportionally longer solve times.  This is well-known
    in wave propagation problems, and can be attributed to the
    spectrum of the boundary integral operator $\linop{A}$ slightly
    de-clustering as $\BeltramiParam$ increases.

    
    \begin{table}[t] 
      \centering
      \begin{tabular}{r r r | r r | r r r r} 
        \hline 
        $\BeltramiParam$ & $\Nunknown$ & $\gmrestol$ &  $(\Nsurf+1)~\gmresiter$ &  $\norm[\infty]{e}$ &  $\Tsetup$ &     $\Tlb$ & $\Tquadeval$ &  $\Tsolve$ \\
        \hline 
                         & $8.8\pexp3$ & $9.6\nexp3$ &                     $24$ &         $2.6\nexp2$ &     $ 1.6$ &    $  8.2$ &      $  4.4$ &   $  14.3$ \\
                         & $2.5\pexp4$ & $6.5\nexp4$ &                     $33$ &         $2.2\nexp3$ &     $ 7.0$ &    $ 39.4$ &      $ 36.9$ &   $  83.2$ \\
                         & $4.8\pexp4$ & $1.5\nexp4$ &                     $38$ &         $3.8\nexp4$ &     $15.8$ &    $101.4$ &      $143.9$ &   $ 261.0$ \\
        $0.5$            & $7.9\pexp4$ & $1.5\nexp5$ &                     $43$ &         $4.9\nexp5$ &     $31.7$ &    $256.9$ &      $434.2$ &   $ 722.7$ \\
                         & $9.8\pexp4$ & $3.6\nexp6$ &                     $46$ &         $2.2\nexp5$ &     $44.7$ &    $342.9$ &      $693.8$ &   $1081.4$ \\
                         & $2.2\pexp5$ & $6.8\nexp8$ &                     $56$ &         $5.3\nexp7$ &    $148.1$ &   $1314.8$ &     $4074.9$ &   $5537.8$ \\
                         & $3.9\pexp5$ & $4.0\nexp9$ &                     $61$ &         $3.1\nexp8$ &    $366.2$ &   $2803.7$ &    $13718.0$ &  $16889.0$ \\
                         & $6.1\pexp5$ & $1.0\nexp9$ &                     $66$ &         $4.4\nexp9$ &    $704.9$ &   $6242.9$ &    $35729.0$ &  $42676.5$ \\
        \hline 
                         & $8.8\pexp3$ & $9.6\nexp3$ &                     $34$ &         $3.3\nexp2$ &     $ 1.6$ &    $ 11.0$ &      $  5.9$ &   $  18.5$ \\
                         & $2.5\pexp4$ & $6.5\nexp4$ &                     $45$ &         $2.5\nexp3$ &     $ 6.8$ &    $ 53.2$ &      $ 48.9$ &   $ 109.0$ \\
                         & $4.8\pexp4$ & $1.5\nexp4$ &                     $50$ &         $4.3\nexp4$ &     $15.7$ &    $135.5$ &      $184.5$ &   $ 335.7$ \\
        $1.0$            & $7.9\pexp4$ & $1.5\nexp5$ &                     $58$ &         $6.2\nexp5$ &     $31.6$ &    $355.4$ &      $568.3$ &   $ 955.3$ \\
                         & $9.8\pexp4$ & $3.6\nexp6$ &                     $63$ &         $2.3\nexp5$ &     $44.1$ &    $478.7$ &      $921.3$ &   $1444.0$ \\
                         & $2.2\pexp5$ & $6.8\nexp8$ &                     $78$ &         $5.5\nexp7$ &    $148.1$ &   $1861.1$ &     $5545.8$ &   $7555.0$ \\
                         & $3.9\pexp5$ & $4.0\nexp9$ &                     $85$ &         $2.9\nexp8$ &    $366.2$ &   $3894.2$ &    $18597.0$ &  $22857.4$ \\
                         & $6.1\pexp5$ & $1.0\nexp9$ &                     $88$ &         $6.5\nexp9$ &    $704.9$ &   $8381.8$ &    $46458.0$ &  $55544.7$ \\
        \hline 

      \end{tabular} 
      \caption{\label{t:conv-taylor0}
        Convergence results for the geometry and reference solution in \pr{f:taylor-result} (left).
        We present results for two different values of the Beltrami parameter $\BeltramiParam=0.5$ and $\BeltramiParam=1$.
        We observe spectral convergence in the error $\norm[\infty]{e}$ with increasing mesh-refinement $\Nunknown$ as we also reduce the \abbrev{GMRES} tolerance $\gmrestol$ and
        correspondingly increase the quadrature accuracy (using
        parameters $\patchdim$ and $\quadorder$).
        We also report the total number of \abbrev{GMRES} iterations $(\Nsurf+1)~\gmresiter$ and the breakdown of the total solve time $\Tsolve \approx \Tsetup + \Tlb + \Tquadeval$ on 60-CPU cores.
      }
    \end{table} 

    \begin{table}[b] 
      \centering
      \caption{\label{t:conv-taylor1} Convergence results for our
        Taylor state solver on a toroidal-shell geometry visualized in
        \pr{f:taylor-result} (right).  We show convergence in
        $\norm[\infty]{e}$ as we decrease the \abbrev{GMRES} tolerance
        $\gmrestol$, increase the mesh refinement $\Nunknown$ and
        correspondingly increase the quadrature accuracy (using
        parameters $\patchdim$ and $\quadorder$).  We also report the
        number of \abbrev{GMRES} iterations $(\Nsurf+1)~\gmresiter$
        and the breakdown of the total solve time $\Tsolve \approx \Tsetup + \Tlb + \Tquadeval$ on 60-CPU cores.  }
      \begin{tabular}{r r c r r | r r r r} 
        \hline 
         $\Nunknown$ &  $\gmrestol$ &~~& $(\Nsurf+1)~\gmresiter$ &  $\norm[\infty]{e}$ &    $\Tsetup$ &     $\Tlb$ & $\Tquadeval$ &  $\Tsolve$ \\
        \hline 
         $2.0\pexp3$ &  $2.0\nexp3$ &~~&                   $ 78$ &         $7.3\nexp3$ &       $ 0.1$ &    $  6.8$ &      $  0.9$ &    $  7.9$ \\
         $8.1\pexp3$ &  $1.3\nexp5$ &~~&                   $104$ &         $5.1\nexp5$ &       $ 1.2$ &    $ 33.7$ &      $ 10.6$ &    $ 45.6$ \\
         $1.8\pexp4$ &  $3.0\nexp7$ &~~&                   $119$ &         $1.4\nexp6$ &       $ 5.9$ &    $111.8$ &      $ 53.7$ &    $171.4$ \\
         $3.2\pexp4$ &  $3.2\nexp8$ &~~&                   $129$ &         $1.1\nexp7$ &       $13.7$ &    $239.7$ &      $173.2$ &    $426.6$ \\
         $5.1\pexp4$ &  $2.6\nexp9$ &~~&                   $139$ &         $1.4\nexp8$ &       $26.4$ &    $552.6$ &      $419.7$ &    $998.8$ \\
        \hline 

      \end{tabular} 
    \end{table} 


    \subsubsection{Toroidal-shell domain\label{sss:shell-conv}} 
    We now present convergence results for the geometry on the right in \pr{f:taylor-result}.
    The domain $\Domain$ is the region between the two toroidal surfaces, each of which is the surface of revolution of a rotating ellipse described by the following equation in cylindrical coordinates,
    \begin{equation}
      \begin{bmatrix}
        R(\theta,\phi) \\
        Z(\theta, \phi) \\
      \end{bmatrix}
      =
      \begin{bmatrix}
        R_0 \\
        0
      \end{bmatrix}
      +
      \begin{bmatrix}
        \cos{\frac{3}{2}\theta} & -\sin{\frac{3}{2}\theta} \\
        \sin{\frac{3}{2}\theta} & \cos{\frac{3}{2}\theta} \\
      \end{bmatrix}
      \begin{bmatrix}
        a & 0 \\
        0 & b
      \end{bmatrix}
      \begin{bmatrix}
        \cos{\frac{3}{2}\theta} & \sin{\frac{3}{2}\theta} \\
        -\sin{\frac{3}{2}\theta} & \cos{\frac{3}{2}\theta} \\
      \end{bmatrix}
      \begin{bmatrix}
        \cos{\phi} \\
        \sin{\phi}
      \end{bmatrix}.
      \label{e:rotating-ellipse}
    \end{equation}
    The outer surface is described by \pr{e:rotating-ellipse} with $R_0=2.0$, $a=0.7$ and $b=1.0$.
    Similarly, the inner surface is described by \pr{e:rotating-ellipse} with $R_0=2.0$, $a=0.3$ and $b=0.55$.
    As in the previous case, the reference solution $\vector{B_0}$ is generated by a generalized Debye current loop around the $\Domain$ and is evaluated using \pr{e:taylor-reference} and trapezoidal quadrature rule.
    We evaluate $\dotprod{\vector{B_0}}{\Normal}$ on the domain boundary $\Boundary$.
    In this setup, we require two flux constraints;
    the toroidal flux $\Flux^{tor}_{\vector{B_0}}$ and
    the poloidal flux $\Flux^{pol}_{\vector{B_0}}$.
    We setup the discretized boundary integral problem for the solution $\vector{B}$ in the following block form,
    \begin{align*}
      \begin{bmatrix}
        \linop{A}      & \vector{u_1} & \vector{u_2} \\
        \vector{v_1}^T &       c_{11} &       c_{12} \\
        \vector{v_2}^T &       c_{21} &       c_{22} \\
      \end{bmatrix}
      \begin{bmatrix}
        \vector{\sigma} \\
        \alpha \\
        \beta \\
      \end{bmatrix}
      =
      \begin{bmatrix}
        \dotprod{\vector{B_0}}{\Normal} \\
        \Flux^{tor}_{\vector{B_0}} \\
        \Flux^{pol}_{\vector{B_0}} \\
      \end{bmatrix}
      .
    \end{align*}
    We then solve this linear system for the unknowns $\vector{\sigma}$ and $\phi$ as follows,
    \begin{align*}
      \begin{bmatrix}
        \alpha \\
        \beta \\
      \end{bmatrix}
           &= \left(
                    \begin{bmatrix}
                      \vector{v_1}^T \\
                      \vector{v_2}^T \\
                    \end{bmatrix}
                    \begin{bmatrix}
                      \vector{d_1} & \vector{d_2}
                    \end{bmatrix}
                    +
                    \begin{bmatrix}
                      c_{11} & c_{12} \\
                      c_{21} & c_{22} \\
                    \end{bmatrix}
              \right)^{-1}
              \left(
                    \begin{bmatrix}
                      \Flux^{tor}_{\vector{B_0}} \\
                      \Flux^{pol}_{\vector{B_0}} \\
                    \end{bmatrix}
                    -
                    \begin{bmatrix}
                      \vector{v_1}^T \\
                      \vector{v_2}^T \\
                    \end{bmatrix}
                    \vector{w}
              \right) \\
      \vector{\sigma} &= \vector{w}
                -
                \begin{bmatrix}
                  \vector{d_1} & \vector{d_2}
                \end{bmatrix}
                \begin{bmatrix}
                  \alpha \\
                  \beta \\
                \end{bmatrix}
    \end{align*}
    where
    $\vector{w} = \linop{A}^{-1} \dotprod{\vector{B_0}}{\Normal}$ and
    $\begin{bmatrix} \vector{d_1} & \vector{d_2} \end{bmatrix} =
    \linop{A}^{-1} \begin{bmatrix} \vector{u_1} &
      \vector{u_2} \end{bmatrix}$.  The numerical results in
    \pr{t:conv-taylor1} show convergence in $\norm[\infty]{e}$ to
    about 8-digits as we decrease the \abbrev{GMRES} tolerance
    $\gmrestol$ and increase the mesh refinement $\Nunknown$.  We also
    report the total number of \abbrev{GMRES} iterations
    $(\Nsurf+1)~\gmresiter$ and the total solve time $\Tsolve$ on
    60-CPU cores.



\subsection{Comparison with SPEC} 
In \pr{t:conv-spec}, we compare our solver with the \abbrev{SPEC} code
of \cite{Loizu2016Verification}.  We present results for the geometry
corresponding to the plasma boundary in the Wendelstein 7-X (W7-X)
stellarator~\cite{Beidler1990,Pedersen2017}.  We constructed reference
solutions for $\BeltramiParam = 0$ and $\BeltramiParam=1$ with our BIE
code using very fine discretization ($\Nunknown=8.2\pexp5$,
$\gmrestol=1\nexp12$).  The two reference solutions are visualized in
\pr{f:conv-spec}.  We show convergence for both codes as we increase
mesh refinement and report the relative $L^{\infty}$-norm of the error
compared to the reference solutions on the domain boundary.  We also
compare the total solve time for both codes on 60-CPU cores.  For the
same solution accuracy our method is significantly faster than the
\abbrev{SPEC} code and at least an order of magnitude faster for the
vacuum field case ($\BeltramiParam=0$). This performance is all the
more noteworthy as our solver presently does not take advantage of the
5-fold rotational symmetry about the vertical axis of W7-X, unlike the
SPEC code. Note that we have not been able to get results for higher
refinement in the SPEC code because the system became
ill-conditioned. This is an advantage of our integral formulation,
which leads to linear systems which are as well-conditioned as the
underlying physical problem.


  \begin{table}[h] 
    \centering
    \caption{\label{t:conv-spec}
      Convergence results for our BIE code and the SPEC code on the W7-X geometry~\cite{Beidler1990,Pedersen2017}.
      We show convergence in the relative $L^{\infty}$-norm of the error by comparing the solutions with a reference solution.
      We also report the total solve time $\Tsolve$ for both codes on 60-CPU cores.
      The parameters $M_{pol}$, $M_{tor}$ and $L_{rad}$ determine the refinement in the poloidal, toroidal and the radial directions in SPEC.
      For $\BeltramiParam = 0$, we use the scheme for vacuum fields described in \pr{sss:vacuum-formulation}.
      This has the advantage that it does not require Laplace-Beltrami solves and only requires computing surface convolutions with the Laplace gradient kernel, which is significantly less expensive than computing convolutions with the Helmholtz kernel. We have not been able to obtain results for higher refinement with the SPEC code, because the system became ill-conditioned.
    }
    \begin{tabular}{r | r r r r | r r r} 
      \hline 
      $ $               &                                   \multicolumn{4}{c|}{BIE method} &                                   \multicolumn{3}{c}{\abbrev{SPEC} code} \\
      $\BeltramiParam$ & $\Nunknown$ &  $\gmrestol$ &  $\norm[\infty]{e}$ &    $\Tsolve$  & $M_{pol} \times M_{tor} \times L_{rad}$ & $\norm[\infty]{e}$ & $\Tsolve$ \\
      \hline 
                 $   $ & $8.8\pexp3$ &   $3\nexp02$ &         $2.3\nexp1$ &     $   0.4$  &                 $11 \times 11 \times 5$ &       $3.3\nexp1$ &       $14$ \\
                 $   $ & $2.5\pexp4$ &   $3\nexp03$ &         $3.1\nexp2$ &     $   2.1$  &                 $13 \times 13 \times 5$ &       $6.0\nexp2$ &       $38$ \\
                 $   $ & $3.5\pexp4$ &   $1\nexp03$ &         $6.6\nexp3$ &     $   4.3$  &                 $15 \times 15 \times 5$ &       $7.9\nexp3$ &      $115$ \\
                 $   $ & $6.3\pexp4$ &   $3\nexp04$ &         $1.6\nexp3$ &     $  11.3$  &                 $21 \times 21 \times 5$ &       $1.6\nexp3$ &      $527$ \\
                 $0.0$ & $7.9\pexp4$ &   $1\nexp04$ &         $5.9\nexp4$ &     $  19.5$  &                                       - &                 - &          - \\
                 $   $ & $1.9\pexp5$ &   $1\nexp06$ &         $9.4\nexp6$ &     $  98.1$  &                                       - &                 - &          - \\
                 $   $ & $3.5\pexp5$ &   $3\nexp08$ &         $6.7\nexp7$ &     $ 302.3$  &                                       - &                 - &          - \\
                 $   $ & $4.7\pexp5$ &   $1\nexp08$ &         $9.0\nexp8$ &     $ 491.8$  &                                       - &                 - &          - \\
                 $   $ & $8.2\pexp5$ &   $1\nexp10$ &         $2.0\nexp9$ &     $1651.3$  &                                       - &                 - &          - \\
      \hline 
                 $   $ & $3.9\pexp3$ &   $1\nexp01$ &         $2.2\nexp1$ &    $    1.0$  &                 $11 \times 11 \times 5$ &       $3.3\nexp1$ &       $13$ \\
                 $   $ & $1.6\pexp4$ &   $1\nexp02$ &         $2.6\nexp2$ &    $   12.2$  &                 $13 \times 13 \times 5$ &       $6.1\nexp2$ &       $38$ \\
                 $   $ & $2.5\pexp4$ &   $3\nexp03$ &         $7.4\nexp3$ &    $   29.8$  &                 $15 \times 15 \times 5$ &       $9.1\nexp3$ &      $118$ \\
                 $   $ & $3.5\pexp4$ &   $1\nexp03$ &         $2.7\nexp3$ &    $   61.6$  &                 $19 \times 19 \times 5$ &       $3.3\nexp3$ &      $389$ \\
                 $   $ & $4.8\pexp4$ &   $3\nexp04$ &         $1.2\nexp3$ &    $  117.1$  &                 $21 \times 21 \times 5$ &       $2.0\nexp3$ &      $541$ \\
                 $1.0$ & $7.9\pexp4$ &   $1\nexp04$ &         $2.3\nexp4$ &    $  317.9$  &                 $23 \times 23 \times 5$ &       $8.5\nexp4$ &      $879$ \\
                 $   $ & $1.9\pexp5$ &   $1\nexp06$ &         $8.6\nexp6$ &    $ 1991.0$  &                                       - &                 - &          - \\
                 $   $ & $2.5\pexp5$ &   $1\nexp07$ &         $9.0\nexp7$ &    $ 3479.6$  &                                       - &                 - &          - \\
                 $   $ & $5.2\pexp5$ &   $1\nexp09$ &         $3.0\nexp8$ &    $16763.3$  &                                       - &                 - &          - \\
                 $   $ & $7.7\pexp5$ &   $1\nexp10$ &         $1.7\nexp9$ &    $40646.2$  &                                       - &                 - &          - \\
      \hline 
    \end{tabular} 
  \end{table} 
  

  \begin{figure}[t] 
    \centering
    \includegraphics[width=0.49\textwidth]{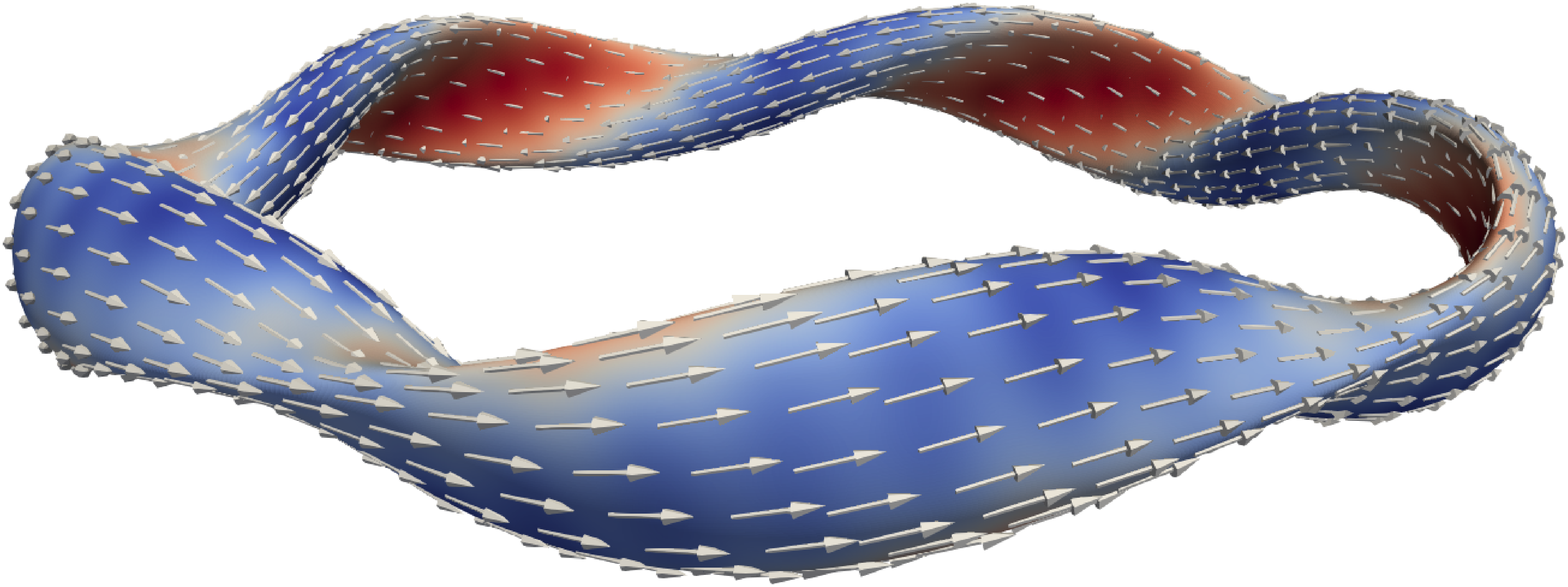}
    \includegraphics[width=0.49\textwidth]{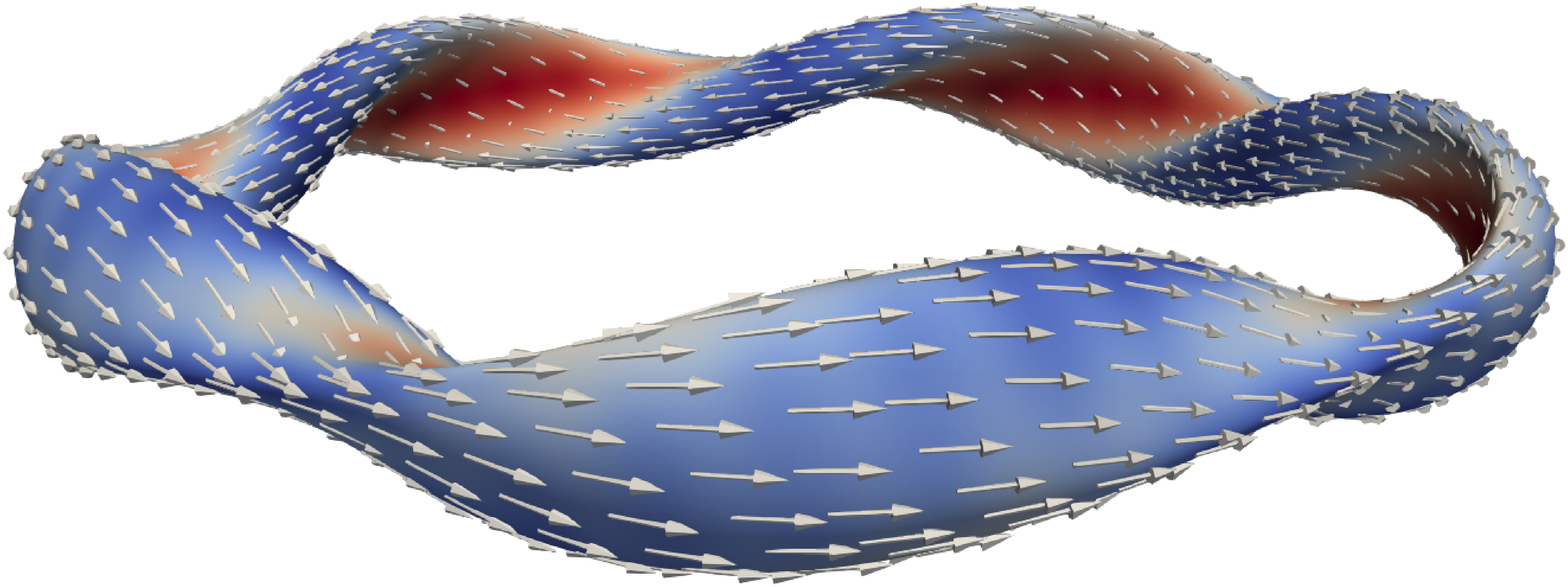}
    \caption{\label{f:conv-spec}
      Reference solutions on the W7X geometry~\cite{Beidler1990,Pedersen2017} for the convergence results in \pr{t:conv-spec}.
      The fields correspond to $\BeltramiParam=0$ and $\BeltramiParam=1$ respectively.
    }
  \end{figure} 

\section{Summary}\label{s:summary}
In this work, we have developed a solver for computing Taylor relaxed
states in toroidal geometries.  This extends our previous work in
\cite{O_Neil_2018_Taylor} on computing such fields in axisymmetric
geometries to non-axisymmetric geometries, and can be applied to the
computation of stepped-pressure equilibria in
stellarators~\cite{Hudson2012}.  The representation of the field using
generalized Debye sources~\cite{Epstein_2012} remains the same as in
the axisymmetric case, but the numerical solver is new.
The solver is based on a Fourier pseudo-spectral discretization of the
surface.  We have developed efficient high-order surface quadratures
for evaluating layer-potentials based on previous work of
\cite{Bruno_2001a, Ying_2006}, but improve on the partition of unity
function used to separate the integrals into smooth and singular
parts.  Our method allows for precomputing the local quadrature
resulting in extremely efficient evaluations when used with iterative
solvers.  We have incorporated other performance optimizations such as
shared memory parallelism and vectorization of the kernel functions.
We have developed optimized implementations of the Laplace-Beltrami
solvers from~\cite{Imbert_G_rard_2017, O_Neil_2018_LaplaceBeltrami},
and have also presented a numerically stable scheme for computing
magnetic fluxes for arbitrary values of the Beltrami parameter
$\lambda$ which are critical for our formulation.  Finally, we have
presented numerical results to show that our method is efficient and
high-order accurate.

Since our boundary integral scheme requires only that the boundary of
the domain be discretized, and not the entire volume, this
significantly reduces the number of unknowns and makes the scheme very
efficient. A direct comparison with the existing code SPEC shows that
our solver compares favorably, and could be further accelerated, for
example, by making use of the usual $n$-fold rotational symmetry of
stellarator surfaces about the vertical axis as SPEC does (3-fold
symmetry for QAS3, 5-fold symmetry for W7-X).  We furthermore observe
that a boundary integral approach is in fact particularly well-suited
for the computation of stepped pressure equilibria because the force
balance conditions between adjacent Taylor states need only be
evaluated at the ideal MHD interfaces on each iteration, i.e on the
boundaries of the computational domain.  Other advantages of our
solver include a well-conditioned second-kind integral formulation and
no difficulties associated with coordinate singularities (as is common
in many volume parameterizations of toroidal domains required in FEM
codes).  A current limitation of our scheme is that the uniform grid
discretization makes it unsuitable for geometries with very sharp
features.  Such features are frequently required in stellarators to
model what are known as \emph{divertors}, used for removing fusion
products and impurities in the plasma.  Computing Taylor states
efficiently in such geometries will require high-order adaptive
surface discretizations, alternative singular and near-singular
quadratures, more efficient Laplace-Beltrami solvers compatible with
such discretizations, and coupling with asymptotically fast solvers,
such as fast multipole methods. This is the basis of ongoing research.

\section*{Ackowledgements}
The authors would like to thank Stuart Hudson (Princeton Plasma Physics Laboratory) and Joaquim Loizu (Ecole Polytechnique F\'ed\'erale de Lausanne) for valuable conversations, for their explanations on how to run the SPEC code, and for sharing their post-processing routines. 


\appendix
\section{Notation}

In \pr{t:notation}, we list some frequently used symbols for easy
reference.

\begin{table}[ht!]
  \centering
  \caption{\label{t:notation} Index of frequently used symbols.}
  \resizebox{\textwidth}{!}{
  \begin{tabular}{l l}
    \toprule 
    Symbol                                                    & Description \\
    \midrule 
    \Domain                                                   & domain \\
    \Boundary                                                 & boundary \\
    \Nsurf                                                    & number of surfaces \\
    \TAngle, \PAngle                                          & toroidal and poloidal angles \\
    \SCoord(\TAngle, \PAngle)                                 & surface position \\
    \Normal(\TAngle, \PAngle)                                 & surface normal (outward from $\Domain$) \\
    \MetricTensor(\TAngle, \PAngle)                           & metric tensor \\
    \midrule 
    \vector{B}                                                & magnetic field \\
    \BeltramiParam                                            & Beltrami parameter \\
    $\CrossSection_i$, $\Flux_i$                              & domain cross-section and the \\
                                                              & corresponding flux condition \\
    \midrule 
    \Nt, \Np                                                  & number of discretization points \\
                                                              & in toroidal and poloidal directions \\
    $f(\TAngle,\PAngle)$, \vectord{f}, \fourier{\vectord{f}}  & function on a surface, its \\
                                                              & discretization and its Fourier \\
                                                              & coefficients \\
    \bottomrule 
  \end{tabular}
  \begin{tabular}{l l}
    \toprule 
    Symbol                                                    & Description \\
    \midrule 
    \FFT, \IFFT                                               &periodic Fourier transform and its inverse \\
    \fftdiff{}{\TAngle}, \fftdiff{}{\PAngle}                  & pseudo-spectral differentiation \\
                                                              & operators in $\TAngle$ and $\PAngle$ directions \\
    \midrule 
    \HelmKer{\BeltramiParam}(\vector{r})                      & Helmholtz kernel function \\
    \SL{\BeltramiParam}, \DL{\BeltramiParam}                  & Helmholtz single- and double-layer \\
                                                              & convolution operators \\
    \patchdim, \quadorder                                     & singular-quadrature parameters \\
    \midrule 
    \Nunknown                                                 & total number of discretization points \\
    \lbiter                                                   & average number of iterations for each \\
                                                              & solve of the Laplace Beltrami problem \\
    \gmresiter                                                & average number of iterations for each \\
                                                              & solve of the boundary integral operator \\
    \Tsolve                                                   & total solve time \\
    \Tsetup                                                   & setup time for singluar-quadratures \\
    \Tquadeval                                                & evaluation time for the quadratures \\
    \Tlb                                                      & Laplace Beltrami solve time \\
    %
    \bottomrule 
  \end{tabular}
  }
\end{table}

\bibliographystyle{abbrvnat}
\bibliography{ref}

\end{document}